\documentclass{amsart}

\usepackage[letterpaper,margin=1.1in]{geometry}

\usepackage{amsmath,amssymb,amsthm}
\usepackage[shortlabels]{enumitem}
\usepackage{hyperref}
\usepackage{mathtools,pifont}
\usepackage{tikz}
\usepackage{mathrsfs}
\usepackage{xcolor}

\newtheorem{theorem}{Theorem}[section]
\newtheorem{observation}[theorem]{Observation}
\newtheorem{lemma}[theorem]{Lemma}
\newtheorem{proposition}[theorem]{Proposition}
\newtheorem{corollary}[theorem]{Corollary}
\newtheorem{fact}[theorem]{Fact}

\newtheorem*{theorem*}{Theorem}

\numberwithin{equation}{section}

\theoremstyle{definition}
\newtheorem{definition}[theorem]{Definition}
\newtheorem{notation}[theorem]{Notation}

\theoremstyle{remark}
\newtheorem{remark}[theorem]{Remark}
\newtheorem{example}[theorem]{Example}

\newcommand\R{\mathbb{R}}
\newcommand\C{\mathbb{C}}
\newcommand\Q{\mathbb{Q}}

\newcommand\N{\mathbb{N}}

\newcommand{\cA}{\mathcal{A}}

\newcommand{\cC}{\mathcal{C}}
\newcommand{\cD}{\mathcal{D}}
\newcommand{\cF}{\mathcal{F}}

\newcommand{\cL}{\mathcal{L}}
\newcommand{\cM}{\mathcal{M}}
\newcommand{\cN}{\mathcal{N}}
\newcommand{\cO}{\mathcal{O}}
\newcommand{\cP}{\mathcal{P}}

\newcommand{\cU}{\mathcal{U}}

\newcommand{\bS}{\mathbf{S}}
\newcommand{\bD}{\mathbf{D}}
\newcommand{\bX}{\mathbf{X}}

\newcommand{\rT}{\mathrm{T}}

\DeclareMathOperator{\Span}{Span}

\DeclareMathOperator{\re}{Re}
\DeclareMathOperator{\im}{Im}

\DeclareMathOperator{\tr}{tr}

\DeclareMathOperator{\sa}{sa}

\DeclareMathOperator{\Aut}{Aut}

\DeclareMathOperator{\diam}{diam}
\DeclareMathOperator{\Spec}{Spec}

\DeclareMathOperator{\tp}{tp}
\DeclareMathOperator{\acl}{acl}
\DeclareMathOperator{\dcl}{dcl}
\DeclareMathOperator{\diag}{diag}

\DeclareMathOperator{\Th}{Th}

\DeclarePairedDelimiter{\norm}{\lVert}{\rVert}
\DeclarePairedDelimiter{\ip}{\langle}{\rangle}

\begin{document}
	
	\title[Optimal transport for types and convex analysis for definable predicates]{Optimal transport for types and convex analysis for definable predicates in tracial $\mathrm{W}^*$-algebras}
	
	
	
	\author{David Jekel}
	\address{\parbox{\linewidth}{Fields Institute for Research in Mathematical Sciences, \\ 222 College St., Toronto, ON M5T 3J1, Canada}}
	\email{djekel@fields.utoronto.ca}
	\urladdr{http://davidjekel.com}
	
	\begin{abstract}
		We investigate the connections between continuous model theory, free probability, and optimal transport/convex analysis in the context of tracial von Neumann algebras.  In particular, we give an analog of Monge-Kantorovich duality for optimal couplings where the role of probability distributions on $\C^n$ is played by model-theoretic types, the role of real-valued continuous functions is played by \emph{definable predicates}, and the role of continuous function $\C^n \to \C^n$ is played by definable functions.  In the process, we also advance the understanding of definable predicates and definable functions by showing that all definable predicates can be approximated by ``$C^1$ definable predicates'' whose gradients are definable functions.  As a consequence, we show that every element in the definable closure of $\mathrm{W}^*(x_1,\dots,x_n)$ can be expressed as a definable function of $(x_1,\dots,x_n)$.  We give several classes of examples showing that the definable closure can be much larger than $\mathrm{W}^*(x_1,\dots,x_n)$ in general.
	\end{abstract}
	
	\maketitle
	
	\section{Introduction}
	
	Free probability theory treats tracial (finite) von Neumann algebras as non-commutative probability spaces and investigates the way that ideas from classical probability can be applied in the non-commutative setting.  In particular, non-commutative distributions or laws often play the role of probability distributions for a tuple of random variables.  The non-commutative analog of Wasserstein distance was introduced in \cite{BV2001}, and optimal transport theory has been in the background of several significant developments in free probability and random matrix theory \cite{HU2006,GS2014,DGS2021,JLS2022}.  Non-commutative Monge-Kantorovich duality was addressed in \cite{GJNS2021}.
	
	In \cite[\S 6.1]{JekelModelEntropy} the present author suggested investigating non-commutative optimal transport theory in the framework of continuous model theory.  Continuous model theory \cite{BYU2010,BYBHU2008} made its first formal contact with tracial von Neumann algebras in \cite{GHS2013,FHS2013,FHS2014,FHS2014b} and has since had many applications \cite{Goldbring2020,Goldbring2021enforceable,Goldbring2021nonembeddable,AGKE2022}; for a survey, see \cite{GH2023}.  The non-commutative law of a tuple $(x_1,\dots,x_n)$ corresponds to the \emph{quantifier-free type} in in the model theory of tracial von Neumann algebras described in \cite{FHS2014}.  However, because diffuse classical probability spaces admit quantifier elimination, there is no distinction between the quantifier-free type and the complete type; see \cite[Fact 2.10]{BY2012}, \cite[\S 6]{BH2023}, \cite[\S 2.3]{JekelModelEntropy}.  Hence, in the non-commutative setting, many tracial von Neumann algebras do not admit quantifier elimination \cite{Farah2024}, and so the complete type of a tuple $(x_1,\dots,x_n)$ could also be viewed as an analog of the probability distribution.
	
	One way that the complete type makes a better analog than the quantifier-free type is the relationship between (complete) types and automorphisms.  In the setting of classical probability (or equivalently commutative von Neumann algebras), two tuples $(X_1,\dots,X_n)$ and $(Y_1,\dots,Y_n)$ of random variables on a diffuse probability space $(\Omega,P)$ have the same probability distribution if and only if they are approximately conjugate by automorphisms of $(\Omega,P)$; see \cite[\S 6]{BH2023}, \cite[\S 2.3]{JekelModelEntropy}.  Analogously, every tracial von Neumann algebra $\cM$ admits an elementary extension $\cN \succeq \cM$, such that two tuples $\mathbf{x}= (x_1,\dots,x_n)$ and $\mathbf{y} = (y_1,\dots,y_n)$ are conjugate by an automorphism of $\cN$ if and only if they have the same type (such an $\cN$ is called \emph{strongly $\aleph_1$-homogeneous}); see \cite[\S 7]{BYBHU2008}.  However, the quantifier-free type does not provide enough information about $(x_1,\dots,x_n)$ to determine it up to automorphism in this sense.
	
	The analog of the Wasserstein distance for types is known in continuous model theory as the $d$-metric \cite[p.\ 44]{BYBHU2008}.  The $L^1$-Wasserstein distance in classical probability was studied from a model-theoretic viewpoint by Song \cite[\S 5.3-5.4]{SongThesis}, while here we handle the $L^2$-Wasserstein distance in the non-commutative setting as in \cite[\S 6.1]{JekelModelEntropy}.  In this discussion we fix a complete theory $\mathrm{T}$, i.e. we restrict our attention to one elementary equivalence class of tracial von Neumann algebras.  For tuples $\mathbf{x} = (x_1,\dots,x_n)$ and $\mathrm{y} = (y_1,\dots,y_n)$, we write $d^{\cM}(\mathbf{x},\mathbf{y}) = \sqrt{\sum_{j=1}^n d^{\cM}(x_j,y_j)^2} = \norm{\mathbf{x} - \mathbf{y}}_{L^2(\cM)^n}$.  For two types $\mu$ and $\nu$ relative to $\mathrm{T}$, the distance $d_W(\mu,\nu)$ is defined as the infimum of $d^{\cM}(\mathbf{x},\mathbf{y})$ over all models $\cM$ of $\mathrm{T}$ and over all $\mathbf{x}, \mathbf{y} \in \cM^n$ such that $\tp^{\cM}(\mathbf{x}) = \mu$ and $\tp^{\cM}(\mathbf{y})$.  We call $(\mathbf{x},\mathbf{y})$ an \emph{optimal coupling} of $(\mu,\nu)$ if $\tp^{\cM}(\mathbf{x}) = \mu$ and $\tp^{\cM}(\mathbf{y})$ and $d^{\cM}(\mathbf{x},\mathbf{y})$ realizes the infimum.  If $\cM$ is is a $\aleph_1$-saturated model of $\mathrm{T}$ (see \S \ref{subsec: elementary extension}), then an optimal coupling of $\mu$ and $\nu$ must exist in $\cM$.  Moreover, if $\cM$ is $\aleph_1$-saturated and strongly $\aleph_1$-homogeneous (see \S \ref{subsec: elementary extension}), then the types $\mu$ and $\nu$ correspond to $\Aut(\cM)$-orbits in $\cM^n$ and $d_W(\mu,\nu)$ is the distance between these two orbits.
	
	Our first main result is a dual characterization of the Wasserstein distance using definable predicates.  In the setting of tracial von Neumann algebras, definable predicates are scalar-valued functions $\phi^{\cM}(x_1,\dots,x_n)$ that are uniform limits (on each ball) of formulas built out of traces of non-commutative polynomials using continuous functions and $\sup$ and $\inf$ operations.  The definable predicates and types are related by a Gelfand duality loosely analogous to $C_0(\C^n,\R)$ being the predual of $\mathcal{P}(\C^n)$.\footnote{Technically, in the language of classical probability spaces, the space of definable predicates would be $C(\mathcal{P}(\C^n),\R)$ rather than $C_0(\C^n,\R)$ or even $C(\C^n,\R)$.  In the general setting, we have to make do without a direct analog of ``points.''}  Indeed, definable predicates can be identified with continuous functions on the space of types, while conversely, types can be identified with continuous real-valued characters on the space of definable predicates.
	
	In classical probability theory, the Wasserstein distance has a dual characterization in terms of continuous functions as follows.  Let $\mu$ and $\nu$ be probability distributions on $\C^n$ (assume they are compactly supported for simplicity).  Let
	\[
	C(\mu,\nu) = \frac{1}{2} \left(\int |x|^2\,d\mu(x) + \int |y|^2\,d\nu(y) - d_W(\mu,\nu)^2 \right).
	\]
	Then Monge-Kantorovich duality asserts that $C(\mu,\nu)$ is the infimum of $\int \phi(x)\,d\mu(x) + \int \psi(y)\,d\nu(y)$ over all pairs $(\phi,\psi)$ of convex functions $\C^n \to \R$ satisfying $\phi(x) + \psi(y) \geq \re \ip{x,y}_{\C^n}$ for all $x, y \in \C^n$; see for instance \cite[Theorem 5.10, Particular Case 5.17]{Villani2008}.  Proposition 1.3 of \cite{GJNS2021} gave an analog of this result for quantifier-free types in tracial von Neumann algebras using $E$-convex functions, a certain analog of convex functions for tracial $\mathrm{W}^*$-algebras.  Here we give an analogous result for complete types using convex definable predicates.  Here we use the following notation:
	\begin{itemize}
		\item For $\mathbf{r} = (r_1,\dots,r_n) \in (0,\infty)^n$, let $D_{\mathbf{r}}^{\cM} = D_{r_1}^{\cM} \times \dots \times D_{r_n}^{\cM}$ be the product of operator norm balls in $\cM$.
		\item For each $\mathbf{r}$, let $\mathbb{S}_{\mathbf{r}}(\mathrm{T})$ be the space of types of elements in $D_{\mathrm{r}}^{\cM}$ for tracial von Neumann algebras $\cM$ satisfying $\mathrm{T}$, and similarly let $\mathbb{S}_n(\mathrm{T})$ be the space of types of elements in $\cM^n$.
		\item For a definable predicate $\phi$ and a type $\mu$, we write $\mu[\phi]$ for the common value of $\phi^{\cM}(\mathbf{x})$ when $\mathbf{x}$ has type $\mu$.
		\item For types $\mu$ and $\nu$, let $C(\mu,\nu)$ be the supremum of $\re \ip{\mathbf{x},\mathbf{y}}_{L^2(\cM)}$ for $\cM$ satisfying $\mathrm{T}$, and $\mathbf{x}$ and $\mathbf{y}$ having types $\mu$ and $\nu$ respectively.  Note $d_W(\mu,\nu)^2 = \mu[\sum_{j=1}^n \tr(x_j^*x_j)] + \nu[\sum_{j=1}^n \tr(x_j^*x_j)] - 2 C(\mu,\nu)$.
		\item A definable predicate $\phi$ is said to be \emph{convex} if $\phi^{\cM}$ is a convex function for each $\cM$, that is, $\phi^{\cM}((1-t)\mathbf{x} + t \mathbf{y}) \geq (1-t) \phi^{\cM}(\mathbf{x}) + t \phi^{\cM}(\mathbf{y})$ for $t \in [0,1]$ and $\mathbf{x}, \mathbf{y} \in \cM^n$.
	\end{itemize}
	
	\begin{theorem} \label{thm: MK duality}
		Fix a complete theory $\mathrm{T}$ of a tracial von Neumann algebra.  Let $\mu$ and $\nu \in \mathbb{S}_n(\mathrm{T})$ be types.  Then there exist convex $\mathrm{T}_{\tr}$-definable predicates $\phi$ and $\psi$ such that
		\begin{equation} \label{eq: admissibility}
			\phi^{\cM}(\mathbf{x}) + \psi^{\cM}(\mathbf{y}) \geq \re \ip{\mathbf{x},\mathbf{y}}_{L^2(\cM)} \text{ for all } \mathbf{x}, \mathbf{y} \in \cM^n \text{ for all } \cM \models \mathrm{T}_{\tr},
		\end{equation}
		and such that equality is achieved when $(\mathbf{x},\mathbf{y})$ is an optimal coupling of $(\mu,\nu)$.  Hence, $C(\mu,\nu)$ is the infimum of $\mu[\phi] + \nu[\psi]$ over all pairs $(\phi,\psi)$ of convex definable predicates satisfying \eqref{eq: admissibility}.
	\end{theorem}
	
	The case of complete types here works out somewhat better than the case of quantifier-free types in \cite{GJNS2021}.  The $E$-convex functions described there were not necessarily definable predicates, and certainly not always quantifier-free definable predicates.  Moreover, we did not necessarily expect to be able to achieve the duality with quantifier-free definable predicates because the weak-$*$ topology on the space of quantifier-free types is weaker than the metric topology produced by the Wasserstein distance \cite[Proposition 1.7]{GJNS2021}.  When using complete types, the analog of Monge-Kantorovich duality works using convex definable predicates, which are continuous with respect to the weak-$*$ (logic) topology, despite the fact that even for complete types the logic topology is weaker than metric topology from the Wasserstein distance \cite[Proposition 12.10]{BYBHU2008}.
	
	For $E$-convex functions, one had to assume monotonicity under condition expectations, i.e.\ $\phi^{\cN}(E_{\cN}(\mathbf{x})) \leq \phi^{\cM}(\mathbf{x})$ for $\cN \subseteq \cM$ and $\mathbf{x} \in L^2(\cM)^n$, as an additional condition beside convexity of $\phi^{\cM}$ for each $\cM$ \cite[Definitions 1.2 and 3.6]{GJNS2021}.    However, for convex definable predicates, monotonicity under conditional expectations onto elementary substructures is automatic.
	
	\begin{proposition} \label{prop: monotonicity under expectation}
		Let $\phi(x_1,\dots,x_n)$ be a convex definable predicate relative to a consistent theory $\mathrm{T}$ containing $\mathrm{T}_{\tr}$.  Fix $\cM \models \mathrm{T}$ and let $\cN$ be an elementary extension of $\cM$.  Let $E_{\cM}: \cN \to \cM$ be the unique trace-preserving conditional expectation (see Fact \ref{fact: conditional expectation}).  Then for $\mathbf{z} \in \cN^n$, we have $\phi^{\cM}(E_{\cM}[\mathbf{z}]) \leq \phi^{\cN}(\mathbf{z})$.
	\end{proposition}
	
	In \cite[Theorem 1.5, \S 4.3]{GJNS2021}, Monge-Kantorovich duality was used to show that if $\mathbf{x}$ and $\mathbf{y}$ are an optimal coupling of \emph{quantifier-free types}, then $\mathrm{W}^*((1-t)\mathbf{x} + t \mathbf{y}) = \mathrm{W}^*(\mathbf{x},\mathbf{y})$ for $t \in (0,1)$.  In the analogous result for complete types, we replace the von Neumann algebra $\mathrm{W}^*(\mathbf{x})$ by the \emph{definable closure} $\dcl^{\cM}(\mathbf{x})$ (see \S \ref{sec: def and alg closures} for background).
	
	\begin{theorem} \label{thm: displacement interpolation}
		Fix a complete theory $\mathrm{T}$ of a tracial von Neumann algebra and let $\mu$ and $\nu \in \mathbb{S}_n(\mathrm{T})$.  If $(\mathbf{x},\mathbf{y})$ is an optimal coupling of $\mu$ and $\nu$ in $\cM$, then for each $t \in (0,1)$, we have $\dcl^{\cM}((1-t)\mathbf{x} + t \mathbf{y}) = \dcl^{\cM}(\mathbf{x},\mathbf{y})$.
	\end{theorem}
	
	The definable closure has the following characterization in terms of automorphisms (see Proposition \ref{prop:DCL} for details).  Suppose that $\cM$ is a strongly $\aleph_1$-homogeneous model of $\mathrm{T}$.  Then $y$ is in the definable closure of a finite or countable tuple $\mathbf{x}$ if and only if $y$ is fixed by every automorphism of $\cM$ that fixes $\mathbf{x}$.  Very little is known about the definable closure for von Neumann algebras in general, and thus in \S \ref{subsec: def and alg closure examples}, we give several examples.  In particular, we compute the definable closure for all $*$-subalgebras of finite-dimensional tracial $\mathrm{W}^*$-algebras (Example \ref{ex: finite dimensional}).  We also show that if $\cA$ is a $\mathrm{W}^*$-subalgebra of $\cM$ with trivial relative commutant and weak spectral gap, then $\dcl^{\cM}(\cA)$ contains the commutator subgroup of the normalizer of $\cA$ in $\cM$ (Proposition \ref{prop: spectral gap normalizer}).  In particular, the definable closure is in some cases much larger than the original subalgebra.
	
	Our next main result relates the definable closure with definable functions (see Definition \ref{def: definable function}).  This result is analogous to \cite[Proposition 13.6.6]{JekelThesis}, \cite[Proposition 2.4]{HJNS2021}, \cite[Propsosition 3.32]{JekelCoveringEntropy} which showed that if $A$ is a subset of a tracial von Neumann algebra $\cM$, then any $x \in \mathrm{W}^*(A)$ can be expressed as a \emph{quantifier-free} definable function over $\mathrm{T}_{\tr}$ of some $(a_j)_{j \in \N}$ in $A$.
	
	\begin{theorem} \label{thm: definable realization}
		Let $A$ be a subset of a tracial von Neumann algebra $\cM$.  Then $z$ is in the definable closure of $A$ in $\cM$ if and only if there exists $(a_j)_{j \in \N}$ in $A$ and a definable function $f$ over $\mathrm{T}_{\tr}$ in countably many variables such that $z = f((a_j)_{j \in \N})$.
	\end{theorem}
	
	The challenge of this result is the assertion that there is a continuous selection of an $f((y_j)_{j \in \N})$ for all $(y_j)_{j \in \N}$, which when evaluated at $(a_j)_{j \in \N}$ yields the given $z \in \dcl^{\cM}(A)$.  The proof is another application of the theory of convex definable predicates.  We first observe that if $z$ is in the definable closure, then $\re \ip{z,x}_{L^2(\cM)} = \phi(x,(a_j)_{j \in \N})$ for some definable predicate $\phi$.  We show that any definable predicate $\phi(x,(y_j)_{j\in \N})$ can be approximated by one which is differentiable with respect to $x$, whose gradient is a definable function; here the gradient is understood by viewing $L^2(\cM)$ as the real Hilbert space with $\re \ip{\cdot,\cdot}$ as the inner product.  The gradient is relevant here since $\nabla_x \phi(x,(a_j)_{j \in \N}) = \nabla_x \re \ip{x,z} =  z$; see \S \ref{subsec: definable function} for the proof.
	
	Hence, the key ingredient is a regularization technique that allows us to approximate an arbitrary definable predicate by one which is ``continuously differentiable.''  In \S \ref{subsec: regularization}, we employ the technique of semiconvex semiconcave regularization by sup-convolutions and inf-convolutions as in \cite{LL1986}, which has played a major role in the theory of Hamilton-Jacobi equations on infinite-dimensional Hilbert spaces \cite{BdP1981,BdP1985a,BdP1985b,CL1985,CL1986a,CL1986b}, to achieve the following result.
	
	\begin{theorem} \label{thm: approximation}
		Fix a theory $\mathrm{T} \models \mathrm{T}_{\tr}$.  Let $\phi(x_1,\dots,x_n,(y_j)_{j \in \N})$ be a definable predicate relative to $\mathrm{T}$.  Then for every $\mathbf{r} = (r_1,\dots,r_n)$, $\mathrm{s} = (s_j)_{j \in \N}$, and $\epsilon > 0$, there exists a definable predicate $\psi$ such that $\psi$ is differentiable in $\mathbf{x}$, $\nabla_{\mathbf{x}} \psi$ is a definable function, and $|\phi^{\cM} - \psi^{\cM}| < \epsilon$ on $D_{\mathbf{r},\mathbf{s}}^{\cM}$ for all $\cM \models \mathrm{T}$.
	\end{theorem}
	
	The paper is organized as follows.  In \S \ref{sec: preliminaries}, we recall background on continuous model theory and tracial von Neumann algebras.  In \S \ref{sec: def and alg closures}, we review the definition and properties of the definable and algebraic closures and give examples in von Neumann algebras.  In \S \ref{sec: duality}, we study Monge-Kantorovich duality for types and convex definable predicates and prove Theorem \ref{thm: MK duality}, Proposition \ref{prop: monotonicity under expectation}, and Theorem \ref{thm: displacement interpolation}.  In \S \ref{sec: semiconvex semiconcave}, we first discuss the properties of semiconvex semiconcave definable predicates, then we use sup-convolutions and inf-convolutions to regularize an arbitrary definable predicate to one which is semiconvex and semiconcave (Theorem \ref{thm: approximation}), and finally deduce Theorem \ref{thm: definable realization}.
	
	\subsection*{Acknowledgements}
	
	This work was partially supported by a postdoctoral fellowship from the National Science Foundation (DMS-2002826), Ilijas Farah's NSERC grant, and the Fields Institute.  I thank Dimitri Shlyahktenko, Wilfrid Gangbo, and Kyeongsik Nam for discussions on optimal couplings as well as Isaac Goldbring, Ilijas Farah, and Jennifer Pi for discussions on model theory of von Neumann algebras.  I thank the referee for numerous corrections and improvement of the exposition.

	\section{Preliminaries} \label{sec: preliminaries}
	
	This section gives background on operator algebras and model theory for the convenience of the reader, as well as citations to other works that provide a fuller exposition.  Readers already familiar with these topics may consult this section only as needed to clarify conventions and notation.
	
	\subsection{Background on operator algebras} \label{subsec:operatoralgebrasbackground}
	
	In this paper, a \emph{tracial von Neumann algebra}, or equivalently \emph{tracial $\mathrm{W}^*$-algebra} refers to a finite von Neumann algebras with a specified tracial state.  We recommend \cite{Ioana2023} for an introduction to the topic, as well as the following standard reference books \cite{KadisonRingroseI,Dixmier1969,Sakai1971,TakesakiI,Blackadar2006,Zhu1993}.
	
	The abstract definitions / characterizations of $\mathrm{C}^*$-algebras and $\mathrm{W}^*$-algebras are as follows.
	\begin{enumerate}[(1)]
		\item A (unital) algebra over $\C$ is a unital ring $A$ with a unital inclusion map $\C \to A$.
		\item A (unital) $*$-algebra is an algebra $A$ equipped with a conjugate linear involution $*$ such that $(ab)^* = b^* a^*$.
		\item A unital $\mathrm{C}^*$-algebra is a $*$-algebra $A$ equipped with a complete norm $\norm{\cdot}$ such that $\norm{ab} \leq \norm{a} \norm{b}$ and $\norm{a^*a} = \norm{a}^2$ for $a, b \in A$.
		\item A \emph{$\mathrm{W}^*$-algebra} is a $\mathrm{C}^*$-algebra $\cA$ that $\cA$ is isomorphic as a Banach space to the dual of some other Banach space (called its predual).
	\end{enumerate}
	The work of Sakai (see e.g.\ \cite{Sakai1971}) showed that the abstract definition given here for a $\mathrm{W}^*$-algebra is equivalent to several other definitions and notions.  Sakai also showed that the predual is unique (and hence so is the weak-$*$ topology on $\cA$).
	
	\begin{notation}
		A \emph{tracial $\mathrm{W}^*$-algebra} is a pair $(M,\tau)$ where $M$ is a $\mathrm{W}^*$-algebra and $\tau: M \to \C$ is a faithful normal tracial state, that is, a linear map satisfying
		\begin{itemize}
			\item \emph{positivity}: $\tau(x^*x) \geq 0$ for all $x \in M$
			\item \emph{unitality}: $\tau(1) = 1$
			\item \emph{traciality}: $\tau(xy) = \tau(yx)$ for $x, y \in A$
			\item \emph{faithfulness}:  $\tau(x^*x) = 0$ implies $x = 0$ for $x \in A$.
			\item \emph{weak-$*$ continuity}:  $\tau: M \to \C$ is weak-$*$ continuous.
		\end{itemize}
		We will often denote $\tau$ by $\tr^{\cM}$ for consistency with the model-theoretic notation introduced below.
	\end{notation}
	
	\begin{fact}
		If $\cM = (M,\tau)$ is a tracial $\mathrm{W}^*$-algebra, then $\re \ip{x,y}_{L^2(\cM)} := \tau(x^*y)$ defines an inner product on $\cM$.  The Hilbert space completion is denoted $L^2(\cM)$.  The map $\cM \to L^2(\cM)$ is injective because $\tau$ is faithful.  Moreover, for $r > 0$, the ball $D_r^{\cM} = \{x \in \cM: \norm{x} \leq r\}$ is a closed subset of $L^2(\cM)$.
	\end{fact}
	
	\begin{fact} \label{fact: J}
		If $\cM = (M,\tau)$ is a tracial $\mathrm{W}^*$-algebra, then $\norm{x^*}_{L^2(\cM)} = \norm{x}_{L^2(\cM)}$.  In particular, the $*$-operation extends to an anti-linear isometry $J: L^2(\cM) \to L^2(\cM)$.  Thus, it makes sense to refer to ``self-adjoint'' elements of $L^2(\cM)$.
	\end{fact}
	
	\begin{definition}
		A \emph{$*$-homomorphism} is a map between $*$-algebras that respects the addition, multiplication, and $*$-operations.  For tracial $\mathrm{W}^*$-algebras, a $*$-homomorphism $\rho: \cM \to \cN$ is said to be \emph{trace-preserving} if $\tr^{\cM}(\rho(x)) = \tr^{\cN}(x)$.
	\end{definition}
	
	\begin{fact} \label{fact: conditional expectation}
		A $*$-homomorphism between $\mathrm{C}^*$-algebras is automatically contractive with respect to the norm.  Moreover, if $\rho: \cM \to \cN$ is a trace-preserving $*$-homomorphism of tracial $\mathrm{W}^*$-algebras, then $\rho$ is also contractive with respect to the $L^2$-norm and hence extends uniquely to a contractive map $L^2(\cM) \to L^2(\cN)$.  Moreover, a trace-preserving $*$-homomorphism is automatically continuous with respect to the weak-$*$ topology.
	\end{fact}
	
	\begin{notation}
		If $\rho: \cM \to \cN$ is a trace-preserving $*$-homomorphism of tracial $\mathrm{W}^*$-algebras, we will also denote its extension $L^2(\cM) \to L^2(\cN)$ by $\rho$.  Furthermore, for a tuple $\mathbf{x} = (x_i)_{i \in I}$, we use the notation $\rho(\mathbf{x}) = (\rho(x_i))_{i \in I}$.
	\end{notation}
	
	\begin{fact}
		Let $\cM \subseteq \cN$ be a trace-preserving inclusion of tracial $\mathrm{W}^*$-algebras.  Let $E_{\cM}: L^2(\cN) \to L^2(\cM)$ be the orthogonal projection.  Then $E_{\cM}$ restricts to a map $\cN \to \cM$ that is contractive with respect to the operator norm.  Moreover, $E_{\cM}$ is the unique conditional expectation (positive $\cN$-$\cN$ bimodule map that restricts to the identity on $\cN$) that preserves the trace.
	\end{fact}
	
	\begin{notation}
		For a $\mathrm{W}^*$-algebra, or more generally, a unital $*$-algebra,
		\begin{enumerate}[(1)]
			\item $\cM_{\sa} = \{x \in \cM: x^* = x\}$ denotes the set of \emph{self-adjoints}.
			\item $L^2(\cM)_{\sa} = \{x \in L^2(\cM): J(x) = x\}$ similarly denotes the set of self-adjoints in $L^2(\cM)$.
			\item $U(\cM) = \{u \in \cM: u^*u = uu^* = 1\}$ denotes the set of \emph{unitaries}.
			\item $P(\cM) = \{ p \in \cM: p^* = p = p^2 \}$ denotes the set of \emph{projections}.
			\item $Z(\cM) = \{x \in \cM: \forall y \in \cM, xy = yx\}$ denotes the \emph{center}.
			\item For $A \subseteq \cM$, we write $A' \cap \cM = \{x \in \cM" \forall y \in A, xy = yx\}$ for the \emph{relative commutant} of $A$ in $\cM$.
		\end{enumerate}
	\end{notation}
	
	\begin{fact}
		A tracial $\mathrm{W}^*$-algebra has trivial center $Z(\cM) = \C 1$ if and only if for any two projections $p$ and $q$ with $\tr^{\cM}(p) = \tr^{\cM}(q)$, there exists a unitary $u \in U(\cM)$ such that $upu^* = q$.
	\end{fact}
	
	\begin{notation}
		A \emph{$\mathrm{II}_1$-factor} is an infinite-dimensional tracial $\mathrm{W}^*$-algebra with trivial center ($Z(\cM) = \C 1$).  The terminology ``factor'' comes from its role in Murray and von Neumann's direct integral decomposition of general von Neumann algebras.
	\end{notation}
	
	\subsection{Languages and structures}
	
	Here let us sketch the setup of continuous model theory, or model theory for metric structures from Ben Yaacov, Berenstein, Henson, and Usvyatsov \cite{BYBHU2008,BYU2010}; see also \cite{Hart2023}.  We will follow the treatment in Farah, Hart, and Sherman \cite{FHS2014} which introduces ``domains of quantification'' to cut down on the number of ``sorts'' needed; see also \cite[\S 2-3]{JekelCoveringEntropy}.
	
	A \emph{metric signature} $\mathcal{L}$ (sometimes called a \emph{language}) consists of:
	\begin{itemize}
		\item A set $\mathcal{S}$ whose elements are called \emph{sorts}.
		\item For each $S \in \mathcal{S}$, a privileged relation symbol $d_S$ (which will represent a metric) and a set $\mathcal{D}_S$ whose elements are called \emph{domains of quantification for $S$}.
		\item For each $S \in \mathcal{S}$ and $D, D' \in \mathcal{D}_S$ an assigned constant $C_{D,D'}$.
		\item \emph{Variable symbols} for each sort $S$.  We will usually denote variables by $x_i$, $y_i$, $z_i$ for $i$ in an appropriate index set $I$.
		\item A set of \emph{function symbols}.
		\item For each function symbol $f$, an assigned tuple $(S_1,\dots,S_n)$ of sorts called the \emph{domain}, another sort $S$ called the \emph{codomain}.  We call $n$ the \emph{arity of $f$}.
		\item For each function symbol $f$ with domain $(S_1,\dots,S_n)$ and codomain $S$, and for every $\mathbf{D} = (D_1,\dots,D_n) \in \mathcal{D}_{S_1} \times \dots \times \mathcal{D}_{S_n}$, there is an assigned $D_{f,\mathbf{D}} \in \mathcal{D}_S$ (representing a range bound), and assigned moduli of continuity $\omega_{f,\mathbf{D},1}$, \dots, $\omega_{f,\mathbf{D},n}$. (Here ``modulus of coninuity'' means a continuous increasing, zero-preserving function $[0,\infty) \to [0,\infty)$).
		\item A set of \emph{relation symbols}.
		\item For each relation symbol $R$, an assigned domain $(S_1,\dots,S_n)$ as in the case of function symbols.
		\item For each relation symbol $R$ and for every $\mathbf{D} = (D_1,\dots,D_n) \in \mathcal{D}_{S_1} \times \dots \times \mathcal{D}_{S_n}$, an assigned bound $N_{R,\mathbf{D}} \in [0,\infty)$ and assigned moduli of continuity $\omega_{R,\mathbf{D},1}$, \dots, $\omega_{R,\mathcal{D},n}$.
	\end{itemize}
	
	Given a language $\mathcal{L}$, an \emph{$\mathcal{L}$-structure} $\cM$ assigns an object to each symbol in $\mathcal{L}$, called the \emph{interpretation} of that symbol, in the following manner:
	\begin{itemize}
		\item Each sort $S \in \mathcal{S}$ is assigned a metric space $S^{\cM}$, and the symbol $d_S$ is interpreted as the metric $d_S^{\cM}$ on $S^{\cM}$.
		\item Each domain of quantification $D \in \mathcal{D}_S$ is assigned a subset $D^{\cM} \subseteq S^{\cM}$, such that $D^{\cM}$ is complete for each $D$, $S^{\cM} = \bigcup_{D \in \mathcal{D}_S} D^{\cM}$, and $\sup_{X \in D, Y \in D'} d_S^{\cM}(X,Y) \leq C_{D,D'}$.
		\item Each function symbol $f$ with domain $(S_1,\dots,S_n)$ and codomain $S$ is interpreted as a function $f^{\mathcal{M}}: S_1^{\cM} \times \dots \times S_n^{\cM} \to S^{\cM}$.  Moreover, for each $\mathbf{D} = (D_1,\dots,D_n) \in \mathcal{D}_{S_1} \times \dots \times \mathcal{D}_{S_n}$, the function $f^{\cM}$ maps $D_1^{\cM} \times \dots \times D_n^{\cM}$ into $D_{f,\mathbf{D}}^{\cM}$.  Finally, $f^{\cM}$ restricted to $D_1^{\cM} \times \dots \times D_n^{\cM}$ is uniformly continuous in the $i$th variable with modulus of continuity of $\omega_{f,\mathbf{D},i}$.
		\item Each relation symbol $R$ with domain $(S_1,\dots,S_n)$ is interpreted as a function $R^{\cM}: S_1^{\cM} \times \dots \times S_n^{\cM} \to \R$.  Moreover, for each $\mathbf{D} = (D_1,\dots,D_n) \in \mathcal{D}_{S_1} \times \dots \times \mathcal{D}_{S_n}$, $f^{\cM}$ is bounded by $N_{R,\mathbf{D}}$ on $M(D_1) \times \dots \times M(D_n)$ and uniformly continuous in the $i$th argument with modulus of continuity of $\omega_{R,\mathbf{D},i}$.
	\end{itemize}
	If $\cM_1$ and $\cM_2$ are $\cL$-structures, then a \emph{homomorphism} $\cM_1 \to \cM_2$ is a collection of maps $\Theta_S: S^{\cM_1} \to S^{\cM_2}$ such that for every function symbol $f: S_1 \times \dots \times S_n \to S$, we have $f^{\cM_2}(\Theta_{S_1},\dots,\Theta_{S_n}) = f^{\cM_1}$ and for every relation symbol $R: S_1 \times \dots \times S_n \to \R$ we have $R^{\cM_2}(\Theta_{S_1},\dots,\Theta_{S_n}) \leq R^{\cM_1}$.
	
	Next, we describe \emph{language $\mathcal{L}_{\tr}$ of tracial $\mathrm{W}^*$-algebras} together with how each symbol will be interpreted for a given tracial von Neumann algebra.
	\begin{itemize}
		\item A single sort, to be interpreted as the $\mathrm{W}^*$-algebra $M$.  If $\cM = (M,\tau)$ is a tracial $\mathrm{W}^*$-algebra, we denote the interpretation of this sort simply by $\cM$.
		\item Domains of quantification $\{D_r\}_{r \in (0,\infty)}$, to be interpreted as the operator norm balls of radius $r$ in $M$.
		\item The metric symbol $d$, to be interpreted as the metric induced by $\norm{\cdot}_{2,\tau}$.
		\item A binary function symbol $+$, to be interpreted as addition.
		\item A binary function symbol $\cdot$, to be interpreted as multiplication.
		\item A unary function symbol $*$, to be interpreted as the adjoint operation.
		\item For each $\lambda \in \C$, a unary function symbol, to be interpreted as multiplication by $\lambda$.
		\item Function symbols of arity $0$ (in other words constants) $0$ and $1$, to be interpreted as additive and multiplicative identity elements.
		\item Two unary relation symbols $\re \tr$ and $\im \tr$, to be interpreted the real and imaginary parts of the trace $\tau$.
		\item For technical reasons (see \cite[Proposition 3.2]{FHS2014} and \cite[p.\ 93]{Hart2023}), we also introduce for each $d$-variable non-commutative polynomial $p$ and $n \in \N$ a symbol $t_p: L^\infty(\cM)^d$ representing the evaluation of $p$, along with the appropriate range bounds $N_{t_p,\mathbf{r}}$ between $M_p$ and $M_p + 1/n$, where $M_p$ is the supremum of $\norm{p(X_1,\dots,X_d)}$ over all $(X_1,\dots,X_d)$ in a tracial $\mathrm{W}^*$-algebra $\cM$.
	\end{itemize}
	Each function and relation symbol is assigned range bounds and moduli of continuity that one would expect, e.g.\ multiplication is supposed to map $D_r \times D_{r'}$ into $D_{rr'}$ with $\omega_{(D_r,D_r'),1}^{\cdot}(t) = r't$ and $\omega_{(D_r,D_r'),2}^{\cdot} = rt$.
	
	Although not every $\mathcal{L}_{\tr}$-structure comes from a tracial $\mathrm{W}^*$-algebra, one can formulate axioms in the language such that any structure satisfying these axioms comes from a tracial $\mathrm{W}^*$-algebra \cite[\S 3.2]{FHS2014}.  For such structures, the notion of homomorphism given above reduces to trace-preserving $*$-homomorphisms of tracial $\mathrm{W}^*$-algebras.
	
	\subsection{Syntax: Terms, formulas, and sentences}
	
	\emph{Terms} in a language $\mathcal{L}$ are expressions obtained by iteratively composing the function symbols and variables.  For example, if $x_1$, $x_2$, \dots are variables in a sort $S$ and $f: S \times S \to S$ 
	and $g: S \times S \to S$ are function symbols, then $f(g(x_1,x_2),x_1)$ is a term.  Each term has assigned range bounds and moduli of continuity in each variable which are the natural ones computed from those of the individual function symbols making up the composition.  Any term $f$ with variables $x_1 \in S_1$, \dots, $x_k \in S_k$ and output in $S$ can be interpreted in an $\mathcal{L}$-structure as a function $S_1^{\cM} \times \dots \times S_k^{\cM} \to S^{\cM}$. For example, in the language $\mathcal{L}_{\tr}$, the terms are expressions obtained from iterating scalar multiplication, addition, multiplication, and the $*$-operation on variables and the unit symbol $1$.  If $(M,\tau)$ is a tracial $\mathrm{W}^*$-algebra, then the interpretation of a term in $\mathcal{M}$ is a $*$-polynomial function.
	
	\emph{Basic formulas} are obtained by evaluating relation symbols on terms.  In other words, if $T_1$, \dots, $T_k$ are terms valued in sorts $S_1$, \dots, $S_k$, and $R$ is a relation $S_1 \times \dots \times S_k \to \R$, then $R(T_1,\dots,T_n)$ is a basic formula.  The basic formulas have assigned range bounds and moduli of continuity similar to the function symbols.  In an $\mathcal{L}$-structure $\cM$, a basic formula $\phi$ is interpreted as a function $\phi^{\cM}: S_1^{\cM} \times \dots \times S_k^{\cM} \to \R$.  In $\mathcal{L}_{\tr}$, a basic formula can take the form $\re \tr(f)$ or $\im \tr(f)$ where $f$ is an expression obtained by iterating the algebraic operations.  Thus, when evaluated in a tracial $\mathrm{W}^*$-algebra, it corresponds to the real or imaginary part of the trace of a non-commutative $*$-polynomial.
	
	\emph{Formulas} are obtained from basic formulas by iterating several operations:
	\begin{itemize}
		\item Given a formulas $\phi_1$, \dots, $\phi_n$ and $F: \R^n \to \R$ continuous, $F(\phi_1,\dots,\phi_n)$ is a formula.
		\item If $\phi$ is a formula, $D$ is a domain of quantification for some sort $S$, and $x$ is one of our variables in $S$, then $\inf_{x \in D} \phi$ and $\sup_{x \in D} \phi$ are formulas.
	\end{itemize}
	Each occurrence of a variable in $\phi$ is either \emph{bound} to a quantifier $\sup_{x \in D}$ or $\inf_{x \in D}$, or else it is \emph{free}.  We will often write $\phi(x_1,\dots,x_n)$ for a formula to indicate that the free variables are $x_1$, \dots, $x_n$.
	
	All these formulas also have assigned range bounds and moduli of continuity. The moduli of continuity of $F(\phi_1,\dots,\phi_n)$ are obtained by composition from the moduli of continuity of $F$ and $\phi_j$ as in \cite[\S 2 appendix and Theorem 3.5]{BYBHU2008}.  Next, if $\phi: S_1 \times \dots \times S_n \to S$ and $D \in \mathcal{D}_{S_n}$
	\[
	\psi(x_1,\dots,x_{n-1}) = \sup_{x_n \in D} \phi(x_1,\dots,x_{n-1},x_n),
	\]
	then
	\[
	\omega_{\psi,(D_1,\dots,D_{n-1}),j} = \omega_{\phi,(D_1,\dots,D_{n-1},D),j}.
	\]
	Each formula has an interpretation in every $\mathcal{L}$-structure $\mathcal{M}$, defined by induction on the complexity of the formula.  If $\phi = F(\phi_1,\dots,\phi_n)$, then $\phi^{\cM} = F(\phi_1^{\cM}, \dots, \phi_n^{\cM})$.  Similarly, if $\psi(x_1,\dots,x_{n-1}) = \sup_{x_n \in D} \psi(x_1,\dots,x_n)$, then
	\[
	\psi^{\cM}(x_1,x_2,\dots,x_n) = \sup_{x_n \in D^{\cM}} \phi^{\cM}(x_1,\dots,x_{n-1}) \text{ for } x_1 \in S_1^{\cM}, \dots, x_n \in S_n^{\cM}.
	\]
	
	A \emph{sentence} is a formula with no free variables.  If $\phi$ is a sentence, then the interpretation $\phi^{\cM}$ in an $\mathcal{L}$-structure is simply a real number.
	
	\subsection{Theories, models, and axioms}
	
	A \emph{theory} $\rT$ in a language $\mathcal{L}$ is a set of sentences.  We will use the following terminology related to theories:
	\begin{enumerate}[(1)]
		\item An $\mathcal{L}$-structure $\mathcal{M}$ \emph{models} the theory $\rT$, or $\mathcal{M} \models \rT$ if $\phi^{\cM} = 0$ for all $\phi \in \rT$.
		\item If $\rT$ and $\rT'$ are two theories, we say that $\rT \models \rT'$ if $\cM \models \rT$ implies $\cM \models \rT'$ for all $\cL$-structures $\cM$.
		\item A theory $\rT$ is \emph{consistent} if there exists an $\cL$-structure $\cM$ with $\cM \models \rT$.
		\item A theory $\rT$ is \emph{complete} if for every sentence $\phi$ we have $\phi - c \in \rT$ for some $c \in \R$, or in other words, $\rT$ uniquely specifies the value of every sentence.
		\item If $\mathcal{M}$ is an $\mathcal{L}$-structure, then the \emph{theory of $\cM$}, denoted $\Th(\mathcal{M})$ is the set of sentences $\phi$ such that $\phi^{\cM} = 0$.
		\item More generally, if $\mathcal{C}$ is a class of $\mathcal{L}$-structures, then $\Th(\cC)$ is the set of all sentences $\phi$ such that $\phi^{\cM} = 0$ for all $\cM$ in $\cC$.
	\end{enumerate}
	
	The class $\cC$ is said to be \emph{axiomatizable} if every $\mathcal{L}$-structure that models $\Th(\cC)$ is actually in $\cC$.  Farah, Hart, and Sherman \cite[\S 3.2]{FHS2014} showed that the class of $\cL_{\tr}$-structures that represent actual tracial $\mathrm{W}^*$-algebras is axiomatizable.  The theory of tracial $\mathrm{W}^*$-algebras will be denoted $\rT_{\tr}$.  Similarly, the class of $\mathrm{II}_1$-factors (infinite-dimensional tracial von Neumann algebras with trivial center) is axiomatizable and we denote its theory by $\rT_{\mathrm{II}_1}$ \cite{FHS2014}.
	
	\subsection{Types, definable predicates, and definable functions} \label{subsec: definable predicate}
	
	Now we can explain the concept of types and definable predicates, which in our paper will play an analogous roles respectively to probability distributions and continuous functions.  Definable predicates are the completion of the vector space of formulas in an appropriate topology, in a similar way to continuous functions on $\R^d$ can be obtained by completing the space of polynomials with respect to uniform convergence on balls.
	
	\begin{definition} \label{def: types}
		Let $I$ be an index set, and let $\mathbf{S} = (S_i)_{i \in I}$ be an $I$-tuple of sorts in a language $\mathcal{L}$.  Let $\cF_{\mathbf{S}}$ be the space of $\cL$-formulas with free variables $(x_i)_{i \in I}$ with $x_i$ from the sort $S_i$.  If $\cM$ is an $\cL$-structure and $\mathbf{x} \in \prod_{i \in I} S_i^{\cM}$, then the \emph{type} of $\mathbf{x}$ is the map
		\[
		\tp^{\cM}(\mathbf{x}): \cF_{\mathbf{S}} \to \R: \quad \phi \mapsto \phi^{\cM}(\mathbf{x}).
		\]
	\end{definition}
	
	The type of $\mathbf{x}$ describes its behavior with respect to all formulas.  Another important variant of this concept describes the behavior of $\mathbf{x}$ with respect to formulas involving $\mathbf{x}$ as well as other elements from some $A \subseteq \cM$, or ``the behavior of $\mathbf{x}$ relative to $A$.''  To formalize this notion, we simply expand the language $\cL$ by adding constant symbols for the elements of $A$.
	
	\begin{notation}
		Let $\cM$ be an $\cL$-structure and let $A \subseteq \cM$ (or more precisely $A \subseteq \bigsqcup_{S \in \mathcal{S}} S^{\cM}$).  Then $\cL_A$ will denote the language obtained from $\cL$ by adding constant symbols (i.e.\ function symbols of arity zero) for every $a \in A$.  Moreover, $M_A$ will denote the $\cL_A$ structure where the constant symbol for each $a \in A$ is interpreted as $a$ itself.
	\end{notation}
	
	\begin{definition}
		Let $I$ be an index set, and let $\mathbf{S} = (S_i)_{i \in I}$ be an $I$-tuple of sorts in a language $\mathcal{L}$.  Let $\cM$ be an $\cL$-structure and $A \subseteq \cM$.  Then the \emph{type of $\mathbf{x}$ over $A$}, denoted $\tp^{\cM}(\mathbf{x} / A)$ is the $\cL_A$-type of $\mathbf{x}$ in $\cM_A$.
	\end{definition}
	
	For the rest of this section, we focus on types as in Definition \ref{def: types}, but the concepts apply equally well to types over $A$ by replacing $\cL$ with $\cL_A$.
	
	\begin{definition}
		Let $\mathbf{S} = (S_i)_{i \in I}$ be an $I$-tuple of sorts in $\mathcal{L}$, and let $\rT$ be an $\cL$-theory.  We denote by $\mathbb{S}_{\bS}(\rT)$ the set of types $\tp^{\cM}(\bX)$ of all $\bX \in \prod_{i \in I} S_i^{\cM}$ for all $\cM \models \rT$.  Similarly, if $\mathbf{D} \in \prod_{i \in I} \mathcal{D}_{S_i}$, then we denote by $\mathbb{S}_{\bD}(\rT)$ the set of types $\tp^{\cM}(\bX)$ of all $\bX \in \prod_{i \in I} D_i^{\cM}$ for all $\cM \models \rT$.
	\end{definition}
	
	The space of types $\mathbb{S}_{\bD}(\rT)$ is equipped with a weak-$*$ topology as follows.  If $\bS$ is an $I$-tuple of $\cL$-sorts, the set $\cF_{\bS}$ of formulas defines a real vector space.  For each $\cL$-structure $\cM$ and $\bX \in \prod_{i \in I} S_j^{\cM}$, the type $\tp^{\cM}(\bX)$ is a linear map $\cF_{\bS} \to \R$.  Thus, for each $\cL$-theory $\rT$ and $\bD \in \prod_{i \in I} \cD_{S_i}$, the space $\mathbb{S}_{\bD}(\rT)$ is a subset of the real dual $\cF_{\mathbf{S}}^\dagger$.  We equip $\mathbb{S}_{\bD}(\rT)$ with the weak-$\star$ topology (also known as the \emph{logic topology}), which makes it into a compact Hausdorff space \cite[Proposition 8.6]{BYBHU2008}.  Moreover, $\mathbb{S}_{\bD}(\rT)$ is metrizable provided that $I$ is countable and the language $\cL$ is separable (Definition \ref{def: language density character}); in particular, this holds for $\cL_{\tr}$ \cite[Observations 3.13 and 3.14]{JekelCoveringEntropy}.  See \cite[\S 7]{Hart2023} for further background.\footnote{Technically, to make $\cL_{\tr}$ separable by Definition \ref{def: language density character}, we need to have only countably many domains, and thus we should restrict the values of $r$ in the domains $D_r$ to a countable set such as $(0,\infty) \cap \Q$.  However, the number of domains is irrelevant to metrizability of $\mathrm{S}_{\mathbf{D}}(\mathrm{T})$ for a fixed tuple of domains anyway.}
	
	In the setting with domains of quantification, it is also convenient to have a topology on $\mathbb{S}_{\bS}(\rT)$.  We say that $\cO \subseteq \mathbb{S}_{\bS}(\rT)$ is \emph{open} if $\cO \cap \mathbb{S}_{\bD}(\rT)$ is open for every $\bD \in \prod_{j \in \N} \cD_{S_j}$; this defines a topology on $\mathbb{S}_{\bS}(\rT)$, which we will also call the \emph{logic topology}.  One can then show that the inclusion map $\mathbb{S}_{\bD}(\rT) \to \mathbb{S}_{\bS}(\rT)$ is a topological embedding \cite[Observation 3.6]{JekelCoveringEntropy}.
	
	Every formula defines a continuous function on $\mathbb{S}_{\bS}(\rT)$ for each theory $\rT$, but the converse is not true.  The objects that correspond to continuous functions on $\mathbb{S}_{\bS}(\rT)$ are a certain completion of the set of formulas, called \emph{definable predicates}; see e.g.\ \cite[\S 5.2]{Hart2023}.  Our approach to the definition will be semantic rather than syntactic, defining these objects immediately in terms of their interpretations.
	
	\begin{definition} \label{def:definablepredicate}
		Let $\mathcal{L}$ be a language and $\rT$ a consistent $\mathcal{L}$-theory.  A \emph{definable predicate relative to $\rT$} is a collection of functions $\phi^{\cM}: \prod_{i \in I} S_i^{\cM} \to \R$ (for each $\cM \models \rT$) such that for every collection of domains $\mathbf{D} = (D_j)_{j \in \N}$ and every $\epsilon > 0$, there exists a finite $F \subseteq I$ and an $\mathcal{L}$-formula $\psi(x_i: i \in F)$ such that whenever $\cM \models \rT$ and $\mathbf{X} \in \prod_{j \in \N} D_j^{\cM}$, we have
		\[
		|\phi^{\cM}(\mathbf{X}) - \psi^{\cM}(X_i: i \in F)| < \epsilon.
		\]
	\end{definition}
	
	In this work, we will use definable predicates relative to $\rT_{\tr}$, to $\rT_{\mathrm{II}_1}$, and to an arbitrary complete $\rT \models \rT_{\tr}$.  In \cite[Theorem 9.9]{BYBHU2008}, the authors show in the setting with a single sort and a single domain that definable predicates relative to $\rT$ correspond to continuous functions on the space of types $\rT$, and this result was extended to the setting with multiple sorts and domains of quantification \cite[Proposition 3.9]{JekelCoveringEntropy}.  Here we only need the following fact, for which we give a self-contained proof (the special case of $\mathrm{T}_{\tr}$ was done in \cite[Remark 3.33]{JekelCoveringEntropy}).
	
	\begin{lemma} \label{lem: continuous function definable predicate}
		Let $\cL$ be a metric signature and $\mathrm{T}$ a consistent $\cL$-theory.  Let $\mathbf{S}$ be a tuple of sorts indexed by the set $I$, and let $\mathbf{D}$ be a corresponding tuple of domains.  Let $f: \mathbb{S}_{\mathbf{D}}(\mathrm{T}) \to \R$ be continuous.  Then there exists a $\mathrm{T}$-definable predicate $\phi$ such that
		\[
		\phi^{\cM}(\mathbf{x}) = f(\tp^{\cM}(\mathbf{x})) \text{ for } \mathbf{x} \in \prod_{i \in I} \cD_i^{\cM} \text{ for } \cM \models \mathrm{T}.
		\]
	\end{lemma}
	
	\begin{proof}
		Note that by definition of the logic topology, each formula $\psi$ in variables $x_i \in S_i$ for $i \in I$ defines a continuous function $g$ on $\mathbb{S}_{\mathbf{D}}(\mathrm{T})$ by the relation $g(\tp^{\cM}(\mathbf{x})) = \psi^{\cM}(\mathbf{x})$ for $\cM \models \mathrm{T}$ and $\mathbf{x} \in \prod_{i \in I} D_i^{\cM}$.  Moreover, these functions separate points because by definition a type is uniquely determined by its values on formulas.  Formulas also comprise an algebra since linear combinations and multiplication are continuous functions $\R^2 \to \R$.  Therefore, by the Stone-Weierstrass theorem, there is a sequence $\psi_k$ of formulas such that the corresponding functions $f_k \in C(\mathbb{S}_{\mathbf{D}}(\mathrm{T}))$ converge uniformly to our given $f$.  Without loss of generality, assume that $\norm{f_{k+1} - f_k}_u \leq 1/2^k$.
		
		By construction, the formulas $\psi_k$ converge uniformly on $\prod_{i \in I} D_i^{\cM}$ for $\cM \models \mathrm{T}$, but our goal now is to modify them suitably to force convergence everywhere, which we accomplish through a standard ``forced limit'' construction.  Let $g: [0,\infty) \to [0,1]$ be continuous and decreasing with $g = 1$ on $[0,1]$ and $g = 0$ on $[2,\infty]$; then let $g_k(t) = g(2^k t)$.  For $\cM \models \mathrm{T}$ and $\mathbf{x} \in \prod_{i \in I} D_i^{\cM}$, let
		\[
		\phi^{\cM}(\mathbf{x}) = \psi_0^{\cM}(\mathbf{x}) + \sum_{k=0}^\infty g_k(|\psi_{k+1}^{\cM}(\mathbf{x}) - \psi_k^{\cM}(\mathbf{x})|) (\psi_{k+1}^{\cM}(\mathbf{x}) - \psi_k^{\cM}(\mathbf{x})).
		\]
		Note that $g_k(|\psi_{k+1}^{\cM}(\mathbf{x}) - \psi_k^{\cM}(\mathbf{x})|) (\psi_{k+1}^{\cM}(\mathbf{x}) - \psi_k^{\cM}(\mathbf{x}))$ is a formula because it is composed of formulas and continuous connectives.  Furthermore, note that because $g_k(|t|) = 0$ unless $|t| \leq 2/2^k$, we have that $|g_k(|\psi_{k+1}^{\cM}(\mathbf{x}) - \psi_k^{\cM}(\mathbf{x})|) (\psi_{k+1}^{\cM}(\mathbf{x}) - \psi_k^{\cM}(\mathbf{x}))| \leq 2/2^k$, which implies that the summation converges uniformly.  Hence, $\phi$ is a definable predicate relative to $\mathrm{T}$.
		
		Moreover, note that if $\mathbf{x} \in \prod_{i \in I} D_i^{\cM}$, then we have
		\[
		\phi^{\cM}(\mathbf{x}) = f_0(\tp^{\cM}(\mathbf{x})) + \sum_{k=0}^\infty g_k(|f_{k+1}(\tp^{\cM}(\mathbf{x})) - f_k(\tp^{\cM}(\mathbf{x}))|) (f_{k+1}(\tp^{\cM}(\mathbf{x})) - f_k(\tp^{\cM}(\mathbf{x}))).
		\]
		Since $\norm{f_{k+1} - f_k}_u \leq 1/2^k$ in $C(\mathbb{S}_{\mathbf{D}}(\mathrm{T}))$, the function $g_k$ evaluates to $1$ and the series telescopes so that
		\[
		\phi^{\cM}(\mathbf{x}) = \lim_{k \to \infty} f_k(\tp^{\cM}(\mathbf{x})) = f(\tp^{\cM}(\mathbf{x})). \qedhere
		\]
	\end{proof}
	
	Definable predicates relative to $\rT$ have many of the same properties as formulas.  Each definable predicate is uniformly continuous, where the same modulus of continuity serves as an upper bound in all models of $\rT$; see \cite[Observation 3.11]{JekelCoveringEntropy}.  The most important fact for this work is that definable predicates are closed under the same operations as formulas.
	
	\begin{fact}[{See \cite[Proposition 9.3]{BYBHU2008} and \cite[Lemma 3.12]{JekelCoveringEntropy}}] ~
		\begin{enumerate}[(1)]
			\item If $I$ and $J$ are index sets, $F: \R^J \to \R$ is continuous (where $\R^J$ has the product topology) and $(\phi_j)_{j \in J}$ are definable predicates $\prod_{i \in I} S_i \to \R$ in $\mathcal{L}$ relative to $\rT$, then $F((\phi_j)_{j\in J})$ is a definable predicate.
			\item If $\phi$ is a definable predicate $\prod_{i \in I} S_i \times \prod_{i \in I'} S_i' \to \R$ in $\mathcal{L}$ relative to $\rT$ and $\mathbf{D}' \in \prod_{i \in I'} \cD_{S_i'}$, then
			\[
			\psi^{\cM}(\mathbf{x}) := \inf_{\mathbf{y} \in \prod_{i \in I'} (D_i')^{\cM}} \phi(\mathbf{x},\mathbf{y})
			\]
			is also definable predicate in $\mathcal{L}$ relative to $\rT$.  The same holds for $\sup$ instead of $\inf$.
		\end{enumerate}
	\end{fact}
	
	Just as with types, there is a version of definable predicates based on expanding the language with constants from some set $A$.
	
	\begin{definition}
		Let $I$ be an index set, and let $\mathbf{S} = (S_i)_{i \in I}$ be an $I$-tuple of sorts in a language $\mathcal{L}$.  Let $\cM$ be an $\cL$-structure and $A \subseteq \cM$.  Then $\phi: \prod_{i \in I} S_i^{\cM} \to \R$ is a \emph{definable predicate over $A$} if it is a definable predicate in with respect to the language $\cL_A$ and the $\cL_A$-theory of $\cM_A$.
	\end{definition}

	We close with the definition of a definable function which plays a key role in Theorem \ref{thm: definable realization}.
	
	\begin{definition} \label{def: definable function}
		Let $\cL$ be a language and $\rT$ a consistent $\cL$-theory.  Let $I$ be an index set, $(S_i)_{i \in I}$ a tuple of sorts, and $S$ another sort.  A \emph{definable function} $\prod_{i \in I} S_i \to S$ relative to $\rT$ is a collection of functions $f^{\cM}: \prod_{i \in I} S_i^{\cM} \to S^{\cM}$ for each $\cM \models \rT$, such that $\phi^{\cM}(\mathbf{x},y) = d_S^{\cM}(f^{\cM}(\mathbf{x}),y)$ is a definable predicate.
	\end{definition}

	\subsection{Elementary extensions, saturation, and strong homogeneity} \label{subsec: elementary extension}
	
	As mentioned in the introduction, if two tuples $(X_1,\dots,X_n)$ and $(Y_1,\dots,Y_n)$ of random variables on $(\Omega,P)$ have the same probability distribution, then they are approximately conjugate by automorphisms.  They are not necessarily \emph{exactly} conjugate.  For instance, there are cases where $(X_1,\dots,X_n)$ generates the underlying $\sigma$-algebra and $(Y_1,\dots,Y_n)$ does not, so it is impossible for them to be automorphically conjugate.  However, by enlarging the probability space, we can arrange that they are conjugate.
	
	For tracial von Neumann algebras and metric structures in general, one can produce a similar enlargement in which tuples with the same type are necessarily conjugate by an automorphism.  For this purpose, it is not enough to include $\cM$ into a larger von Neumann algebra by a $*$-homomorphism, because a $*$-homomorphism does not necessarily preserve the values of formulas with suprema and infima in them, as it changes the domain over which the suprema and infima are taken.  Rather, we need to require that our inclusion $\cM \to \cN$ preserves the values of formulas, which is the concept of an elementary extension.
	
	\begin{definition}
	A \emph{substructure} $\cM \subseteq \cN$ is an $\cL$-structure where $S^{\cM} \subseteq S^{\cN}$ for each sort $S$, such that for each domain $D$ in a sort $S$, we have $D^{\cM} = D^{\cM} \cap S^{\cN}$, and for each function symbol $f: S_1 \times \dots \times S_n \to S$, we have $f^{\cM} = f^{\cN}|_{S_1^{\cM} \times \dots \times S_n^{\cM}}$.
	\end{definition}
	
	In the case of von Neumann algebras, the reader may verify that the definition of substructure here coincides with the notion of von Neumann subalgebra, by using the Kaplansky density theorem to show that $D^{\cM} = D^{\cN} \cap S^{\cM}$.  In fact, this usage of the Kaplansky density theorem is already incorporated into the proof that von Neumann algebras can be axiomatized as $\cL_{\tr}$-structures with the domains corresponding to operator norm balls \cite[Proposition 3.3]{FHS2014}.
	
	\begin{definition}
		Let $\cM$ be a substructure of $\cN$.  We say that $\cM$ is an \emph{elementary substructure} of $\cN$, or $\cN$ is an \emph{elementary extension} of $\cM$, if for every formula $\phi: S_1 \times \dots \times S_n \to \R$ and every $\mathbf{x} \in S_1^{\cM} \times \dots \times S_n^{\cM}$, we have $\phi^{\cM}(\mathbf{x}) = \phi^{\cN}(\mathbf{x})$.
	\end{definition}
	
	The following is a widely applicable criterion for $\cM$ being an elementary submodel of $\cN$.
	
	\begin{proposition}[Continuous Tarski-Vaught test {\cite[Proposition 4.5]{BYBHU2008}}] \label{prop: TV test}
		Let $\cM$ and $\cN$ be $\cL$-structures with $\cM \subseteq \cN$.  Then the following are equivalent:
		\begin{enumerate}
			\item $\cM \preceq \cN$.
			\item For every formula $\phi$ in $n+1$ variables from sorts $S_1$, \dots, $S_{n+1}$, for every domain $D$ in $S_{n+1}$, for all $x_1 \in S_1^{\cM}$, \dots, $x_n \in S_n^{\cM}$,
			\[
			\inf_{y \in D^{\cN}} \phi^{\cN}(x_1,\dots,x_n,y) = \inf_{y \in D^{\cM}} \phi^{\cN}(x_1,\dots,x_n,y).
			\]
		\end{enumerate}
	\end{proposition}
	
	The first thing that we hope to achieve by constructing a suitable elementary extension is \emph{saturation}, which roughly means that if there is a tuple $(x_i)_{i \in I}$ that approximately satisfy some relations, then there is a tuple that exactly satisfies these relations.  However, to make this easier to accomplish, we will limit the cardinality $|I|$ of the tuples under consideration to be much smaller than the density character of the ambient metric space. 
	
	\begin{definition}
		The \emph{density character} $\chi(X)$ of a metric space $X$ is the minimum cardinality of a dense subset.  Thus, $X$ is separable if and only if $\chi(X) \leq \aleph_0$.
	\end{definition}
	
	\begin{definition} \label{def: language density character}
		If $\cL$ is a language and $\kappa$ is a cardinal, we say $\cL$ has \emph{density character at most} $\kappa$ if
		\begin{enumerate}
			\item The set of sorts and the set of domains have cardinality at most $\kappa$.
			\item For any index set $I$ with $|I| \leq \kappa$ and any set of variables $(x_i)_{i \in I}$ from sorts $(S_i)_{i \in I}$, there exists $\Omega \subseteq \mathcal{F}_{\mathbf{S}}$ with $|\Omega| \leq \kappa$ such that for every formula $\phi \in \mathcal{F}_{\mathbf{S}}$ every tuple $(D_i)_{i \in I}$ of domains, and every $\epsilon > 0$, there exists a formula $\psi \in \Omega$ such that $|\phi^{\cM} - \psi^{\cM}| \leq \epsilon$ on $\bD^{\cM}$ for all $\cL$-structures $\cM$.
		\end{enumerate}
		The density character $\chi(\cL)$ is the smallest cardinal $\kappa$ such that $\cL$ has density character at most $\kappa$.  In particular, if $\chi(\cL) \leq \aleph_0$, then $\cL$ is called \emph{separable}.
	\end{definition}
	
	An important result that we will use later on is that from a large model we can extract elementary submodels of smaller density character that contain whichever elements we are working with.
	
	\begin{theorem}[{Downward L{\"o}wenheim-Skolem theorem \cite[Proposition 7.3]{BYBHU2008}}] \label{thm: DLS}
		Let $\cM$ be an $\cL$-structure and $A \subseteq \cM$.  Then there exists an elementary substructure $\cN \prec \cM$ containing $A$ such that $\chi(\cN) \leq \max(\chi(\cL),\chi(A))$.
	\end{theorem}
	
	Now we are ready to formally define saturation.  See also 
	\cite[\S 4.4]{FHS2014} and \cite[p.\ 35ff]{BYBHU2008}.
	
	\begin{definition}
		Let $\cM$ be an $\cL$-structure and let $\Phi$ be a set of $\cL$-formulas in variables $(x_i)_{i \in I}$.  For each $i$, let $D_i$ be a domain. Then $\Phi$ is \emph{satisfiable in $\prod_{i \in I} D_i^{\cM}$} if there exists some $\mathbf{x} \in \prod_{i \in I} D_i^{\cM}$ with $\phi(\mathbf{x}) = 0$.  We say $\Phi$ is \emph{finitely satisfiable in $\prod_{i \in I} D_i^{\cM}$} if every finite subset of $\Phi$ is satisfiable in $\cM$.  We say $\Phi$ \emph{finitely approximately satisfiable in $\prod_{i \in I} D_i^{\cM}$} if for every $\delta > 0$ and finite set of formulas $F \subseteq \Phi$, there exists some $\mathbf{x} \in \prod_{i \in I} D_i^{\cM}$ with $|\phi(\mathbf{x})| < \delta$ for $\phi \in F$.
	\end{definition}
	
	\begin{definition}
		Assume $\chi(\cL) < \kappa$.  An $\cL$-structure $\cM$ is said to be \emph{$\kappa$-saturated} if for any $A \subseteq \cM$ with $|A| < \kappa$, any set of variables $(x_i)_{i \in I}$ from $(D_i)_{i \in I}$ with $|I| < \kappa$, and any set $\Phi$ of $\cL_A$-formulas in $(x_i)_{i \in I}$, if $\Phi$ is finitely approximately satisfiable in $\prod_{i \in I} D_i^{\cM_A}$, then $\Phi$ is satisfiable in $\prod_{i \in I} D_i^{\cM_A}$.
	\end{definition}
	
	The notion of $\kappa$-saturation has as strong resemblance to compactness, since both are conditions that allow one to deduce existence of elements satisyfing some relations exactly from knowing the existence of elements that satisfy them approximately.  In fact, the following observation gives an explicit connection with compact sets.
	
	\begin{proposition} \label{prop: compact saturation}
		Let $\cL$ be a metric signature and let $\cM$ be an $\cL$-structure.  Suppose that $D^{\cM}$ is compact for each domain $D$ in each sort $S$.  Then $\cM$ is $\kappa$-saturated for every $\kappa > \chi(\cL)$.
	\end{proposition}
	
	\begin{proof}
		Fix $\kappa > \chi(\cL)$, let $A \subseteq \cM$ with $\chi(A) < \kappa$.  Let $\Phi$ be a set of $\cL_A$-formulas in variables $x_i$ from sort $S_i$, indexed by a set $I$ with $|I| < \kappa$.  Let $D_i$ be a domain in $S_i$.  Suppose that $\Phi$ is finitely approximately satisfiable in $\prod_{i \in I} D_i^{\cM}$.  For $\delta > 0$ and $\phi \in \Phi$, let
		\[
		K_{\phi,\delta} = \left\{\mathbf{x} \in \prod_{i \in I} D_{r_i}^{\cM}: |\phi^{\cM_A}(\mathbf{x})| \leq \delta \right\}.
		\]
		Then $K_{\phi,\delta}$ is a closed subset of the compact set $\prod_{i \in I} D_{r_i}^{\cM}$.  Moreover, finite approximate satisfiability of $\Phi$ means that any finite collection of sets $K_{\phi,\delta}$ has nonempty intersection.  Therefore, by compactness, the whole collection of $K_{\phi,\delta}$'s has nonempty intersection, which means that $\Phi$ is satisfiable.
	\end{proof}
	
	We close the section with the notion of strong homogeneity, which is the property of a structure $\cM$ that allows us to identify types with automorphism orbits.
	
	\begin{definition}
		Let $\kappa > \chi(\cL)$.  An $\cL$-structure $\cM$ is said to be \emph{strongly $\kappa$-homogeneous} if whenever $(S_i)_{i \in I}$ is a tuple of sorts with $|I| < \kappa$ and $\mathbf{x}, \mathbf{y} \in \prod_{i \in I} S_i^{\cM}$ such that $\tp^{\cM}(\mathbf{x}) = \tp^{\cM}(\mathbf{y})$, there exists an automorphism $\Theta$ of $\cM$ such that $\Theta(\mathbf{x}) = \mathbf{y}$ (that is, $\Theta(x_i) = y_i$ for all $I \in I$).
	\end{definition}
	
	\begin{theorem}[{\cite[Proposition 7.12]{BYBHU2008}}] \label{thm: existence of nice model}
		Let $\cL$ be a metric signature and $\kappa > \chi(\cL)$ an infinite cardinal.  Every $\cL$-structure $\cM$ has an elementary extension that is $\kappa$-saturated and strongly $\kappa$-homogeneous.
	\end{theorem}
	
	\begin{remark}
		In \cite{BYBHU2008} and \cite{Hart2023}, the definition of saturation is given as above with $|I|$ finite.  Thus, the definition given here for saturation is \emph{a priori} stronger than theirs.  The two definitions can be shown to be equivalent.  However, for our purposes, it is sufficient to know that the proof of the theorem above via \cite[Lemma 7.9, Proposition 7.10, Proposition 7.12]{BYBHU2008} works just as well with a set of variables of cardinality less than $\kappa$ rather than a finite set of variables.  The only modification is to use ultrafilters on the set of finite subsets of some index set $J$, with $|J| \geq \kappa$ rather than merely $|J| > \chi(\cL)$.
	\end{remark}
	
	
	
	\section{Definable and algebraic closures} \label{sec: def and alg closures}
	
	In this section, we review the definable closure and the related concept of algebraic closure.  We give several examples of definable and algebraic closures in tracial von Neumann algebras.  In particular, we completely compute the definable and algebraic closures in the finite-dimensional setting (Example \ref{ex: finite dimensional}), and in the setting of $\mathrm{II}_1$ factors, we relate the definable and algebraic closures with the normalizer for inclusions with spectral gap (Proposition \ref{prop: spectral gap normalizer}).
	
	\subsection{Definition and Properties} \label{subsec: def and alg closure def}
	
	\begin{definition}
		Let $\cM$ be an $\cL$-structure and let $A$ be a subset of $\cM$, and let $C$ be a closed subset of $S^{\cM}$ for some sort $S$.  We say that $C$ is \emph{definable over $A$} if $d^{\cM}(x,C) = \phi^{\cM}(x)$ for some definable predicate $\phi$ with respect to $\Th(\cM_A)$ in the language $\cL_A$.
	\end{definition}
	
	\begin{definition} \label{def: DCL and ACL}
		Let $\cM$ be an $\cL$-structure and let $A$ be a subset of $\cM$.
		\begin{enumerate}[(1)]
			\item The \emph{definable closure} of $A$ in $\cM$, denoted $\dcl^{\cM}(A)$, is the set of $x$ in $\cM$ such that $\{x\}$ is definable over $A$.
			\item The \emph{algebraic closure} of $A$ in $\cM$, denoted $\acl^{\cM}(A)$, is the set of $x$ from some domain $D^\cM$ such that there exists a compact set $C \subseteq D^{\cM}$ with $x \in C$, such that $C$ is definable over $A$.
		\end{enumerate}
	\end{definition}
	
	We will use the following fact from \cite{BYBHU2008}.
	
	\begin{proposition}[{\cite[Corollary 10.5]{BYBHU2008}}] \label{prop: closure and elem ext}
		Let $\cM$ be an $\cL$-structure, and let $\cN$ be an elementary extension of $\cM$.  Let $A \subseteq \cM$.  Then $\dcl^{\cM}(A) = \dcl^{\cN}(\cA)$ and $\acl^{\cM}(A) = \acl^{\cN}(A)$.
	\end{proposition}
	
	\begin{proposition}[{compare \cite[Exercise 10.7]{BYBHU2008}}] \label{prop:DCL}
		Let $A$ be a subset of $\cM$.  Suppose $\chi(\cL) < \kappa$ and $\chi(A) < \kappa$.  Let $D$ be a domain in a sort $S$, and let $x \in D^{\cM}$.
		\begin{enumerate}[(1)]
			\item If $x \in \dcl^{\cM}(A)$, then $x$ is the unique realization in $\cM$ of its type over $A$ in $\cM$.
			\item The converse to (1) holds if $\cM$ is $\kappa$-saturated.
			\item If $x \in \dcl^{\cM}(A)$, then every automorphism of $\cM$ that fixes $A$ pointwise also fixes $x$.
			\item The converse to (3) holds if $\cM$ is $\kappa$-saturated and strongly $\kappa$-homogeneous.
		\end{enumerate}
	\end{proposition}
	
	\begin{proof}
		(1) Let $y$ be another element in $\cM$, from the same sort $S$ as $x$, with the same type over $A$ as $x$.  Since $d_S^{\cM}(z,x) = \phi^{\cM}(z)$ for some $\Th(\cM_A)$-definable predicate $\phi$ in $\cL_A$, we have $d_S^{\cM}(y,x) = d_S^{\cM}(x,x) = 0$, so $y = x$.
		
		(2) Given $\epsilon > 0$, let $\mathcal{E}_\epsilon$ be the set of $\cL_{A \cup \{x\}}$-formulas in a variable $y$ given by
		\[
		\{ \phi(y): \phi \text{ is an } \cL_A\text{-formula and } \phi^{\cM_A}(x) = 0\} \cup \{ \epsilon \dot{-} d_S^{\cM}(y,x) \},
		\]
		where $a \dot{-} b = \max(a - b,0)$.  Note that $y$ satisfies $\mathcal{E}_\epsilon$ if and only if $d_S^{\cM}(x,y) \geq \epsilon$ and $y$ has the same type over $A$ as $x$.  Because $x$ is the unique realization of its type, $\mathcal{E}_{\epsilon}$ is not satisfiable in $D^{\cM}$.  Hence, by $\kappa$-saturation of $\cM$, it is not finitely approximately satisfiable either.  Thus, for every $\epsilon > 0$, there exist $\delta > 0$ and finitely many $\cL_A$-formulas $\phi_1$, \dots, $\phi_n$ such that $\phi_j^{\cM_A}(x) = 0$ and if $|\phi_j^{\cM_A}(y)| < \delta$ for $j = 1$, \dots, $n$, then $d_S^{\cM_A}(x,y) < \epsilon$.  In fact, we can reduce this to a single $\cL_A$-formula by taking $\psi = |\phi_1| + \dots + |\phi_n|$.  By taking $\epsilon = 1/m$, we see that there are $\delta_m > 0$ and $\cL_A$-formulas $\psi_m$ such that $\psi_m^{\cM_A}(x) = 0$ and if $y \in D^{\cM}$ and $|\psi_m^{\cM_A}(y)| < \delta_m$, then $d_S^{\cM}(x,y) < 1/m$.  Thus, by \cite[Proposition 9.19]{BYBHU2008}, $\{x\}$ is definable in $\cM$ over $A$.
		
		(3) Let $\Theta$ be an automorphism of $\cM$ that fixes $A$ pointwise.  Then the type of $\Theta(x)$ over $A$ is the same as the type of $x$ over $A$.  Hence, by (1), $\Theta(x) = x$.
		
		(4) Suppose that $x$ is not in $\dcl^{\cM}(A)$.  Then by (2), there exists another $y$ with the same type over $A$ as $x$.  Then $(A,x)$ and $(A,y)$ viewed as tuples have the same type, so by strong $\kappa$-homogeneity, there exists an automorphism $\Theta$ taking $(A,x)$ and $(A,y)$ as tuples, i.e.\ fixing $A$ pointwise and taking $x$ to $y$.
	\end{proof}
	
	\begin{proposition}[{compare \cite[Exercise 10.8]{BYBHU2008}}] \label{prop: ACL} 
		Let $\kappa > \chi(\cL)$ be an uncountable ordinal.  Let $A$ be a subset of $\cM$ with density character less than $\kappa$, and suppose that $\cM$ is $\kappa$-saturated.  Let $S$ be a sort, $D$ a domain for $S$, and $x \in D^{\cM}$.  Then the following are equivalent:
		\begin{enumerate}[(1)]
			\item $x \in \acl^{\cM}(A)$.
			\item The set of realizations of $\tp^{\cM}(x / A)$ is compact.
			\item The set of realizations of $\tp^{\cM}(x / A)$ has density character less than $\kappa$.
			\item For every $\epsilon > 0$, there exists an $\cL_{A}$-formula $\phi$ and a $\delta > 0$ such that $\phi^{\cM}(x) = 0$ and the set $\{y \in D^{\cM}: \phi^{\cM}(y) <  \delta \}$ can be covered by finitely many $\epsilon$-balls.
			\item $x$ is in the intersection of all elementary submodels of $\cM$ that contain $A$.
		\end{enumerate}
		Furthermore, (1) $\implies$ (2) and (1) $\implies$ (5) hold even without assuming that $M$ is $\kappa$-saturated.
	\end{proposition}
	
	\begin{proof}
		(1) $\implies$ (2).  Since $x \in \acl^{\cM}(A)$, there exists a compact set $C \subseteq D^{\cM}$ definable over $A$ in $\cM$ such that $x \in C$.  Since $C$ is definable over $A$, if $y$ has the same type over $\cA$ as $x$ does, then $y \in C$.  Thus, the set of realizations of $\tp^{\cM}(x / A)$ is contained in the compact set $C$, and it is closed because formulas are continuous.  Thus, the set of realizations of $\tp^{\cM}(x/\cA)$ is compact.
		
		(2) $\implies$ (3) because every compact set in a metric space has (at most) countable density character.
		
		(3) $\implies$ (4).  We proceed by contrapositive.  Suppose that (4) fails for some $\epsilon$.  Let $\kappa' < \kappa$.  Let $\mathcal{E}$ be the set of $\cL_{A}$-formulas $\phi: S \to \R$ such that $\phi^{\cM}(x) = 0$.  Let $\Sigma$ be the set of $\mathcal{L}_{\cA}$ conditions in $\kappa'$ many variables $(x_i)_{i \in I}$ given by
		\[
		\{\phi(x_i) = 0: i \in I, \phi \in \mathcal{E} \} \cup \{\epsilon \dot{-} d(x_i,x_j) : i, j \in I, i \neq j \}.
		\]
		Note that $\Sigma$ is finitely approximately satisfiable in $\cM$.  Indeed, let $\delta > 0$ and consider a finite collection of formulas $\phi_1$, \dots, $\phi_n \in \mathcal{E}$.  Let $\phi = \max(|\phi_1|,\dots,|\phi_n|)$.  By assumption, there exist infinitely many $y_j$'s satisfying $\phi^{\cM}(y_j) < \delta$ and $d(y_j,y_k) \geq \epsilon$ for $j \neq k$.  This is enough to be able to approximately satisfy any finite collection of formulas from $\Sigma$.  Since $|\Sigma| \leq \kappa' < \kappa$ and $\cM$ is $\kappa$-saturated, $\Sigma$ is satisfiable in $\cM$.  This means there exist $\kappa'$ many $\epsilon$-separated elements $x_j \in D^{\cM}$ such that $\phi^{\cM}(x_j) = 0$ for all $\phi \in \mathcal{E}$, and hence $\tp^{\cM}(x_j / A) = \tp^{\cM}(x / A)$.  Thus, the space of realizations of this type has density character greater than or equal to $\kappa'$.  Since $\kappa' < \kappa$ was arbitrary, its density character is at least $\kappa$.
		
		(4) $\implies$ (1).  For each $k \in \N$, fix some formula $\phi_k$ and some $\delta_k > 0$ such that $\phi_k^{\cM}(x) = 0$ and the set $\{y: \phi_k^{\cM}(y) < \delta_k \}$ can be covered by finitely many $\epsilon$-balls.  Let $C = \{y: \phi_k^{\cM}(x) = 0 \text{ for all } k \in \N\}$.  Then $C$ is a countable intersection of zerosets of definable predicates and hence is the zeroset of a definable predicate by \cite[Proposition 9.14]{BYBHU2008}.  Note $C$ is closed (hence complete) by continuity of formulas.  Moreover, by our choice of $\phi_k$, $C$ is totally bounded.  Thus, $C$ is compact.  By \cite[Proposition 10.6]{BYBHU2008}, a compact zeroset is definable.  Thus, $C$ is a compact definable set containing $x$, and so $x \in \acl^{\cM}(A)$.
		
		(1) $\implies$ (5).  If $x \in \acl^{\cM}(A)$, then there exists a compact definable set $C$ such that $x \in C$.  By \cite[Proposition 10.4]{BYBHU2008}, $C$ (and hence $x$) must be contained in any elementary submodel of $\cM$ that contains $A$.  Note that this argument does not use saturation of $\cM$.
		
		(5) $\implies$ (3).  We proceed by contrapositive.  Suppose that the space of realizations of $\tp^{\cM}(x / A)$ in $\cM$ has density character greater than or equal to $\kappa$.  By the downward L{\"o}wenheim-Skolem theorem (Theorem \ref{thm: DLS}), there exists an elementary submodel $\cP$ of $\cM$ that contains $A$ and has density character equal to that of $A$.  Because the set of realizations of $\tp^{\cM}(x/A)$ has strictly greater density character, there must be some realization $y$ of $\tp^{\cM}(x / A)$ that is not contained in $\cP$.  Let $\mathbf{z}$ be a tuple of elements of $A$ on an index set $I$ of cardinality $\chi(A)$ so that $\mathbf{z}$ contains a dense subset of $A \cap D$ for each domain $D$.  Then $(x,\mathbf{z})$ and $(y,\mathbf{z})$ have the same type because every $\cL$-formula in $(x,\mathbf{z})$ is an $\cL$-formula in $x$.
		
		Let $\mathbf{w}$ be a tuple of elements of $\mathcal{P}$ with index set $I'$ of cardinality less than $\kappa$ such that the elements of $\mathbf{w}$ in each domain $D_0$ are dense in $\cP \cap D^{\cM}$; let $w_i$ be from the sort $S_i$ and domain $D_i$ for $i \in I'$.  We claim that there exists a tuple $\tilde{\mathbf{w}}$ indexed by $I'$, of elements of $\cM$ from appropriate sorts and domains, such that $\tp^{\cM}(x,\mathbf{z},\tilde{\mathbf{w}}) = \tp^{\cM}(y,\mathbf{z},\mathbf{w})$.  To see this, consider the set $\Phi$ of formulas $\phi$ with variables indexed by $\{0\} \sqcup I \sqcup I'$ given by
		\[
		\Phi = \{\phi: \phi^{\cM}(y,\mathbf{z},\mathbf{w}) = 0\}.
		\]
		Let $B$ be the set of values taken by the tuple $(x,\mathbf{z})$, and let $\Phi'$ be the set of $L_B$ formulas on variables indexed by $I'$, obtained from $\phi \in \Phi$ be plugging in $(x,\mathbf{z})$ for the first two arguments and viewing the formula as a function of the last argument.  We claim that $\Phi'$ is finitely approximately satisfiable in $\cM$ as an $\cL_B$ structure.  To see this, consider a finite list of formulas $\phi_1'$, \dots, $\phi_n'$ from $\Phi'$ corresponding to $\phi_1$, \dots, $\phi_n$ from $\Phi$.  Note that $\phi_1$, \dots, $\phi_n$ only depend on a finite subset of the variables (by definition of formulas).  Hence, it makes sense to define a formula in variables $(x',\mathbf{z}',\mathbf{w}')$ by
		\[
		\psi(x',\mathbf{z}') = \inf_{\mathbf{w} \in \prod_{i \in I'} D_i} \sum_{j=1}^n |\phi_j(x',\mathbf{z}',\mathbf{w}')|.
		\]
		Note that $\psi^{\cM}(y,\mathbf{z}) = 0$ because we can plug in $\mathbf{w}$ for the variable $\mathbf{w}'$.  Since $(y,\mathbf{z})$ and $(x,\mathbf{z})$ have the same type, we also get $\psi^{\cM}(x,\mathbf{z}) = 0$.  Therefore, for every $\epsilon > 0$, there exists $\tilde{\mathbf{w}}$ with
		\[
		\sum_{j=1}^n |(\phi_j')^{\cM_B}(\tilde{\mathbf{w}})| = \sum_{j=1}^n |\phi_j^{\cM}(x,\mathbf{z},\tilde{\mathbf{w}})| < \epsilon.
		\]
		This demonstrates the finite approximate satisfiability of $\Phi'$.  Then since $\cM$ is $\kappa$-saturated, $\Phi'$ is satisfiable, so there exists $\tilde{\mathbf{w}} \in \prod_{i \in I'} D_i^{\cM}$ such that for $\phi' \in \Phi$, we have $(\phi')^{\cM}(\tilde{\mathbf{w}}) = 0$.  This means that for $\phi \in \Phi$, we have $\phi^{\cM}(x,\mathbf{z},\tilde{\mathbf{w}}) = 0$.  Note that for every formula $\phi$, there is some constant $c$ such that $\phi - c \in \Phi$.  Therefore, we obtain that $\phi^{\cM}(x,\mathbf{z},\tilde{\mathbf{w}}) = \phi^{\cM}(y,\mathbf{z},\mathbf{w})$ for all formulas $\phi$, that is, $\tp^{\cM}(x,\mathbf{z},\tilde{\mathbf{w}}) = \tp^{\cM}(x,\mathbf{z},\mathbf{w})$.
		
		Our goal in the rest of the proof is to show that the substructure $\tilde{\cP}$ generated by $\tilde{\mathbf{w}}$ is an elementary substructure that contains $\mathbf{y}$ but not $x$.  First, we give a concrete description of this substructure.  For each sort $S_0$ and domain $D_0$, for each $w_i$ and $w_j$ in $D_0^{\cP}$, we have that $d_{S_0}^{\cM}(w_i,w_j) = d_{S_0}^{\cM}(\tilde{w}_i,\tilde{w}_j)$, and in particular, if $w_i = w_j$, then $\tilde{w}_i = \tilde{w}_j$.  We chose $\mathbf{w}$ such that the elements of $\mathbf{w}$ in a domain $D_0^{\cM}$ are dense in $D_0^{\cP}$.  Since an isometry defined on a dense subset of a metric space extends to the entire metric space if the codomain is complete, there is an isometry $\alpha_{D_0}: D_0^{\cP} \to D_0^{\cM}$ such that $\alpha_{D_0}(w_i) = \tilde{w}_i$ whenever $w_i \in D_0$.  Moreover, for two domains $D_0$ and $D_1$ in the sort $S$, we have that $\alpha_{D_0}$ and $\alpha_{D_1}$ agree on $D_0^{\cP} \cap D_1^{\cP}$; to verify this, it suffices to show that $d_S^{\cM}(\alpha_{D_0}(\xi),\alpha_{D_1}(\eta)) = d_S^{\cM}(\xi,\eta)$ holds for $\xi \in D_0^{\cP}$ and $\eta \in D_1^{\cP}$; this holds when $\xi$ and $\eta$ come from $\mathbf{w}$, by virtue of $\mathbf{w}$ and $\tilde{\mathbf{w}}$ having the same type, and then it extends to all of $D_0^{\cP}$ and $D_1^{\cP}$ by density.  Define $\tilde{\cP}$ by setting $D_0^{\tilde{\cP}} := \alpha_{D_0}(D_0^{\cP})$ for each sort $S_0$ and domain $D_0$, and setting $S_0^{\tilde{\cP}} = \bigcup_{D_0 \in \mathcal{D}_{S_0}} D_0^{\tilde{\cP}}$.  We claim $\tilde{\cP}$ is a substructure.  First, $D^{\tilde{\cP}}$ is complete because is the image of a complete metric space under an isometry.  Second, we must show that $D_0^{\cM} \cap S_0^{\tilde{\cP}} = D_0^{\tilde{\cP}}$ for each domain $D_0$; because $D_0^{\cM} \cap S_0^{\cP} = D_0^{\cP}$, it suffices to check that for every domain $D_1$ for $\xi \in D_1^{\cP}$, we have
		\[
		\inf_{\eta \in D_0^{\cM}} d_S^{\cM}(\alpha_{D_1}(\xi),\eta) = \inf_{\eta \in D_0^{\cM}} d_S^{\cM}(\xi,\eta);
		\]
		again this holds when $\xi$ is an element of $\mathbf{w}$ because $\mathbf{w}$ and $\tilde{\mathbf{w}}$ have the same type, and then it extends by continuity.  Third, one shows that $\tilde{P}$ is closed under application of function symbols by verifying, in a similar manner to previous claims, that for each function symbol $f$ that maps domains $D_1 \times \dots \times D_n$ into a domain $D_0$, we have
		\[
		d^{\cM}(f(\alpha_{D_1}(\xi_1),\dots,\alpha_{D_n}(\xi_n)),\alpha_{D_0}(\eta)) = d^{\cM}(f^{\cM}(\xi_1,\dots,\xi_n),\eta) \text{ for } \xi_1 \in D_1^{\cP}, \dots, \xi_n \in D_n^{\cP}, \eta \in D_0^{\cP},
		\]
		and hence $f^{\cM}(\alpha_{D_1}(\xi_1),\dots,\alpha_{D_n}(\xi_n)) = \alpha_{D_0}(f^{\cM}(\xi_1,\dots,\xi_n))$.
		
		Next, we verify that $\tilde{\cP} \preceq \cM$ using the Tarski-Vaught test.  Let $\phi$ be a formula in variables from sorts $S_1$, \dots, $S_{n+1}$, and let $D$ be a domain for $S_{n+1}$.  Let $\tilde{w}_{i_1}$, \dots, $\tilde{w}_{i_n}$ be elements of $\tilde{w}$ from $S_1$, \dots, $S_n$ respectively.  Because the elements $w_i$ in a domain $D^{\cM}$ are dense in $D^{\cP}$, hence also the elements $\tilde{w}_i = \alpha_D(w_i)$ in $D^{\cM}$ are dense in $D^{\tilde{\cP}}$, we have that
		\[
		\inf_{\eta \in D^{\tilde{\cP}}} \phi^{\cM}(\tilde{w}_{i_1},\dots,\tilde{w}_{i_n},\eta) = \inf_{\substack{i \in I \\ \tilde{w}_i \in D^{\tilde{\cP}}}} \phi^{\cM}(\tilde{w}_{i_1},\dots,\tilde{w}_{i_n},\tilde{w}_i).
		\]
		Then by virtue of $w$ and $\tilde{w}$ having the same type,
		\[
		\inf_{\substack{i \in I \\ \tilde{w}_i \in D^{\tilde{\cP}}}} \phi^{\cM}(\tilde{w}_{i_1},\dots,\tilde{w}_{i_n},\tilde{w}_i) = \inf_{\substack{i \in I \\ w_i \in D^{\tilde{\cP}}}} \phi^{\cM}(w_{i_1},\dots,w_{i_n},w_i) = \inf_{\eta \in D^{\cP}} \phi^{\cM}(w_{i_1},\dots,w_{i_n},\eta),
		\]
		where the second equality follows from density of $w$ in $\cP \cap D^{\cM}$.  Then since $\cP \preceq \cM$, we obtain
		\[
		\inf_{\eta \in D^{\cP}} \phi^{\cM}(w_{i_1},\dots,w_{i_n},\eta) = \inf_{\eta \in D^{\cM}} \phi^{\cM}(w_{i_1},\dots,w_{i_n},\eta) = \inf_{\eta \in D^{\cM}} \phi^{\cM}(\tilde{w}_{i_1},\dots,\tilde{w}_{i_n},\eta),
		\]
		where the last equality follows from $w$ and $\tilde{w}$ having the same type.  Altogether,
		\[
		\inf_{\eta \in D^{\tilde{\cP}}} \phi^{\cM}(\tilde{w}_{i_1},\dots,\tilde{w}_{i_n},\eta) = \inf_{\eta \in D^{\cM}} \phi^{\cM}(\tilde{w}_{i_1},\dots,\tilde{w}_{i_n},\eta).
		\]
		Because the $\tilde{w}_i$'s that are in $D_j^{\tilde{\cP}}$ form a dense subset of $D_j^{\tilde{\cP}}$, and since both sides are continuous functions of the the first $n$ inputs, we obtain that for $\xi_1 \in D_1^{\tilde{\cP}}$, \dots, $\xi_n \in D_n^{\tilde{\cP}}$,
		\[
		\inf_{\eta \in D^{\tilde{\cP}}} \phi^{\cM}(\xi_1,\dots,\xi_n,\eta) = \inf_{\eta \in D^{\cM}} \phi^{\cM}(\xi_1,\dots,\xi_n,\eta).
		\]
		Thus, by Proposition \ref{prop: TV test}, $\tilde{\cP} \preceq \cM$.
		
		Next, we claim $\tilde{\cP}$ contains $A$.  This is because for each domain $D_0$, each element of $\mathbf{z}$ in the domain $D_0$ is a limit of elements from $\mathbf{w}$ inside $D_0^{\cP}$, hence the same is true for $\tilde{\mathbf{w}}$ since $d^{\cM}(z_i,w_j) = d^{\cM}(z_i,\tilde{w}_j)$ by virtue of the types of $(\mathbf{z},\mathbf{w})$ and $(\mathbf{z},\tilde{\mathbf{w}})$ being the same.
		
		Finally, similar reasoning shows that $\tilde{\cP}$ does not contain $x$.  Indeed, $d(x,D^{\tilde{\cP}})$ is the infimum of distances between $x$ and $\tilde{w}_i$ for $i \in I'$ such that $\tilde{w}_i \in D^{\tilde{P}}$.  But $d(x,\tilde{w}_i) = d(y,w_i)$.  Thus, $d(x,D^{\tilde{\cP}}) = d(y,D^{\cP}) > 0$.  Hence, $\tilde{\cP}$ is an elementary submodel of $\cM$ that contains $A$ but not $x$, which completes the proof.
	\end{proof}
	
	\begin{corollary} \label{cor: DCL substructure}
		Let $\cM$ be an $\cL$-structure and let $A$ be a subset of $\cM$.  Then $\dcl^{\cM}(A)$ and $\acl^{\cM}(A)$ are substructures of $\cM$, that is, they are closed under application of the function symbols and the part of $\dcl^{\cM}(A)$ and $\acl^{\cM}(A)$ in each sort $S$ is a closed subset of $S^{\cM}$.  In particular, for a tracial $\mathrm{W}^*$-algebra $\cM$, the definable and algebraic closures of a set $A$ are $\mathrm{W}^*$-subalgebras.
	\end{corollary}
	
	\begin{proof}
		Let $\cN$ be a $\kappa$-saturated and strongly $\kappa$-homogeneous elementary extension of $\cM$. By Proposition \ref{prop: closure and elem ext}, $\dcl^{\cM}(A) = \dcl^{\cN}(A)$.  If $\Theta$ is an automorphism of $\cN$, then the fixed points of $\Theta$ form a substructure.  Taking the intersection over $\Theta \in \Aut(\cN)$ and using Proposition \ref{prop:DCL}, we obtain that $\dcl^{\cN}(A)$ is a substructure, hence so is $\dcl^{\cM}(A)$.  Similarly, each elementary submodel of $\cN$ containing $A$ is a substructure, and hence so is the intersection of these elementary submodels, which means $\acl^{\cN}(A)$ is a substructure.  By Proposition \ref{prop: closure and elem ext}, $\acl^{\cN}(A) = \acl^{\cM}(A)$.
	\end{proof}
	
	In the case where each domain in $\cM$ is compact, the algebraic closure is rather trivial.  The following statement is likely known to experts, so I make no claim of originality.
	
	\begin{proposition} \label{prop: compact ACL}
		Let $\cL$ be a metric signature and let $\cM$ be an $\cL$-structure.  Suppose that $D^{\cM}$ is compact for each domain $D$ in each sort $S$.  Then
		\begin{enumerate}
			\item $\acl^{\cM}(\varnothing) = \cM$.
			\item If $\cN$ is an $\cL$-structure with $\cM \preceq \cN$, then $\cM = \cN$.
			\item $\cM$ is $\kappa$-saturated and strongly $\kappa$-homogeneous for every $\kappa > \chi(\cL)$.
		\end{enumerate}
	\end{proposition}
	
	\begin{proof}
		(1) We already showed in Proposition \ref{prop: compact saturation} that $\cM$ is $\kappa$-saturated for every $\kappa > \chi(\cL)$ and hence Proposition \ref{prop: ACL} applies.  Let $S$ be a sort, $D$ a domain for $S$, and $x \in D^{\cM}$.  The set of realizations of $\tp^{\cM}(x)$ is a closed subset of $D^{\cM}$ and hence is compact, so by Proposition \ref{prop: ACL}, $x \in \acl^{\cM}(\varnothing)$.
		
		(2)  Let $\cN$ be any elementary extension of $\cM$.  Let $\kappa > \chi(\cL)$, and let $\cN_1$ be an elementary extension of $\cN$ that is strongly $\kappa$-homogeneous.  Fix a sort $S$ and domain $D$ for $S$, and let $y \in D^{\cN_1}$.  Let $\Phi$ be the set of formulas such that $\phi^{\cN_1}(y) = 0$.  For any $\phi_1$, \dots, $\phi_n \in \Phi$, we have
		\[
		\inf_{x \in D^{\cM}} \sum_{j=1}^n |\phi_j^{\cM}(x)| = \inf_{x \in D^{\cN_1}} \sum_{j=1}^n |\phi_j^{\cM}(x)| = 0
		\]
		because $\cM \preceq \cN_1$.  Therefore, $\Phi$ is approximately satisfiable in $\cM$, and hence $\Phi$ is satisfiable in $\cM$ because $\cM$ is $\kappa$-saturated.  If $x \in D^{\cM}$ satisfies $\Phi$, then $\tp^{\cN_1}(x) = \tp^{\cM}(x) = \tp^{\cN_1}(y)$.  Since $\cN_1$ is strongly $\kappa$-homogeneous, there exists an automorphism $\Theta$ of $\cN_1$ with $\Theta(x) = y$.  By (1) and Proposition \ref{prop: closure and elem ext}, $x \in \acl^{\cN_1}(\varnothing)$, which is the intersection of all elementary submodels of $\cN_1$ by Proposition \ref{prop: ACL}. Therefore, because the automorphic image of an elementary submodel is also an elementary submodel, $y = \Theta(x)$ is in the intersection of all elementary submodels of $\cN_1$, which is $\acl^{\cN_1}(\varnothing) = \acl^{\cM}(\varnothing) = \cM$.  Therefore, $\cN_1 = \cM$ and also $\cN = \cM$.
		
		(3) Let $\cN$ be a $\kappa$-saturated and strongly $\kappa$-homogeneous elementary extension of $\cM$.  Then $\cM = \cN$ by (2), so we are done.
	\end{proof}
	
	\subsection{Examples in von Neumann algebras} \label{subsec: def and alg closure examples}
	
	Definable and algebraic closures have not yet been studied in von Neumann algebras (except for their relationship with $1$-bounded entropy in \cite[\S 4.5]{JekelCoveringEntropy}), and therefore, we develop several basic properties and examples here.  Given that the definable closure is always a $\mathrm{W}^*$-subalgebra by Corollary \ref{cor: DCL substructure}, we have $\dcl^{\cM}(A) = \dcl^{\cM}(\mathrm{W}^*(A))$ for $A \subseteq \cM$.  Hence, we restrict our attention to computing the definable and algebraic closures of a $\mathrm{W}^*$-subalgebra $\cA \subseteq \cM$.
	
	\begin{observation}[Definable closure contained in relative bicommutant]
		Let $\cA \subseteq \cM$ be an inclusion of tracial von Neumann algebras.  The definable closure $\dcl^{\cM}(\cA)$ is contained in the relative bicommutant $(\cA' \cap \cM)' \cap \cM$.  In particular, if $\cA$ is a maximal abelian subalgebra of $\cM$, then $\dcl^{\cM}(\cA) = \cA$.
	\end{observation}
	
	\begin{proof}
		Let $x \in \dcl^{\cM}(\cA)$.  To show that $x$ is in the relative bicommutant, it suffices to show that $x$ commutes with every unitary $u \in \cA' \cap \cM$.  Conjugation by $u$ defines an automorphism of $\cM$ which fixes $\cA$ pointwise, and therefore $uxu^*$ has the same type over $\cA$ as $x$ does.  Since $x$ is definable over $\cA$, this implies $uxu^* = x$.  Therefore, $x \in (\cA' \cap \cM)' \cap \cM$.
	\end{proof}
	
	\begin{example}[Definable and algebraic closures in $M_n(\C)$]
		Let $\cM = M_n(\C)$ and $\cA \subseteq \cM$ be a $*$-subalgebra.  Then $\cA'' = \cA$ and hence $\dcl^{M_n(\C)}(\cA) = \cA$.  On the other hand, $\acl^{M_n(\C)}(\cA) = M_n(\C)$ by Proposition \ref{prop: compact ACL}.
	\end{example}
	
	\begin{example}[Definable and algebraic closures for finite-dimensional $\cM$] \label{ex: finite dimensional}
		Next, we will compute the definable closure of $\cA$ in $\cM$ whenever $\cA \subseteq \cM$ is an inclusion of finite-dimensional tracial $*$-algebras (as above, $\acl^{\cM}(\cA) = \cM$).  By Proposition \ref{prop: compact ACL}, $\cM$ is $\kappa$-saturated and strongly $\kappa$-homogeneous.  Therefore, by Proposition \ref{prop:DCL}, $\dcl^{\cM}(\cA)$ is the set of fixed points in $\cM$ of the action of automorphisms that fix $\cA$ pointwise.
		
		The structure theory of finite-dimensional $*$-algebras (see for instance \cite[\S 3.1]{Davidson1996}, \cite[\S 3.2]{JonesSunder1997}) says there is a decomposition
		\[
		\cM \cong \bigoplus_{j=1}^J M_{n_j}(\C), \qquad \tr_{\cM} = \bigoplus_{j=1}^J \beta_j \tr_{n_j}.
		\]
		and a similar decomposition
		\[
		\cA \cong \bigoplus_{i=1}^I M_{m_i}(\C), \qquad \tr_{\cA} = \bigoplus_{i=1}^I \alpha_i \tr_{m_i}.
		\]
		Up to an automorphism of $\cM$, we can assume the inclusion $\iota: \cA \to \cM$ takes the form
		\[
		\iota \left( \bigoplus_{i=1}^I x_i \right) = \bigoplus_{j=1}^J \diag(\underbrace{x_1,\dots,x_1}_{k(1,j)},\dots,\underbrace{x_I,\dots,x_I}_{k(I,j)}),
		\]
		where, as a consequence of the inclusion being trace-preserving,
		\[
		\sum_{j=1}^J k(i,j) \frac{\alpha_j}{n_j} = \frac{\beta_i}{m_i}.
		\]
		
		Suppose $\Theta$ is an automorphism of $\cM$ that fixes $\cA$ pointwise.  Then $\Theta$ must map each summand $M_{n_j}(\C)$ onto another summand $M_{n_{j'}}(\C)$ such that $n_j = n_{j'}$ and $\alpha_j = \alpha_{j'}$.  Besides that, because $\Theta \circ \iota = \iota$, we must have $k(i,j) = k(i,j')$ for all $i$.  Therefore, let us define an equivalence relation on $\{1,\dots,J\}$ by saying that $j \sim j'$ if $n_j = n_{j'}$, $\beta_j = \beta_{j'}$, and $k(i,j) = k(i,j')$ for every $i$.  For each equivalence class $C \subseteq [j]$, let $p_C$ be the central projection onto $\bigoplus_{j \in C} M_{n_j}(\C)$.  Then $p_C$ must be fixed by every automorphism $\Theta$ of $\cM$ that fixes $\cA$ pointwise, and therefore, $\dcl^{\cM}(\cA)$ contains the projections $p_C$.
		
		We claim that in fact
		\[
		\dcl^{\cM}(\cA) = \Span(p_C \cA: C \text{ an equivalence class});
		\]
		we already know the inclusion $\supseteq$ because $\dcl^{\cM}(\cA)$ is a $*$-algebra and contains $\cA$ and the $p_C$'s.  To prove $\subseteq$, suppose that $y = \bigoplus_{j=1}^n y_j$ is fixed by every automorphism of $\cM$ that fixes $\cA$ pointwise.  Since $y$ is the sum of $p_C y$ over the equivalence classes $C$, it suffices to show that for each $C$, we have $p_C y \in p_C \cA$.  Note $p_C y = \bigoplus_{j \in C} y_j$.
		
		Let $\sigma$ be a permutation of $C$. Let $\sigma$ be the automorphism of $\cM$ that permutes the direct summands indexed by $C$ according to $\sigma$ while leaving the other direct summands alone (which is trace-preserving because $\beta_j$ is constant on $C$).  Because $k(i,j)$ is constant over $j \in C$ for each $i = 1, \dots, I$, the automorphism $\Theta$ fixes $\cA$ pointwise.  Therefore, $y$ must also be fixed by $\Theta$, meaning that $y_j$ is constant for $j \in C$.
		
		Thus, $p_C y = \bigoplus_{j \in \C} z$ for some $z \in M_n(\C)$ where $n$ is the value of $n_j$ for $j \in \C$.  Now since $k(i,j)$ is constant for $j \in C$, every $a \in \cA$ satisfies $p_C \iota(a) = \iota_C(a) \oplus \dots \iota_C(a)$ where $\iota_C$ is the projection of $\iota$ onto one of the $M_{n_j}(\C)$'s for $j \in C$.  Thus, if $u$ is a unitary in $\iota_C(\cA)' \cap M_n(\C)$, then there is automorphism $\Theta$ of $\cM$ that fixes $\cA$ pointwise such that $\Theta$ is conjugation by $u$ on $M_{n_j}(\C)$ when $j \in \C$ and the identity on $M_{n_j}(\C)$ for $j \not \in C$.  This automorphism must fix $p_C y$, and hence $uzu^* = z$.  By the bicommutant theorem on $M_n(\C)$, we conclude that $z \in \iota_C(\cA)$.  Thus, $p_Cy = p_C\iota(a) = p_Ca$ for some $a \in \cA$.
	\end{example}
	
	The next example relates definable closure with normalizers for inclusions of irreducible subfactors $\cA \subseteq \cM$ that have weak spectral gap, meaning that the only sequences in $\cM$ that asymptotically commute with $\cA$ are the ones that are asymptotically trivial.  Weak spectral gap is an important property in the study of $\mathrm{II}_1$ factors that was explicitly formulated in \cite{Popa2012}, although it was used implicitly earlier through its connection with other rigidity phenomena such as property (T).  Recently, Goldbring made several connections between spectral gap and definability of sets \cite{Goldbring2023spectralgap} which are thematically related to our Proposition \ref{prop: spectral gap normalizer} below.  We recall the following characterizations of weak spectral gap; this is in the same spirit as the ultraproduct characterization (\cite[Remark 2.2]{Popa2012}) of weak spectral gap and the treatment of property Gamma in \cite[\S 3.2.2]{FHS2014b}.  For the reader's convenience we give a self-contained proof, which is also a good illustration of definition of $\kappa$-saturation.
	
	\begin{lemma} \label{lem: spectral gap}
		Let $\cA \subseteq \cM$ be an inclusion of tracial von Neumann algebras, and let $\kappa > \chi(\cA)$.  The following are equivalent:
		\begin{enumerate}[(1)]
			\item For every $\epsilon > 0$ and $r > 0$, there exist a finite set $F \subseteq A$ and $\delta > 0$ such that, for $x \in D_r^{\cM}$, if $\max_{a \in F} d^{\cM}(xa,ax) < \delta$, then $d^{\cM}(x, \tr^{\cM}(x)1) < \epsilon$.
			\item For every elementary extension $\cN \succeq \cM$, we have $\cA' \cap \cN = \C 1$.
			\item There exists a $\kappa$-saturated elementary extension $\cN \succeq \cM$ such that $\cA' \cap \cN = \C 1$.
		\end{enumerate}
	\end{lemma}
	
	\begin{proof}
		(1) $\implies$ (2).  Let $\cN$ be an elementary extension of $\cM$, and let $y \in \cA' \cap \cN$, and let $r = \norm{y}$.  Fix $\epsilon > 0$.  Let $F \subseteq A$ and $\delta > 0$ be given by condition (1).  Then
		\[
		\sup_{x \in D_r^{\cM}} \min\left( \delta - \max_{a \in F} d^{\cM}(xa,ax), d^{\cM}(x, \tr^{\cM}(x)1) - \epsilon \right) \leq 0.
		\]
		Because $\cN$ is an elementary extension, the same equation holds when we replace $\cM$ by $\cN$, and in particular
		\[
		\min\left( \delta - \max_{a \in F} d^{\cN}(ya,ay), d^{\cM}(y, \tr^{\cN}(y)1) - \epsilon \right) \leq 0.
		\]
		Since $d^{\cN}(ya,ay) = 0$, the first option in the minimum is positive, so we obtain $d^{\cM}(y, \tr^{\cN}(y)1) \leq \epsilon$.  Since $\epsilon$ was arbitrary, $y \in \C$.  Thus, $\cA' \cap \cN = \C 1$.
		
		(2) $\implies$ (3) is immediate since a $\kappa$-saturated elementary extension exists by Theorem \ref{thm: existence of nice model}.
		
		(3) $\implies$ (1).  We proceed by contrapositive.  Suppose (1) fails and we will show that for every $\kappa$-saturated elementary extension $\cN$, we have $\cA' \cap \cN \neq \C 1$.  Let $A$ be a dense subset of $\cA$ with cardinality less than $\kappa$.  For $a \in A$, consider the $\cL_A$-formula
		\[
		\phi_a(x) = d(xa,ax).
		\]
		Since (1) fails, there are some $r > 0$ and $\epsilon > 0$ such that for every finite set $F \subseteq A$ and $\delta > 0$, there exists $x \in \cM$ such that $\max_{a \in F} \phi_a^{\cM_A}(x) < \delta$ and $d(x,\tr^{\cM}(x)1) \geq \epsilon$.  Note $\phi_a^{\cM_A}(x) = \phi_a^{\cN_A}(x)$.  This means that
		\[
		\{\phi_a: a \in A\} \cup \{ \epsilon \dot{-} d(x,\tr(x)1) \}
		\]
		is finitely approximately satisfiable in $\cN$.  By $\kappa$-saturation, it is satisfiable in $\cN$, so there exists $y \in \cN$ such that $d^{\cN}(ya,ay) = 0$ for $a \in A$ and also $d^{\cN}(y,\tr^{\cN}(y)1) \geq \epsilon$.  Hence, $y \in \cA' \cap \cN$ so $\cA' \cap \cN \neq \C 1$.
	\end{proof}
	
	\begin{proposition}[Definable closure and normalizer] \label{prop: spectral gap normalizer}
		Let $\cA \subseteq \cM$ be an inclusion of tracial von Neumann algebras satisfying the equivalent conditions of Lemma \ref{lem: spectral gap}.  Then $\acl^{\cM}(\cA)$ contains the normalizer $\operatorname{Norm}_{\cM}(\cA) = \{u \in U(\cM): u \cA u^* = \cA\}$, and $\dcl^{\cM}(\cA)$ contains the commutator subgroup of $\operatorname{Norm}_{\cM}(\cA)$.
	\end{proposition}
	
	\begin{proof}
		Let $\cN$ be a $\kappa$-saturated elementary extension of $\cM$, where $\kappa$ is greater than the density character of $\cA$.  Let $u \in \operatorname{Norm}_{\cM}(\cA)$.  We claim that $u$ is the unique realization of $\operatorname{tp}^{\cM}(u / \cA)$ in $\cN$ up to multiplication by a complex number of modulus $1$.  Suppose that $v \in \cM^{\cU}$ has the same type over $\cA$ as $u$.  Note that $v$ must be unitary.  For each $a \in \cA$, the element $b = uau^*$ is also in $\cA$.  Since $d(b,uau^*) = 0$, we also have $d(b,vav^*) = 0$.  Therefore, $uau^* = vav^*$ for all $a \in \cA$.  Hence, $uv^* \in \cA' \cap \cN$, which equals $\C 1$ by Lemma \ref{lem: spectral gap}. Thus, $v = \lambda u$ for some complex number $\lambda$ of modulus $1$.  Thus, in particular, the space of realizations of $\tp^{\cM}(u)$ in $\cN$ is $\{\lambda u: \lambda \in \C, |\lambda| = 1\}$.  This is compact, and hence $u \in \acl(\cA)$.
		
		For the second claim, let $\cN$ be a strongly $\kappa$-homogeneous elementary extension of $\cM$.  Let $\Theta$ be an automorphism of $\cN$ that fixes $\cA$ pointwise.  By the foregoing argument, for each $u \in \operatorname{Norm}_{\cM}(\cA)$, we have $\Theta(u) = \lambda(u) u$ for some scalar $\lambda(u)$.  Now $\lambda(uv) uv = \Theta(uv) = \Theta(u) \Theta(v) = \lambda(u) \lambda(v) uv$, hence $\lambda$ is a group homomorphism $\operatorname{Norm}_{\cM}(\cA) \to S^1$.  Therefore, $\lambda$ must vanish on the commutator subgroup of $\operatorname{Norm}_{\cM}(\cA)$.  Hence, every automorphism $\Theta$ fixes the commutator subgroup of $\operatorname{Norm}_{\cM}(\cA)$ pointwise, so the commutator subgroup is contained in $\dcl^{\cM}(\cA)$.
	\end{proof}
	
	Many inclusions with weak spectral gap arise from property (T), an important rigidity property defined for groups by Kazhdan \cite{Kazhdan1967}, for $\mathrm{II}_1$ factors by Connes and Jones \cite{ConnesJones1985}, and for tracial von Neumann algebras by Popa \cite{Popa2006Betti}.  Indeed, if $\cA \subseteq \cM$ is an inclusion with $\cA' \cap \cM = \C$ and if $\cA$ has property (T), then $\cA \subseteq \cM$ automatically satisfies the spectral gap condition Lemma \ref{lem: spectral gap} (1); see e.g.\ \cite[\S 2.3]{Tan2022}.
	
	\begin{example}
		By the work of Chifan, Ioana, Osin, and Sun \cite[\S 6]{CIOS2023}, there exists II$_1$ factor $\cA$ with property (T) such that the outer automorphism group of $\cA$ contains the free group $\mathbb{F}_n$.  In particular, there exists an action $\alpha: \mathbb{F}_n \to \Aut(\cA)$ that is properly outer, which means that the crossed product $\cM = \cA \rtimes \mathbb{F}_n$ satisfies $\cA' \cap \cM = \C$ (see e.g.\ \cite[Proposition 1.4.4(i)]{JonesSunder1997}).  Because $\cA$ has property (T), $\cA$ automatically satisfies the spectral gap condition Lemma \ref{lem: spectral gap} (1).  Therefore, by Proposition \ref{prop: spectral gap normalizer}, $\acl^{\cM}(\cA)$ contains all the group elements in $\mathbb{F}_n$.  Since $\acl^{\cM}(\cA)$ is a von Neumann subalgebra, we have $\acl^{\cM}(\cA) = \cM$.  Furthermore, if $\cN$ is an elementary extension of $\cM$, then using Proposition \ref{prop: closure and elem ext}), we also have $\acl^{\cN}(\cA) = \cM$.
	\end{example}
	
	\begin{remark}
		The present author in \cite[\S 4.5]{JekelCoveringEntropy} studied the interaction between definable closures and the Jung-Hayes free entropy from \cite{Jung2007S1B,Hayes2018}.  The fact that taking the algebraic closure does not increase the entropy \cite[Corollary 4.21]{JekelCoveringEntropy} was reminiscent of the behavior of entropy under normalizers (and more generally the singular subspace) obtained earlier by Hayes \cite[Theorem 3.1]{Hayes2018}.  Proposition \ref{prop: spectral gap normalizer} gives an explicit connection between these two results by exhibiting some normalizers as examples of algebraic closure.  Of course, Hayes' original results on the singular subspace are far more general than normalizers of subfactors with spectral gap.  Nonetheless, this example may provide the first hint as to what the analog of Hayes' results should be in the context of complete types.
	\end{remark}
	
	\section{Monge-Kantorovich duality for types and definable predicates} \label{sec: duality}
	
	In this section, we study convex definable predicates and prove three of the main results: the non-commutative model-theoretic Monge-Kantorovich duality Theorem \ref{thm: MK duality}, the analog of Jensen's inequality Proposition \ref{prop: monotonicity under expectation}, and Theorem \ref{thm: displacement interpolation} on the displacement interpolation.
	
	\subsection{Proof of duality} \label{subsec: duality proof}
	
	As in \cite[\S 3]{GJNS2021}, a key point of the argument for Monge-Kantorovich duality is a non-commutative version of the Legendre transform.  As the first step, we prove a weaker version where the predicates $\phi$ and $\psi$ are not required to be convex and only satisfy \eqref{eq: admissibility} on the domain $D_{\mathbf{r}}$ rather than everywhere.
	
	\begin{proposition} \label{prop: pre MK duality}
		Fix a complete theory $\mathrm{T}$ of tracial von Neumann algebras (i.e. a consistent complete theory containing $\mathrm{T}_{\tr}$).  Let $\mathbf{r} = (r_1,\dots,r_n)$.  Let $\mu$ and $\nu$ be types in $\mathbb{S}_{\mathbf{r}}(\mathrm{T})$.  There exist $\mathrm{T}_{\tr}$-definable predicates $\phi_0$ and $\psi_0$ such that
		\begin{equation} \label{eq: admissibility 2}
			\phi_0^{\cM}(\mathbf{x}) + \psi_0^{\cM}(\mathbf{y}) \geq \re \ip{\mathbf{x},\mathbf{y}}_{L^2(\cM)^n} \text{ for } \mathbf{x}, \mathbf{y} \in D_{\mathbf{r}}^{\cM} \text{ for } \cM \models \mathrm{T}_{\tr},
		\end{equation}
		and such that equality is achieved when $(\mathbf{x},\mathbf{y})$ is an optimal coupling of $(\mu,\nu)$ in $\cM$.
	\end{proposition}
	
	Note that in this proposition (and in Theorem \ref{thm: MK duality})we want $\phi_0$ and $\psi_0$ to be definable relative to $\mathrm{T}_{\tr}$ rather than only $\mathrm{T}$.  In other words, $\phi_0$ and $\psi_0$ will be functions that make sense to apply to tuples from \emph{any} tracial von Neumann algebra, not only those that satisfy $\mathrm{T}$.  This is potentially useful for situations where you cannot restrict to models of a fixed complete theory; for instance, if $\cM$ is an ultraproduct of tracial von Neumann algebras $\cM_i$ that are not elementarily equivalent to each other, e.g.\ $\cM = \prod_{n \to \cU} M_n(\C)$.
	
	Since we work with definable predicates relative to $\mathrm{T}_{\tr}$, it is convenient to extend the definition of $C(\mu,\nu)$ for $\mu$ and $\nu$ in $\mathbb{S}_{\mathbf{r}}(\mathrm{T}_{\tr})$.  Write
	\begin{equation} \label{eq: C general definition}
		C(\mu,\nu) = \sup \{ \re \ip{\mathbf{x},\mathbf{y}}_{L^2(\cM)}: \cM \models \mathrm{T}_{\tr}; \mathbf{x}, \mathbf{y} \in D_{\mathbf{r}}; \tp^{\cM}(\mathbf{x}) = \mu; \tp^{\cM}(\mathbf{y}) = \nu\}.
	\end{equation}
	Note that if $\cM \models \mathrm{T}_{\tr}$ and $\mathbf{x} \in \cM^n$, then $\tp^{\cM}(\mathbf{x})$ uniquely determines $\Th(\cM)$; this is because every sentence $\phi$ can be viewed trivially as a formula in $n$ variables, and hence $\tp^{\cM}(\mathbf{x})$ specifies the value of $\phi^{\cM}$.  It follows that if $\mathrm{T} \neq \mathrm{T}'$ are two complete theories and $\mu \in \mathbb{S}_{\mathbf{r}}(\mathrm{T})$ and $\nu \in \mathbb{S}_{\mathbf{r}}(\mathrm{T}')$, then it is impossible to realize both $\mu$ and $\nu$ in the same $\cM$ and hence the supremum in \eqref{eq: C general definition} is $-\infty$.  On the other hand, if $\mu$ and $\nu$ are in $\mathbb{S}_{\mathbf{r}}(\mathrm{T})$, then the supremum \eqref{eq: C general definition} is over models of $\mathrm{T}$ and hence agrees with the definition of $C(\mu,\nu)$ given in the introduction.

	\begin{proof}[Proof of Proposition \ref{prop: pre MK duality}]
		
		First, since the space of types $\mathbb{S}_{\mathrm{r}}(\mathrm{T}_{\tr})$ is compact and metrizable (see \S \ref{subsec: definable predicate}), there exists a continuous function $f: \mathbb{S}_{\mathrm{r}}(\mathrm{T}_{\tr}) \to [0,1]$ such that $f(\lambda) = 0$ if and only if $\lambda = \mu$.  By Lemma \ref{lem: continuous function definable predicate}, there exists a $\mathrm{T}_{\tr}$-definable predicate $\eta$ such that whenever $\cM \models \mathrm{T}_{\tr}$ and $\mathbf{x} \in D_{\mathbf{r}}$, we have $\eta^{\cM}(\mathbf{x}) = f(\tp^{\cM}(\mathbf{x}))$, and hence $\eta^{\cM}(\mathbf{x}) \geq 0$ with equality if and only if $\tp^{\cM}(\mathbf{x}) = \mu$.  For $\epsilon > 0$, let $\theta_\epsilon$ be the definable predicate relative to $\mathrm{T}_{\tr}$ given by
		\[
		\theta_\epsilon^{\cM}(\mathbf{y}) = \sup_{\mathbf{x} \in D_{\mathbf{r}}} \left[ \re \ip{\mathbf{x},\mathbf{y}}_{L^2(\cM)^n} - \frac{1}{\epsilon} \eta^{\cM}(\mathbf{x}) \right] \text{ for } \mathbf{y} \in \cM^n \text{ for } \cM \models \mathrm{T}_{\tr}.
		\]
		(The $\theta_\epsilon$'s will be our first approximation to $\psi_0$.)  Now because $\eta^{\cM}(\mathbf{x})$ vanishes when $\tp^{\cM}(\mathbf{x}) = \mu$, we have
		\[
		\theta_\epsilon^{\cM}(\mathbf{y}) \geq \sup_{\mathbf{x}: \tp^{\cM}(\mathbf{x}) = \mu} \re \ip{\mathbf{x},\mathbf{y}}_{L^2(\cM)^n} = C(\mu, \tp^{\cM}(\mathbf{y})).
		\]
		Furthermore, as $\epsilon$ decreases, $-(1/\epsilon) \eta$ also decreases, and thus $\theta_\epsilon$ decreases.  Thus, $\lim_{\epsilon \searrow 0} \theta_\epsilon^{\cM}(\mathbf{y})$ exists in $[-\infty,\infty)$ pointwise.  We claim that $\lim_{\epsilon \searrow 0} \theta_\epsilon^{\cM}(\mathbf{y}) = C(\tp^{\cM}(\mathbf{y}),\mu)$ for $\mathbf{y} \in D_{\mathbf{r}}^{\cM}$.
		
		First, consider the case where $\Th(\cM) = \mathrm{T}$.  Fix $\mathbf{y} \in D_{\mathbf{r}}^{\cM}$.  Let $\cN$ be an $\aleph_1$-saturated elementary extension of $\cM$.  Then by saturation there exists $\mathbf{x}_{\epsilon}$ that achieves the supremum in the definition of $\theta_\epsilon$, so that
		\[
		\theta_\epsilon^{\cM}(\mathbf{y}) = \theta_\epsilon^{\cN}(\mathbf{y}) = \re \ip{\mathbf{x}_{\epsilon},\mathbf{y}}_{L^2(\cN)^n} - \frac{1}{\epsilon} \eta^{\cN}(\mathbf{x}_\epsilon).
		\]
		Note that
		\begin{align*}
			\frac{1}{\epsilon} \eta^{\cN}(\mathbf{x}_\epsilon) &= \re \ip{\mathbf{x}_\epsilon,\mathbf{y}}_{L^2(\cN)^n} -  \theta_{\epsilon}^{\cN}(\mathbf{y}) \\
			&\leq |\mathbf{r}|^2 - C(\tp^{\cN}(\mathbf{y}),\mu) \\
			&\leq 2 |\mathbf{r}|^2,
		\end{align*}
		hence,
		\[
		\eta^{\cN}(\mathbf{x}_\epsilon) \leq 2 \epsilon |\mathbf{r}|^2 \to 0.
		\]
		By compactness of the space of types of $2n$-tuples, there is a sequence $\epsilon_k$ decreasing to $0$ such that $\lim_{k \to \infty} \tp^{\cN}(\mathbf{x}_{\epsilon_k},\mathbf{y})$ exists.  Again by $\aleph_1$-saturation, there exists some $\mathbf{x} \in D_{\mathbf{r}}^{\cN}$ such that $\lim_{k \to \infty} \tp^{\cN}(\mathbf{x}_{\epsilon_k},\mathbf{y}) = \tp^{\cN}(\mathbf{x},\mathbf{y})$.  Hence, $\eta^{\cN}(\mathbf{x}) = \lim_{k \to \infty} \eta^{\cN}(\mathbf{x}_{\epsilon_k}) = 0$, and so $\tp^{\cN}(\mathbf{x}) = \mu$.  Then
		\[
		\theta_{\epsilon_k}^{\cN}(\mathbf{y}) \leq \re \ip{\mathbf{x}_{\epsilon_k},\mathbf{y}}_{L^2(\cN)^n} \to \re \ip{\mathbf{x},\mathbf{y}}_{L^2(\cN)^n} \leq C(\tp^{\cN}(\mathbf{y}),\mu).
		\]
		Since we already know $\theta_\epsilon^{\cN}(\mathbf{y}) \geq C(\tp^{\cN}(\mathbf{y}),\mu)$ and it is monotone in $\epsilon$, this proves that $\lim_{\epsilon \searrow 0} \theta_\epsilon^{\cN}(\mathbf{y}) = C(\tp^{\cN}(\mathbf{y}),\mu) = C(\tp^{\cM}(\mathbf{y}),\mu)$.
		
		On the other hand, suppose that $\mathrm{T}' := \Th(\cM) \neq \mathrm{T}$.  Since $\mathbb{S}_{\mathbf{r}}(\mathrm{T}')$ is a compact subset of $\mathbb{S}_{\mathbf{r}}(\mathrm{T}_{\tr})$, the definable predicate $\eta$ achieves a minimum $\delta$ on it, and $\delta > 0$ since $\mu \not \in \mathbb{S}_{\mathbf{r}}(\mathrm{T}')$.  Then
		\begin{align*}
			\theta_{\epsilon}^{\cM}(\mathbf{x}) &= \sup_{\mathbf{x} \in D_{\mathbf{r}}} \left[ \re \ip{\mathbf{x},\mathbf{y}}_{L^2(\cM)^n} - \frac{1}{\epsilon} \eta^{\cM}(\mathbf{x}) \right] \\
			&\leq 2 |\mathbf{r}|^2 - \frac{\delta}{\epsilon},
		\end{align*}
		which tends to $-\infty$ as $\epsilon \to 0$.
		
		Thus, $\theta_\epsilon^{\cM}(\mathbf{y})$ converges pointwise to $C(\tp^{\cM}(\mathbf{y}),\mu)$, but there is no guarantee that it converges uniformly (for instance because the limit is sometimes infinite).  Therefore, we will execute a ``forced limit'' construction where we modify $\theta_\epsilon$ outside a neighborhood of the type $\nu$ that we are interested in to obtain uniform convergence; the limit will be our function $\psi_0$. Fix some $\cM_0 \models \mathrm{T}$ and some $\mathbf{y}_0 \in \cM_0^n$ with type $\nu$. Fix a sequence $(\epsilon_k)_{k \in \N}$ decreasing to zero such that
		\[
		\theta_{\epsilon_k}^{\cM}(\mathbf{y}_0) - \theta_{\epsilon_{k+1}}^{\cM}(\mathbf{y}_0) < \frac{1}{2^k}
		\]
		Let $g: [0,\infty) \to [0,1]$ be continuous and decreasing with $g = 1$ on $[0,1]$ and $g = 0$ on $[2,\infty]$; then let $g_k(t) = g(2^k t)$.  Let
		\[
		\chi_k(\mathbf{y}) = g_k(\theta_{\epsilon_k} - \theta_{\epsilon_{k+1}}).
		\]
		Thus, $\chi_k$ is a definable predicate, $0 \leq \chi_k \leq 1$, and $\chi_k^{\cM_0}(\mathbf{y}_0) = 1$.  Moreover, by construction, for any $\cM \models \mathrm{T}_{\tr}$ and $\mathbf{y} \in \cM^n$, we have $\chi_k^{\cM}(\mathbf{y}) = 0$ unless $\theta_{\epsilon_k}^{\cM}(\mathbf{y}) - \theta_{\epsilon_{k+1}}^{\cM}(\mathbf{y}) < 2/2^k$.  Hence,
		\[
		0 \leq \chi_k^{\cM}(\mathbf{y})(\theta_{\epsilon_k}^{\cM}(\mathbf{y}) - \theta_{\epsilon_{k+1}}^{\cM}(\mathbf{y})) \leq \frac{2}{2^k}.
		\]
		Let $\psi_0$ be the definable predicate
		\[
		\psi_0^{\cM}(\mathbf{y}) = \theta_{\epsilon_1}^{\cM}(\mathbf{y}) + \sum_{k=1}^\infty \chi_k^{\cM}(\mathbf{y})(\theta_{\epsilon_{k+1}}^{\cM}(\mathbf{y}) - \theta_{\epsilon_k}^{\cM}(\mathbf{y})).
		\]
		Note that the $K$th partial sum satisfies
		\[
		\theta_{\epsilon_1}^{\cM}(\mathbf{y}) + \sum_{k=1}^K \chi_k^{\cM}(\mathbf{y})(\theta_{\epsilon_{k+1}}^{\cM}(\mathbf{y}) - \theta_{\epsilon_k}^{\cM}(\mathbf{y})) \geq \theta_{\epsilon_1}^{\cM}(\mathbf{y}) + \sum_{k=1}^K (\theta_{\epsilon_{k+1}}^{\cM}(\mathbf{y}) - \theta_{\epsilon_k}^{\cM}(\mathbf{y})) = \theta_{\epsilon_{K+1}}^{\cM}(\mathbf{y}) \geq C(\tp^{\cM}(\mathbf{y}),\mu).
		\]
		Hence, for all $\mathbf{y}$, we have
		\[
		\psi_0^{\cM}(\mathbf{y}) \geq C(\tp^{\cM}(\mathbf{y}),\mu),
		\]
		and equality is achieved for $\mathbf{y}_0$ (hence for any $\mathbf{y}$ with type $\nu$) because
		\[
		\psi_0^{\cM_0}(\mathbf{y}_0) = \theta_{\epsilon_1}^{\cM_0}(\mathbf{y}_0) + \sum_{k=1}^\infty (\theta_{\epsilon_{k+1}}^{\cM_0}(\mathbf{y}_0) - \theta_{\epsilon_k}^{\cM}(\mathbf{y}_0)) = \lim_{k \to \infty} \theta_{\epsilon_{k+1}}^{\cM_0}(\mathbf{y}_0) = C(\mu,\nu).
		\]
		
		Having constructed $\psi_0$, we define
		\[
		\phi_0(\mathbf{x}) = \sup_{\mathbf{y} \in D_{\mathbf{r}}} \left[ \re \ip{\mathbf{x},\mathbf{y}} - \psi_0(\mathbf{y}) \right],
		\]
		which is a version of the \emph{Legendre transform} for definable predicates.
		By construction,
		\[
		\phi_0^{\cM}(\mathbf{x}) + \psi_0^{\cM}(\mathbf{y}) \geq \re \ip{\mathbf{x},\mathbf{y}}_{L^2(\cM)^n} \text{ for } \mathbf{x}, \mathbf{y} \in D_{\mathbf{r}}^{\cM} \text{ for } \cM \models \mathrm{T}_{\tr}.
		\]
		Moreover, when $\tp^{\cM}(\mathbf{x}) = \mu$, then since $\psi_0^{\cM}(\mathbf{y}) \geq C(\tp^{\cM}(\mathbf{y}),\mu)$, we have
		\[
		\phi_0^{\cM}(\mathbf{x}) \leq \sup_{\mathbf{y} \in D_{\mathbf{r}}^{\cM}} \re \ip{\mathbf{x},\mathbf{y}}_{L^2(\cM)^n} - C(\mu,\tp^{\cM}(\mathbf{y})) \leq 0.
		\]
		Thus, when $(\mathbf{x},\mathbf{y})$ is an optimal coupling of $(\mu,\nu)$ in $\cM$, then we have $\phi_0^{\cM}(\mathbf{x}) = 0$ and $\psi_0^{\cM}(\mathbf{y}) = C(\mu,\nu)$, so
		\[
		\phi_0^{\cM}(\mathbf{x}) + \psi_0^{\cM}(\mathbf{y}) \leq 0 + C(\mu,\nu) = \re \ip{\mathbf{x},\mathbf{y}}_{L^2(\cM)^n}.
		\]
		Thus, equality is achieved since we already showed $\phi_0^{\cM}(\mathbf{x}) + \psi_0^{\cM}(\mathbf{y}) \geq \re \ip{\mathbf{x},\mathbf{y}}_{L^2(\cM)^n}$ in general.
	\end{proof}
	
	\begin{proof}[Proof of Theorem {\ref{thm: MK duality}}]  Fix a complete theory $\mathrm{T} \models \mathrm{T}_{\tr}$ and types $\mu, \nu \in \mathbb{S}(\mathrm{T})$.  Fix $\mathbf{r} \in (0,\infty)^n$ sufficiently large that $\mu, \nu \in \mathbb{S}_{\mathbf{r}}(\mathrm{T})$.
		
		By Proposition \ref{prop: pre MK duality}, there exists some pair of $\mathrm{T}_{\tr}$-definable predicates $\phi_0$ and $\psi_0$ satisfying \eqref{eq: admissibility 2} with equality when $(\mathbf{x},\mathbf{y})$ is an optimal coupling of $(\mu,\nu)$.  It remains to arrange a pair $(\phi,\psi)$ of $\mathrm{T}_{\tr}$-definable predicates such that
		\begin{enumerate}[(1)]
			\item $\phi$ and $\psi$ are convex;
			\item $\phi^{\cM}(\mathbf{x}) + \psi^{\cM}(\mathbf{y}) \geq \re \ip{\mathbf{x},\mathbf{y}}_{L^2(\cM)^n}$ for all $\mathbf{x}$ and $\mathbf{y}$ in $\cM^n$, not only those in $D_{\mathbf{r}}$, or in other words, \eqref{eq: admissibility} holds globally.
		\end{enumerate}
		We first arrange (1) and then arrange (2).
		
		To arrange (1), recall how in the proof of Proposition \ref{prop: pre MK duality}, we took
		\[
		\phi_0^{\cM}(\mathbf{x}) = \sup_{\mathbf{y} \in D_{\mathbf{r}}^{\cM}} \left[ \re \ip{\mathbf{x},\mathbf{y}}_{L^2(\cM)^n} - \psi_0^{\cM}(\mathbf{y}) \right]
		\]
		for $\cM \models \mathrm{T}_{\tr}$ and $\mathbf{x} \in D_{\mathbf{r}}^{\cM}$.  Note that $\phi_0^{\cM}$ is convex it is a supremum of affine functions of $\mathbf{x}$.  Now let
		\[
		\psi_1^{\cM}(\mathbf{y}) = \sup_{\mathbf{x} \in D_{\mathbf{r}}^{\cM}} \left[ \re \ip{\mathbf{x},\mathbf{y}}_{L^2(\cM)^n} - \phi_0^{\cM}(\mathbf{x}) \right].
		\]
		Then $\psi_1$ is convex for the same reason that $\phi_0$ is.  Also, by construction, $(\phi_0,\psi_1)$ satisfies
		\begin{equation} \label{eq: admissibility again}
			\phi_0^{\cM}(\mathbf{x}) + \psi_1^{\cM}(\mathbf{y}) \geq \re \ip{\mathbf{x},\mathbf{y}}_{L^2(\cM)^n} \text{ for } \mathbf{x}, \mathbf{y} \in D_{\mathbf{r}}^{\cM}.
		\end{equation}
		Furthermore, $\psi_1 \leq \psi_0$ on $D_{\mathbf{r}}^{\cM}$ because by our assumption on $(\phi_0,\psi_0)$, for all $\mathbf{x}, \mathbf{y} \in D_{\mathbf{r}}^{\cM}$, we have
		\[
		\re \ip{\mathbf{x},\mathbf{y}}_{L^2(\cM)^n} - \phi_0^{\cM}(\mathbf{x}) \leq \psi_0^{\cM}(\mathbf{y}),
		\]
		and taking the supremum over $\mathbf{x} \in D_{\mathbf{r}}^{\cM}$ yields $\psi_1^{\cM}(\mathbf{y}) \leq \psi_0^{\cM}(\mathbf{y})$.   Hence, if $(\mathbf{x},\mathbf{y})$ is an optimal coupling of $(\mu,\nu)$, then
		\[
		\re \ip{\mathbf{x},\mathbf{y}}_{L^2(\cM)^n} \leq \phi_0^{\cM}(\mathbf{x}) + \psi_1^{\cM}(\mathbf{y}) \leq \phi_0^{\cM}(\mathbf{x}) + \psi_0^{\cM}(\mathbf{y}) = \re \ip{\mathbf{x},\mathbf{y}}_{L^2(\cM)^n},
		\]
		and thus equality is achieved in \eqref{eq: admissibility again}.  Thus, we have arranged property (1), convexity.
		
		To arrange (2) that \eqref{eq: admissibility} holds globally, we modify $\phi_0$ and $\psi_1$ appropriately outside of $D_{\mathbf{r}}$.  Let $\delta$ be the formula
		\[
		\delta^{\cM}(\mathbf{x}) = d(\mathbf{x},D_{\mathbf{r}}^{\cM}) = \inf_{\widehat{\mathbf{x}} \in D_{\mathbf{r}}^{\cM}} d^{\cM}(\mathbf{x},\widehat{\mathbf{x}}).
		\]
		We record several observations about $\delta$ for later use:
		\begin{enumerate}[(a)]
			\item We have $\delta^{\cM}(\mathbf{x}) = 0$ for $\mathbf{x} \in D_{\mathbf{r}}^{\cM}$.
			\item The infimum in the definition of $\delta$ is achieved; indeed, for each $\mathbf{x}$, there is closest point $\widehat{\mathbf{x}}$ in $D_{\mathbf{r}}^{\cM}$ because $D_{\mathbf{r}}^{\cM}$ is a closed convex set in the Hilbert space $L^2(\cM)^n$.
			\item Note that $\delta$ is convex because given $\mathbf{x}, \mathbf{y} \in \cM^n$ and $\lambda \in [0,1]$, we have
			\begin{align*}
				\delta^{\cM}((1-\lambda)\mathbf{x} + \lambda \mathbf{y}) &\leq d^{\cM}((1-\lambda)\mathbf{x} + \lambda \mathbf{y}, (1-\lambda)\mathbf{x} + \lambda \mathbf{y}) \\
				&\leq (1 - \lambda) d^{\cM}(\mathbf{x},\widehat{\mathbf{x}}) + \lambda d^{\cM}(\mathbf{y},\widehat{\mathbf{y}}) \\
				&= (1 - \lambda) \delta^{\cM}(\mathbf{x}) + \lambda \delta^{\cM}(\mathbf{y}).
			\end{align*}
			\item Finally, $\delta^2$ is convex, because the square of any nonnegative convex function is convex.  (We leave the elementary proof of this fact as an exercise.)
		\end{enumerate}
		Now let $|\mathbf{r}| = (r_1^2 + \dots + r_n^2)^{1/2}$, and let
		\begin{align*}
			\phi_2^{\cM}(\mathbf{x}) &= \phi_0^{\cM}(\mathbf{x}) + \frac{1}{2} \delta^{\cM}(\mathbf{x})^2 + 2 |\mathbf{r}| \delta^{\cM}(\mathbf{x})\\
			\psi_2^{\cM}(\mathbf{y}) &= \psi_1^{\cM}(\mathbf{y}) + \frac{1}{2} \delta^{\cM}(\mathbf{y})^2.
		\end{align*}
		Then $\phi_2$ and $\psi_2$ are convex because $\phi_0$, $\psi_1$, $\delta$, and $\delta^2$ are convex.  Furthermore, $\phi_2$ and $\psi_2$ agree on $D_{\mathbf{r}}^{\cM}$ with $\phi_0$ and $\psi_1$ respectively.  To show \eqref{eq: admissibility}, we first note by the construction of $\psi_1$ that
		\[
		\psi_1^{\cM}(\mathbf{y}) \geq \re \ip{\widehat{\mathbf{x}},\mathbf{y}}_{L^2(\cM)^n} - \phi_0^{\cM}(\widehat{\mathbf{x}}),
		\]
		since $\widehat{\mathbf{x}} \in D_{\mathbf{r}}^{\cM}$.  Thus,
		\begin{equation} \label{eq: two terms}
			\phi_0^{\cM}(\mathbf{x}) + \psi_1^{\cM}(\mathbf{y}) \geq \re \ip{\widehat{\mathbf{x}},\mathbf{y}}_{L^2(\cM)^n} + [\phi_0^{\cM}(\mathbf{x}) - \phi_0^{\cM}(\widehat{\mathbf{x}})].
		\end{equation}
		Recall that $\phi_0^{\cM}(\mathbf{x})$ is the supremum over $\mathbf{z} \in D_{\mathbf{r}}^{\cM}$ of $\re \ip{\mathbf{x},\mathbf{z}}_{L^2(\cM)^n} - \psi_0^{\cM}(\mathbf{z})$.  Since $\norm{\mathbf{z}}_{L^2(\cM)^n} \leq |\mathbf{r}|$, each function $\re \ip{\mathbf{x},\mathbf{z}}_{L^2(\cM)^n} - \psi_0^{\cM}(\mathbf{z})$ is $|\mathbf{r}|$-Lipschitz in $\mathbf{x}$.  Hence, $\phi_0^{\cM}$ is $|\mathbf{r}|$-Lipschitz in $\mathbf{x}$ as well.  Thus,
		\begin{equation} \label{eq: second term}
			\phi_0^{\cM}(\mathbf{x}) - \phi_0^{\cM}(\widehat{\mathbf{x}}) \geq -|\mathbf{r}| \norm{\mathbf{x} - \widehat{\mathbf{x}}}_{L^2(\cM)} = -|\mathbf{r}| \delta^{\cM}(\mathbf{x}).
		\end{equation}
		Meanwhile,
		\[
		\re \ip{\widehat{\mathbf{x}},\mathbf{y}}_{L^2(\cM)^n} = \re \ip{\mathbf{x},\mathbf{y}}_{L^2(\cM)^n} - \re \ip{\mathbf{x} - \widehat{\mathbf{x}},\mathbf{y}}_{L^2(\cM)^n},
		\]
		and the ``error term'' (the second term on the right) can be estimated by 	
		\begin{align*}
			\re \ip{\mathbf{x} - \widehat{\mathbf{x}},\mathbf{y}}_{L^2(\cM)^n} &= \re \ip{\mathbf{x} - \widehat{\mathbf{x}},\widehat{\mathbf{y}}}_{L^2(\cM)^n} + \re \ip{\mathbf{x} - \widehat{\mathbf{x}},\mathbf{y} - \widehat{\mathbf{y}}}_{L^2(\cM)^n} \\
			&\leq \norm{\widehat{\mathbf{y}}}_{L^2(\cM)^n} \norm{\mathbf{x} - \widehat{\mathbf{x}}}_{L^2(\cM)^n} + \frac{1}{2} \norm{\mathbf{x} - \widehat{\mathbf{x}}}_{L^2(\cM)^n}^2 + \frac{1}{2} \norm{\mathbf{y} - \widehat{\mathbf{y}}}_{L^2(\cM)^n}^2 \\
			&\leq |\mathbf{r}| \delta^{\cM}(\mathbf{x}) + \frac{1}{2} \delta^{\cM}(\mathbf{x})^2 + \frac{1}{2} \delta^{\cM}(\mathbf{y})^2,
		\end{align*}
		so that
		\begin{equation} \label{eq: first term}
			\re \ip{\widehat{\mathbf{x}},\mathbf{y}}_{L^2(\cM)^n} \geq \re \ip{\mathbf{x},\mathbf{y}}_{L^2(\cM)^n} - |\mathbf{r}| \delta^{\cM}(\mathbf{x}) - \frac{1}{2} \delta^{\cM}(\mathbf{x})^2 - \frac{1}{2} \delta^{\cM}(\mathbf{y})^2.
		\end{equation}
		Altogether, substituting \eqref{eq: first term} and \eqref{eq: second term} into \eqref{eq: two terms}, we get that
		\[
		\phi_0^{\cM}(\mathbf{x}) + \psi_1^{\cM}(\mathbf{y}) \geq \re \ip{\mathbf{x},\mathbf{y}}_{L^2(\cM)^n} - |\mathbf{r}| \delta^{\cM}(\mathbf{x}) - \frac{1}{2} \delta^{\cM}(\mathbf{x})^2 - \frac{1}{2} \delta^{\cM}(\mathbf{y})^2 - |\mathbf{r}| \delta^{\cM}(\mathbf{x}),
		\]
		which rearranges to
		\[
		\phi_0^{\cM}(\mathbf{x}) + \frac{1}{2} \delta^{\cM}(\mathbf{x})^2 + 2 |\mathbf{r}| \delta^{\cM}(\mathbf{x}) + \psi_1^{\cM}(\mathbf{y}) + \frac{1}{2} \delta^{\cM}(\mathbf{y}) \geq \re \ip{\mathbf{x},\mathbf{y}}_{L^2(\cM)^n},
		\]
		which means that $(\phi_2,\psi_2)$ satisfies \eqref{eq: admissibility}.  Moreover, equality is achieved when $(\mathbf{x},\mathbf{y})$ is an optimal coupling of $(\mu,\nu)$ because in that case $\mathbf{x}$ and $\mathbf{y}$ are in $D_{\mathbf{r}}^{\cM}$ and so $\phi_2^{\cM}(\mathbf{x})$ and $\psi_2^{\cM}(\mathbf{y})$ agree with $\phi_0^{\cM}(\mathbf{x})$ and $\psi_1^{\cM}(\mathbf{y})$, for which we already established the equality.
		
		It remains to prove the final claim of Theorem \ref{thm: MK duality}, namely that $C(\mu,\nu)$ is the infimum of $\mu[\phi] + \nu[\psi]$ over all pairs $(\phi,\psi)$ of convex definable predicates satisfying \eqref{eq: admissibility}.  If $(\phi,\psi)$ satisfy \eqref{eq: admissibility}, then whenever $\mathbf{x}$ has type $\mu$ and $\mathbf{y}$ has type $\nu$, we have
		\begin{equation} \label{eq: duality 2}
			\re \ip{\mathbf{x},\mathbf{y}}_{L^2(\cM)^n} \leq \phi^{\cM}(\mathbf{x}) + \psi^{\cM}(\mathbf{y}) = \mu[\phi] + \nu[\psi].
		\end{equation}
		Taking the supremum over couplings $(\mathbf{x},\mathbf{y})$ on the left-hand side, we get $C(\mu,\nu)$.  Since $(\phi,\psi)$ were arbitrary, we see that $C(\mu,\nu)$ is less than or equal to the infimum of $\mu[\phi] + \nu[\psi]$ for pairs $(\phi,\psi)$ of convex definable predicates satisfying \eqref{eq: admissibility}.  On the other hand, the preceding argument showed that there exists some $(\phi,\psi)$ that achieves equality in \eqref{eq: duality 2} when $(\mathbf{x},\mathbf{y})$ is an optimal coupling, and hence $C(\mu,\nu)$ is equal to the infimum of $\mu[\phi] + \nu[\psi]$ over such pairs $(\phi,\psi)$.
	\end{proof}
	
	\subsection{Properties of convex definable predicates} \label{subsec: convex}
	
	An important notion in convex analysis and optimal transport theory is the subgradient of a convex function $\phi$, which captures the existence of supporting hyperplanes to the graph of $\phi$.  We will also later use the subgradient and supergradient for functions that are not necessarily convex.
	
	\begin{definition}
		Let $H$ be a Hilbert space.  Let $\phi: H \to (-\infty,\infty]$.  For $x, y \in H$, we say that $y \in \underline{\nabla} \phi(x)$ if
		\[
		\phi(x') - \phi(x) \geq \re \ip{x' - x, y} + o(\norm{x' - x}) \text{ for } x' \in H.
		\]
		Similarly, we say $y \in \overline{\nabla} \phi(x)$ if
		\[
		\phi(x') - \phi(x) \leq \re \ip{x' - x, y} + o(\norm{x' - x}) \text{ for } x' \in H.
		\]
	\end{definition}
	
	\begin{fact} \label{fact: convex subgradient}
		If $\phi$ is convex and $y \in \underline{\nabla} \phi(x)$, then we have
		\[
		\phi(x') - \phi(x) \geq \re \ip{x' - x, y} \text{ for } x' \in H.
		\]
	\end{fact}
	
	\begin{proof}
		For $\lambda \in [0,1]$,
		\[
		\lambda \phi(x') + (1 - \lambda) \phi(x) \geq \phi(\lambda x' + (1 - \lambda) x) \geq \phi(x) + \re \ip{\lambda x' + (1 - \lambda)x - x,y} + o(\lambda \norm{x' - x}),
		\]
		and hence
		\[
		\lambda(\phi(x') - \phi(x)) \geq \lambda \re \ip{x'-x,y} + o(\lambda \norm{x' - x}).
		\]
		Divide by $\lambda$ and take the limit as $\lambda  \searrow 0$.
	\end{proof}
	
	The following is a fundamental result in convex analysis; see e.g.\ \cite[Theorem 9.23, Proposition 16.4]{BC2017}.
	
	\begin{fact} \label{fact: convex subgradient 2}
		If $H$ is a Hilbert space and $\phi: H \to (-\infty,\infty]$ is convex and lower semicontinuous, $\underline{\nabla} \phi(x)$ is a nonempty, closed, convex set for each $x \in H$.
	\end{fact}

	\begin{proposition} \label{prop: definable subgradient}
		Fix a consistent theory $\mathrm{T}$ containing $\mathrm{T}_{\tr}$.  Let $\phi(x_1,\dots,x_n)$ be a convex definable predicate relative to $\mathrm{T}$.  Fix $\cM \models \mathrm{T}$ and $\mathbf{x} \in \cM^n$.  Then there exists $\mathbf{y} \in L^2(\dcl^{\cM}(\mathbf{x}))^n$ (here $L^2(\dcl^{\cM}(\mathbf{x}))$ is the closure in $L^2(\cM)$ of $\dcl^{\cM}(\mathbf{x})$), such that
		\begin{equation} \label{eq: predicate subgradient}
			\phi^{\cM}(\mathbf{x'}) - \phi^{\cM}(\mathbf{x}) \geq \re \ip{\mathbf{x}' - \mathbf{x},\mathbf{y}}_{L^2(\cM)} \text{ for } \mathbf{x}' \in \cM^n.
		\end{equation}
	\end{proposition}
	
	\begin{proof}
		Let $\kappa > \chi(\cM)$, and let $\cN$ be a strongly $\kappa$-homogeneous extension of $\cM$.
		
		Fix $\mathbf{r}$ such that $\mathbf{x} \in D_{\mathbf{r}}^{\cM}$.  Let $\psi: L^2(\cN)^n \to (-\infty,+\infty]$ be given by
		\[
		\psi(\mathbf{x}') = \begin{cases} \phi^{\cN}(\mathbf{x}'), & \mathbf{x}' \in D_{2\mathbf{r}}^{\cN} \\ +\infty, & \text{otherwise.} \end{cases}
		\]	
		Note $\psi$ is continuous on $D_{2 \mathbf{r}}$ by continuity of formulas.  Some elementary casework then shows that $\psi$ is convex and lower-semicontinuous globally.  Therefore, by Fact \ref{fact: convex subgradient 2}, $\underline{\nabla} \psi(\mathbf{x})$ is a nonempty closed convex set in $L^2(\cM)^n$.  As such, $\underline{\nabla} \psi(\mathbf{x})$ has a unique element of minimal $L^2$-norm, call it $\mathbf{y}$.
		
		If $\Theta$ is an automorphism of $\cN$ that fixes $\mathbf{x}$, then $\phi^{\cN} \circ \Theta = \phi^{\cN}$ and hence $\psi \circ \Theta = \psi$.  Hence, $\underline{\nabla} \psi(\mathbf{x})$ is also invariant under (the extension to $L^2(\cN)$ of) $\Theta$.  Therefore, $\Theta(\mathbf{y}) = \mathbf{y}$.
		
		In order to say that $\mathbf{y}$ is in $L^2$ of $\dcl^{\cN}(A)$, we need to reckon with the fact that $(y_1,\dots,y_n)$ are not elements of $\cN$, but rather of $L^2(\cN)$. We recall that the elements of $L^2(\cN)$ can be identified with certain unbounded operators on $L^2(\cN)$ affiliated with $\cN$; see \cite[p.\ 482-483]{BrownOzawa2008}. Write $\mathbf{y} = (y_1,\dots,y_n)$ and let $y_j = a_j + ib_j$ where $a_j$ and $b_j$ are self-adjoint elements of $L^2(\cN)$ (see Fact \ref{fact: J}).  In the following, we denote the extension of $\Theta$ to $L^2(\cN)$ by $\tilde{\Theta}$ for clarity.  Because $\Theta$ respects the adjoint, $\tilde{\Theta}$ commutes with the adjoint operator $J$ of Fact \ref{fact: J}.  Therefore, $a_j$ and $b_j$ are fixed by $\Theta$ as well.  This means that, viewed as operators on $L^2(\cN)$, $a_j$ and $b_j$ commute with $\tilde{\Theta}$.  Since $a_j$ and $b_j$ are affiliated operators, we know that for $f \in C_0(\R)$, $f(a_j)$ and $f(b_j)$ are elements of $\cN$.  Moreover, as operators on $L^2(\cN)$, $f(a_j)$ and $f(b_j)$ commute with $\tilde{\Theta}$.  This means that $\Theta(f(a_j)) = f(a_j)$ and $\Theta(f(b_j)) = f(b_j)$ holds for all $\Theta$ and hence $f(a_j)$ and $f(b_j)$ are in $\dcl^{\cN}(\mathbf{x})$ for $f \in C_0(\R)$.  By spectral theory for unbounded self-adjoint operators, since $1 \in L^2(\cN)$ is in the domain of $a_j$ and $b_j$, there exist functions $f_k \in C_0(\R)$ for $k \in \N$ such that $f_k(a_j)1 \to a_j 1$ in $L^2(\cN)$, and similarly for $b_j$.  Hence, $a_j$ and $b_j$ are limits in $L^2(\cN)$ of $f_k(a_j)$ and $f_k(b_j)$, which are in $\dcl^{\cN}(\mathbf{x})$.  Recall also that $\dcl^{\cN}(\mathbf{x}) = \dcl^{\cM}(\mathbf{x})$ by Proposition \ref{prop: closure and elem ext}.  Hence, $a_j$ and $b_j$ are in $L^2(\dcl^{\cM}(\mathbf{x}))$, and hence $\mathbf{y} \in L^2(\dcl^{\cM}(\mathbf{x}))^n$.
		
		It remains to show \eqref{eq: predicate subgradient}.  By our choice of $\psi$, we already know $\phi^{\cM}(\mathbf{x}') - \phi^{\cM}(\mathbf{x}) \geq \re \ip{\mathbf{x}' - \mathbf{x}, \mathbf{y}}_{L^2(\cM)}$ for $\mathbf{x}' \in D_{2\mathbf{r}}^{\cM}$.  If $\mathbf{x}'$ is arbitrary, then $\lambda \mathbf{x}' + (1 - \lambda) \mathbf{x}$ is in $D_{2 \mathbf{r}}$ for sufficiently small $\lambda > 0$, and
		\[
		\lambda \phi^{\cM}(\mathbf{x}') + (1 - \lambda) \phi^{\cM}(\mathbf{x}) \geq \phi^{\cM}(t\mathbf{x}' + (1 - \lambda) \mathbf{x}) + \re \ip{\lambda \mathbf{x}' +(1-\lambda) \mathbf{x} - \mathbf{x},\mathbf{y}}_{L^2(\cM)^n},
		\]
		hence, $\phi^{\cM}(\mathbf{x}') - \phi^{\cM}(\mathbf{x}) \geq \re \ip{\mathbf{x}' - \mathbf{x}, \mathbf{y}}_{L^2(\cM)^n}$.
	\end{proof}
	
	As a corollary, we obtain monotonicity of a convex definable predicate $\phi$ under conditional expectations onto elementary submodels (Proposition \ref{prop: monotonicity under expectation}).  Monotonicity under conditional expectations in general was a requirement in \cite[Definitions 1.2 and 3.6]{GJNS2021} for $E$-convexity.  The condition is motivated by Jensen's inequality for convex functions on $\R^n$ \cite[\S 1.4, Lemma 1.10]{GJNS2021}, as well as the behavior of unitarily invariant convex functions on $M_n(\C)$ \cite[\S 1.5, Lemma 1.17]{GJNS2021}.  Now let us show that the analogous condition for definable predicates is automatic.
	
	\begin{proof}
		Fix a consistent $\mathrm{T} \models \mathrm{T}_{\tr}$.  Let $\cM \preceq \cN$ be models of $\mathrm{T}$, and let $\mathbf{x} \in \cM^n$.  Let $\phi$ be a convex definable predicate relative to $\mathrm{T}$.  Let $\mathbf{x} = E_{\cM}[\mathbf{z}]$.  By the previous proposition, there exists $\mathbf{y} \in L^2(\dcl^{\cN}(\mathbf{x}))$ such that
		\[
		\phi^{\cN}(\mathbf{x}') - \phi^{\cN}(\mathbf{x}) \geq \re \ip{\mathbf{x}' - \mathbf{x},\mathbf{y}}_{L^2(\cN)^n} \text{ for } \mathbf{x}' \in \cN^n.
		\]
		Since $\dcl^{\cN}(\mathbf{x}) \subseteq \cM$ by Proposition \ref{prop: closure and elem ext}, we have $\mathbf{y} \in L^2(\cM)$, and thus,
		\[
		\phi^{\cN}(\mathbf{z}) - \phi^{\cN}(\mathbf{x}) \geq \re \ip{\mathbf{z} - \mathbf{x},\mathbf{y}}_{L^2(\cN)^n} = \re \ip{\mathbf{z} - E_{\cM}[\mathbf{z}], \mathbf{y}}_{L^2(\cN)^n} = 0.
		\]
		Hence, $\phi^{\cN}(\mathbf{z}) - \phi^{\cN}(E_{\cM}[\mathbf{z}]) \geq 0$ as desired.
	\end{proof}
	
	Our next goal is to prove Theorem \ref{thm: displacement interpolation}.  The main ingredient in the proof is that if $\phi$ is \emph{strongly convex} and $\mathbf{y} \in \underline{\nabla} \phi^{\cM}(\mathbf{x})$, then $\mathbf{x} \in \dcl^{\cM}(\mathbf{y})$.  Here we recall the definition and some basic facts about strong convexity.
	
	\begin{definition}
		Let $H$ be an inner product space and $c > 0$.  A function $\phi: H \to \R$ is \emph{$c$-strongly convex} if $\phi(x) - (c/2) \norm{x}^2$ is convex.
	\end{definition}
	
	The next observation is a characterization of strong convexity that follows directly from the definition by inner-product computations; see \cite[Corollary 2.15, Proposition 10.8]{BC2017}.
	
	\begin{fact}
		$c$-strong convexity of $\phi: H \to \R$ is equivalent to the condition that for $x, x' \in H$ and $\lambda \in (0,1)$,
		\[
		\phi((1 - \lambda) x + \lambda x') \leq (1 - \lambda) \phi(x) + \lambda \phi(x') - \frac{c}{2} \lambda(1 - \lambda) \norm{x' - x}^2.
		\]
	\end{fact}
	
	The following estimate shows that if $\phi$ is strongly convex, then ``$\nabla \phi$ expands distances.''
	
	\begin{fact} \label{fact: strongly convex bound}
		Let $H$ be an inner product space and let $\phi$ be $c$-strongly convex.  Let $x, x' \in H$.  Let $y \in \underline{\nabla} \phi(x)$ and $y' \in \underline{\nabla} \phi(x')$.  Then $\norm{x' - x} \leq (1/c) \norm{y' - y}$.
	\end{fact}
	
	\begin{proof}
		Observe that for $\lambda \in (0,1)$,
		\begin{align*}
			(1 - \lambda) \phi(x) + \lambda \phi(x') &\geq \phi((1 - \lambda) x + \lambda x') + \frac{c}{2} \lambda(1 - \lambda) \norm{x' - x}^2 \\
			&\geq \phi(x) + \re \ip{(1 - \lambda) x + \lambda x') - x,y} + \frac{c}{2} \lambda(1 - \lambda) \norm{x' - x}^2,
		\end{align*}
		hence
		\[
		\lambda (\phi(x') - \phi(x)) \geq \lambda \re \ip{x' - x,y} + \frac{c}{2} \lambda(1 - \lambda) \norm{x' - x}^2.
		\]
		Dividing by $\lambda$ and taking $\lambda \searrow 0$, we get
		\[
		\phi(x') - \phi(x) \geq \re \ip{x' - x,y} + \frac{c}{2} \norm{x' - x}^2.
		\]
		Symmetrically,
		\[
		\phi(x) - \phi(x') \geq \re \ip{x - x',y'} + \frac{c}{2} \norm{x' - x}^2.
		\]
		Hence,
		\[
		\re \ip{x' - x,y} + \frac{c}{2} \norm{x' - x}^2 \leq \phi(x') - \phi(x) \leq \re \ip{x'-x,y'} - \frac{c}{2} \norm{x' - x}^2.
		\]
		Rearranging yields
		\[
		c \norm{x' - x}^2 \leq \re \ip{x'-x,y'-y} \leq \norm{x' - x} \norm{y' - y}.
		\]
		Hence, $\norm{x' - x} \leq (1/c)\norm{y' - y}$.
	\end{proof}
	
	\begin{proposition} \label{prop: definable subgradient 2}
		Let $\phi(x_1,\dots,x_n)$ be a $c$-strongly convex definable predicate relative to a consistent theory $\mathrm{T} \models \mathrm{T}_{\tr}$.  Let $\cM \models \mathrm{T}$ and $\mathbf{x} \in \cM^n$, and let $\mathbf{y} \in L^2(\cM)^n$ be a subgradient vector as in Proposition \ref{prop: definable subgradient}.  Let $\mathrm{W}^*(\mathbf{y})$ be the smallest von Neumann subalgebra of $\cM$ whose $L^2$-closure contains $y_1, \dots, y_n$.  Then $\mathbf{x} \in \dcl^{\cM}(\mathrm{W}^*(\mathbf{y}))^n$.
	\end{proposition}
	
	\begin{proof}
		Let $\cN$ be an $\aleph_1$-saturated and strongly $\aleph_1$-homogeneous elementary extension of $\cM$, and let $\Theta$ be an automorphism of $\cN$ that fixes pointwise $\mathrm{W}^*(\mathbf{y})$ (hence also $\mathbf{y}$).  Note that $\mathbf{y} = \Theta(\mathbf{y}) \in \underline{\nabla} \phi(\Theta(\mathbf{x}))$.  Hence, by the previous fact, $\norm{\Theta(\mathbf{x}) - \mathbf{x}}_{L^2(\cM)^n} \leq (1/c) \norm{\mathbf{y} - \mathbf{y}}_{L^2(\cM)^n}$.  Hence, $\mathbf{x}$ must be fixed by $\Theta$.  Therefore, $\mathbf{x} \in \dcl^{\cN}(\mathrm{W}^*(\mathbf{y}))^n = \dcl^{\cM}(\mathrm{W}^*(\mathbf{y}))$ as desired.
	\end{proof}
	
	\begin{proof}[Proof of Theorem \ref{thm: displacement interpolation}]
		Let $(\mathbf{x},\mathbf{y})$ be an optimal coupling in $\cM$ of $(\mu,\nu)$ from $\mathbb{S}_n(\mathrm{T})$.  Let $\phi$ and $\psi$ be convex definable predicates as in Theorem \ref{thm: MK duality}, so that
		\[
		\phi^{\cM}(\mathbf{x}') + \psi^{\cM}(\mathbf{y}') \geq \re \ip{\mathbf{x}',\mathbf{y}'} \text{ for } \mathbf{x}', \mathbf{y}' \in \cM^n
		\]
		with equality at $(\mathbf{x},\mathbf{y})$.  Note that for $\mathbf{x}' \in \cM^n$,
		\[
		\phi^{\cM}(\mathbf{x}') - \phi^{\cM}(\mathbf{x}) \geq \left[ \re \ip{\mathbf{x}',\mathbf{y}}_{L^2(\cM)^n} - \psi^{\cM}(\mathbf{y}) \right] - \left[ \re \ip{\mathbf{x},\mathbf{y}}_{L^2(\cM)^n} - \psi^{\cM}(\mathbf{y}) \right] = \re \ip{\mathbf{x}'-\mathbf{x},\mathbf{y}}_{L^2(\cM)^n}.
		\]
		Now fix $\lambda \in (0,1)$, and let
		\[
		\psi^{\cM}(\mathbf{x}') = \lambda \phi^{\cM}(\mathbf{x}') + (1-\lambda) \frac{1}{2} \norm{\mathbf{x}'}_{L^2(\cM)^n}^2.
		\]
		Then $\psi$ is $(1-\lambda)$-strongly convex by construction.  Moreover, $\lambda \mathbf{y} + (1 - \lambda) \mathbf{x} \in \underline{\nabla} \psi^{\cM}(\mathbf{x})$ because $\mathbf{x}$ is a subgradient vector to $\mathbf{x}' \mapsto (1/2) \norm{\mathbf{x}'}_{L^2(\cM)^n}^2$ at the point $\mathbf{x}$.  Thus, applying Proposition \ref{prop: definable subgradient 2} to $\psi$, we see that $\mathbf{x} \in \dcl^{\cM}((1 - \lambda) \mathbf{x} + \lambda \mathbf{y})$.  A symmetrical argument shows that $\mathbf{y} \in \dcl^{\cM}((1 - \lambda) \mathbf{x} + \lambda \mathbf{y})$.  Hence, $\dcl^{\cM}(\mathbf{x},\mathbf{y}) \subseteq \dcl^{\cM}((1-\lambda)\mathbf{x} + \lambda \mathbf{y})$.  The opposite inclusion $\dcl^{\cM}((1-\lambda)\mathbf{x} + \lambda \mathbf{y}) \subseteq \dcl^{\cM}(\mathbf{x},\mathbf{y})$ is immediate.
	\end{proof}
	
	\section{Semiconcave and semiconvex definable predicates} \label{sec: semiconvex semiconcave}
	
	\subsection{Semiconvexity, Semiconcavity, and Differentiability} \label{subsec: differentiability}
	
	\begin{definition}
		Let $\phi$ be a function on an inner-product space $H$ and let $c > 0$.  We say that $f: H \to \R$ is \emph{$c$-semiconvex} if $\phi(x) + \frac{1}{2} c \norm{x}^2$ is convex, or equivalently (after some computation), for $x_0, x_1 \in H$ and $\lambda \in [0,1]$,
		\[
		\phi((1 - \lambda)x_0 + \lambda x_1) \leq (1 - \lambda) \phi(x_0) + \lambda \phi(x_1) + \frac{c \lambda(1 - \lambda)}{2} \norm{x_1 - x_0}^2.
		\]
		Similarly, we say that $\phi$ is $c$-semiconcave if $\phi(x) - \frac{c}{2} \norm{x}^2$ is concave, or equivalently,
		\[
		\phi((1 - \lambda)x_0 + \lambda x_1) \geq (1 - \lambda) \phi(x_0) + \lambda \phi(x_1) - \frac{c \lambda(1 - \lambda)}{2} \norm{x_1 - x_0}^2.
		\]
	\end{definition}
	
	\begin{proposition} \label{prop: semiconvex semiconcave gradient}
		Fix a consistent theory $\mathrm{T} \models \mathrm{T}_{\tr}$.  Let $\phi(\mathbf{y},\mathbf{z}) = \phi(y_1,\dots,y_n, (z_j)_{j \in \N})$ be a definable predicate in countably many variables.  Suppose that $\phi$ is $c$-semiconvex and $c$-semiconcave as a function of $\mathbf{y}$ for some $c > 0$.
		\begin{enumerate}[(1)]
			\item If $\cM \models \mathrm{T}$, $\phi^{\cM}$ is a differentiable function of $\mathbf{x}$, and more precisely, for each $\mathbf{x} \in \cM^n$ and $\mathbf{y} \in \cM^{\N}$, there exists a unique $\nabla_{\mathbf{x}} \phi^{\cM}(\mathbf{x},\mathbf{y}) \in L^2(\dcl^{\cM}(\mathbf{x},\mathbf{y}))^n$ such that
			\[
			\left|\phi^{\cM}(\mathbf{x}',\mathbf{y}) - \phi^{\cM}(\mathbf{x},\mathbf{y}) - \re \ip{\mathbf{x}' - \mathbf{x}, \nabla_{\mathbf{x}} \phi^{\cM}(\mathbf{x},\mathbf{y})} \right| \leq \frac{c}{2} \norm{\mathbf{x}' - \mathbf{x}}_{L^2(\cM)}^2.
			\]
			\item The function
			\[
			\iota^{\cM}(\mathbf{x}, \mathbf{y},\mathbf{z}) = \re \ip{\mathbf{z},\nabla \phi^{\cM}(\mathbf{x},\mathbf{y})}_{L^2(\cM)^n}
			\]
			is a definable predicate relative to $\mathrm{T}$.
			\item Suppose that for each $\mathbf{r} \in (0,\infty)^n$ and $\mathbf{r}' \in (0,\infty)^{\N}$, there exists $\mathbf{R}$ such that $\nabla_{\mathbf{y}} \phi$ maps $D_{\mathbf{r},\mathbf{r}'}^{\cM}$ into $D_{\mathbf{R}}^{\cM}$.  Then $\nabla_{\mathbf{y}} \phi$ is a definable function.
		\end{enumerate}
	\end{proposition}
	
	\begin{proof}
		(1) Since $\phi^{\cM}(\mathbf{x},\mathbf{y}) + (c/2) \norm{\mathbf{x}}_{L^2(\cM)^n}^2$ is a convex function of $\mathbf{x}$, the same reasoning as in Proposition \ref{prop: definable subgradient} shows that there exists some $\mathbf{z} \in L^2(\dcl^{\cM}(\mathbf{x},\mathbf{y}))^n$ such that for $\mathbf{x}' \in \cM^n$,
		\[
		\phi^{\cM}(\mathbf{x}',\mathbf{y}) + \frac{c}{2} \norm{\mathbf{x}'}_{L^2(\cM)^n}^2 - \phi^{\cM}(\mathbf{x},\mathbf{y}) - \frac{c}{2} \norm{\mathbf{x}}_{L^2(\cM)^n}^2 \geq \re \ip{\mathbf{x}' - \mathbf{x}, \mathbf{z}}_{L^2(\cM)^n},
		\]
		or equivalently
		\begin{align*}
			\phi^{\cM}(\mathbf{x}',\mathbf{y}) - \phi^{\cM}(\mathbf{x},\mathbf{y}) &\geq \re \ip{\mathbf{x}' - \mathbf{x}, \mathbf{z} - c \mathbf{x}}_{L^2(\cM)^n} - \frac{c}{2} \norm{\mathbf{x}' - \mathbf{x}}_{L^2(\cM)^n}^2 \\
			&= \re \ip{\mathbf{x}' - \mathbf{x}, \mathbf{z}_1}_{L^2(\cM)^n} - \frac{c}{2} \norm{\mathbf{x}' - \mathbf{x}}_{L^2(\cM)^n}^2
		\end{align*}
		where we let $\mathbf{z}_1 = \mathbf{z} - c\mathbf{x}$.  Symmetrically, there exists $\mathbf{z}_2 \in L^2(\dcl^{\cM}(\mathbf{x},\mathbf{y}))$ such that
		\[
		\phi^{\cM}(\mathbf{x}',\mathbf{y}) - \phi^{\cM}(\mathbf{x},\mathbf{y}) \leq \re \ip{\mathbf{x}' - \mathbf{x}, \mathbf{z}_2}_{L^2(\cM)^n} + \frac{c}{2} \norm{\mathbf{x}' - \mathbf{x}}_{L^2(\cM)^n}^2.
		\]
		By combining these two inequalities, we get
		\[
		\re \ip{\mathbf{x}' - \mathbf{x}, \mathbf{z}_2 - \mathbf{z}_1}_{L^2(\cM)^n} \leq c \norm{\mathbf{x}' - \mathbf{x}}_{L^2(\cM)^n}^2.
		\]
		Since $\mathbf{x}'$ is arbitrary, this forces $\mathbf{z}_1 = \mathbf{z}_2$.  Calling this vector $\nabla_{\mathbf{x}} \phi^{\cM}(\mathbf{x},\mathbf{y})$, we obtain the inequality asserted in (1) which entails that $\phi$ is differentiable in $\mathbf{x}$.
		
		(2) Note that for $t > 0$, we have
		\[
		-\frac{c}{2} t^2 \norm{\mathbf{z}}_{L^2(\cM)^n}^2 \leq \phi(\mathbf{x} + t \mathbf{z}, \mathbf{y}) - \phi^{\cM}(\mathbf{x},\mathbf{y})- \re \ip{\nabla_{\mathbf{x}} \phi^{\cM}(\mathbf{x},\mathbf{y}), t\mathbf{z}}_{L^2(\cM)} \leq \frac{c}{2} t^2 \norm{\mathbf{z}}_{L^2(\cM)^n}^2.
		\]
		In particular, we have
		\[
		\left|\frac{\phi(\mathbf{x} + t \mathbf{z}, \mathbf{y}) - \phi^{\cM}(\mathbf{x},\mathbf{y})}{t} - \re \ip{\nabla_{\mathbf{x}} \phi^{\cM}(\mathbf{x},\mathbf{y}), \mathbf{z}}_{L^2(\cM)} \right| \leq \frac{ct}{2} \norm{\mathbf{z}}_{L^2(\cM)^n}^2.
		\]
		In particular, $[\phi(\mathbf{x} + t \mathbf{z}, \mathbf{y}) - \phi^{\cM}(\mathbf{x},\mathbf{y})]/t$ converges uniformly on $L^2$-balls to $\re \ip{\nabla_{\mathbf{x}} \phi^{\cM}(\mathbf{x},\mathbf{y}), \mathbf{z}}_{L^2(\cM)}$ as $t \searrow 0$, which makes the latter a definable predicate.
		
		(3) Fix $\mathbf{r}$ and $\mathbf{r}'$, and let $\mathbf{R}$ be as in the assumption of (3).  Note that
		\[
		\psi_{\epsilon}^{\cM}(\mathbf{x},\mathbf{y}) = \sup_{\mathbf{z} \in D_{\mathbf{R}}^{\cM}} \frac{(\re \ip{\mathbf{z},\nabla_{\mathbf{x}} \phi(\mathbf{x},\mathbf{y})}_{L^2(\cM)^n})^2}{\epsilon + \norm{\mathbf{z}}_{L^2(\cM)^n}^2}
		\]
		is a definable predicate.  Since $\nabla_{\mathbf{x}} \phi(\mathbf{x},\mathbf{y}) \in D_{\mathbf{R}}^{\cM}$ for $\mathbf{x} \in D_{\mathbf{r}}$ and $\mathbf{y} \in D_{\mathbf{r}'}$, we have
		\[
		\psi_{\epsilon}^{\cM}(\mathbf{x},\mathbf{y}) \geq \frac{\norm{\nabla_{\mathbf{x}} \phi(\mathbf{x},\mathbf{y})}_{L^2(\cM)^n}^4}{\epsilon + \norm{\nabla_{\mathbf{x}} \phi(\mathbf{x},\mathbf{y})}_{L^2(\cM)^n}^2},
		\]
		while on the other hand by the Cauchy-Schwarz inequality,
		\[
		\psi_{\epsilon}^{\cM}(\mathbf{x},\mathbf{y}) \leq \sup_{\mathbf{z} \in D_{\mathbf{R}}} \frac{\norm{\mathbf{z}}_{L^2(\cM)}^2 \norm{\nabla_{\mathbf{x}} \phi(\mathbf{x},\mathbf{y})}_{L^2(\cM)^n}^2}{\epsilon + \norm{\mathbf{z}}_{L^2(\cM)^n}^2} \leq \norm{\nabla_{\mathbf{x}} \phi(\mathbf{x},\mathbf{y})}_{L^2(\cM)^n}^2.
		\]
		It is an exercise in analysis to show that $t^4 / (\epsilon + t^2)$ converges uniformly on compact sets to $t^2$.  Therefore, these bounds prove that $\psi_{\epsilon}^{\cM}(\mathbf{x},\mathbf{y})$ converges uniformly to $\norm{\nabla_{\mathbf{x}} \phi(\mathbf{x},\mathbf{y})}_{L^2(\cM)^n}^2$ for $\mathbf{x} \in D_{\mathbf{r}}$ and $\mathbf{y} \in D_{\mathbf{r}'}$ as $\epsilon \searrow 0$.  Hence, $\norm{\nabla_{\mathbf{x}} \phi(\mathbf{x},\mathbf{y})}_{L^2(\cM)^n}^2$ is a definable predicate on $D_{\mathbf{r}}^{\cM} \times D_{\mathbf{r}'}^{\cM}$ for every $\mathbf{r}$ and $\mathbf{r}'$, meaning that $\norm{\nabla_{\mathbf{x}} \phi(\mathbf{x},\mathbf{y})}_{L^2(\cM)^n}^2$ is globally a definable predicate.
		
		This furthermore shows that
		\[
		\norm{\mathbf{z} - \nabla_{\mathbf{x}} \phi^{\cM}(\mathbf{x},\mathbf{y})}_{L^2(\cM)^n} = \norm{\mathbf{z}}_{L^2(\cM)^n}^2 - 2 \re \ip{\mathbf{z},\nabla \phi^{\cM}(\mathbf{x},\mathbf{y})}_{L^2(\cM)^n} + \norm{\nabla_{\mathbf{x}} \phi^{\cM}(\mathbf{x},\mathbf{y})}_{L^2(\cM)^n}^2
		\]
		is a definable predicate, which means that $\nabla_{\mathbf{x}} \phi^{\cM}(\mathbf{x},\mathbf{y})$ is a definable function.
	\end{proof}
	
	Furthermore, in the setting of Proposition \ref{prop: semiconvex semiconcave gradient}, $\nabla_{\mathbf{x}} \phi(\mathbf{x},\mathbf{y})$ is automatically Lipschitz in $\mathbf{x}$.  To prove this, we rely on the following elementary estimate.
	
	\begin{fact} \label{fact: second difference}
		Let $H$ be an inner product space and let $\phi: H \to \R$ be $c$-semiconvex and $c$-semiconcave.  Then for $x, y, z \in H$, we have
		\[
		|\phi(x + y + z) - \phi(x + y) - \phi(x + z) + \phi(x)| \leq c \norm{y} \norm{z}.
		\]
	\end{fact}
	
	\begin{proof}
		Assume $y$ and $z$ are nonzero, since otherwise the inequality is trivial.  First observe by semiconvexity $\phi(\frac{x+y}{2}) \leq \frac{1}{2} \phi(x) + \frac{1}{2} \phi(y) + \frac{1}{2}(1 - \frac{1}{2}) \frac{c}{2} \norm{x - y}^2$, and hence
		\[
		\phi(x) + \phi(y) \geq 2 \phi\left( \frac{x + y}{2} \right) - \frac{c}{4} \norm{x - y}^2.
		\]
		Analogously, semiconcavity implies
		\[
		\phi(x) + \phi(y) \leq 2 \phi\left( \frac{x + y}{2} \right) + \frac{c}{4} \norm{x - y}^2.
		\]
		Therefore,
		\begin{align*}
			& \phi(x + y + z) - \phi(x + y) - \phi(x + z) + \phi(x) \\
			&= [\phi(x + y + z) + \phi(x)] - [\phi(x + y) + \phi(x+z)] \\
			&\leq \left[ 2\phi \left( \frac{2x + y + z}{2} \right) + \frac{c}{4} \norm{y+z}^2 \right] - \left[ 2\phi \left( \frac{2x + y + z}{2} \right) - \frac{c}{4} \norm{y-z}^2 \right] \\
			&= \frac{c}{2} \left(\norm{y}^2 + \norm{z}^2 \right),
		\end{align*}
		so
		\begin{equation} \label{eq: halfway estimate}
			\phi(x + y + z) - \phi(x + y) - \phi(x + z) + \phi(x) \leq \frac{c}{2} \left( \norm{y}^2 + \norm{z}^2 \right).
		\end{equation}
		Let $m, n \in \N$.  By telescoping summation,
		\begin{multline*}
			\phi(x + y + z) - \phi(x + y) - \phi(x + z) + \phi(x) \\
			= \sum_{i=1}^m \sum_{j=1}^n \left[ \phi\left(x + \tfrac{i}{m} y + \tfrac{j}{n} z\right) - \phi\left(x + \tfrac{i}{m}y + \tfrac{j-1}{n} \right) - \phi\left(x + \tfrac{i-1}{m}y + \tfrac{j}{m} z \right) + \phi\left(x + \tfrac{i-1}{m} y + \tfrac{j-1}{n} z \right) \right].
		\end{multline*}
		Now applying \eqref{eq: halfway estimate} to $x + \frac{i-1}{m}y + \frac{j-1}{n} z$ and $\frac{1}{m}y$ and $\frac{1}{n}z$ in place of $x$, $y$, and $z$, we get
		\[
		\phi(x + y + z) - \phi(x + y) - \phi(x + z) + \phi(x)  \leq \sum_{i=1}^m \sum_{j=1}^n \frac{c}{2} \left( \frac{1}{m^2} \norm{y}^2 + \frac{1}{n^2} \norm{z}^2 \right) = \frac{c}{2} \left( \frac{n}{m} \norm{y}^2 + \frac{m}{n} \norm{z}^2 \right).
		\]
		Now apply this estimate to sequences of integers $m_k$ and $n_k$ such that $m_k/n_k \to \norm{y} / \norm{z}$.  The right-hand side thus reduces to $c \norm{y} \norm{z}$ which gives us the desired upper bound for $\phi(x + y + z) - \phi(x + y) - \phi(x + z) + \phi(x)$.  The lower bound follows by a symmetrical argument.
	\end{proof}
	
	\begin{corollary} \label{cor: gradient Lipschitz}
		Consider the same setup as Proposition \ref{prop: semiconvex semiconcave gradient}.  Then
		\begin{equation} \label{eq: gradient Lipschitz}
			\norm{\nabla_{\mathbf{x}} \phi^{\cM}(\mathbf{x}',\mathbf{y}) - \nabla_{\mathbf{x}} \phi^{\cM}(\mathbf{x},\mathbf{y})}_{L^2(\cM)} \leq c \norm{\mathbf{x}' - \mathbf{x}}_{L^2(\cM)}.
		\end{equation}
	\end{corollary}
	
	\begin{proof}
		Let $\mathbf{y} \in \cM^{\N}$.  By applying Fact \ref{fact: second difference} to $\phi^{\cM}(\cdot,\mathbf{y})$, we obtain for $\mathbf{x}'$ and $\mathbf{x}$ and $\mathbf{z} \in \cM^n$ that
		\[
		\left| \frac{\phi^{\cM}(\mathbf{x}' + t \mathbf{z},\mathbf{y}) - \phi^{\cM}(\mathbf{x}',\mathbf{y})}{t} -  \frac{\phi^{\cM}(\mathbf{x}' + t \mathbf{z},\mathbf{y}) - \phi^{\cM}(\mathbf{x}',\mathbf{y})}{t} \right| \leq c \norm{\mathbf{x}' - \mathbf{x}}_{L^2(\cM)^n} \norm{\mathbf{z}}_{L^2(\cM)^n}.
		\]
		Taking $t \to 0$,
		\[
		\left| \re \ip{\nabla \phi^{\cM}(\mathbf{x}',\mathbf{y}),\mathbf{z}}_{L^2(\cM)^n} -   \re \ip{\nabla \phi^{\cM}(\mathbf{x},\mathbf{y}),\mathbf{z}}_{L^2(\cM)^n} \right| \leq c \norm{\mathbf{x}' - \mathbf{x}}_{L^2(\cM)^n} \norm{\mathbf{z}}_{L^2(\cM)^n}.
		\]
		Since $\mathbf{z}$ was arbitrary, we obtain \eqref{eq: gradient Lipschitz}.
	\end{proof}

	In Proposition \ref{prop: semiconvex semiconcave gradient} (3), we made an additional boundedness assumption that $\nabla_{\mathbf{y}} \phi$ maps a product of domains (operator-norm balls) into a product of domains, which must be verified in applications.  We will see examples of definable predicates in the next section where the boundedness condition holds by construction.  However, we remark that the boundedness condition is automatic in the case where $\phi$ is a function of $\mathbf{x}$ alone (i.e., there is no $\mathbf{y}$) and $\mathrm{T} \models \mathrm{T}_{\mathrm{II}_1}$.  This follows from a much more general fact about unitarily equivariant Lipschitz functions on tuples from a $\mathrm{II}_1$ factor.  A weaker version of this lemma was given earlier in \cite[Lemma 11.5.4]{JekelThesis} using a random matrix argument.  In the following, $L^2(\cM)_{\sa}$ denotes the real Hilbert space of self-adjoint elements in $L^2(\cM)$.  Later on, we will apply this result to $\nabla \phi$ after showing it is Lipschitz.
	
	\begin{proposition} \label{prop: Lipschitz bound}
		Let $\cM$ be $\mathrm{II}_1$ factor (i.e. a tracial $\mathrm{W}^*$-algebra with trivial center), and let $F: \cM_{\sa}^n \to L^2(\cM)_{\sa}$ be a function such that
		\begin{enumerate}[(1)]
			\item $F$ is $L$-Lipschitz with respect to the $L^2(\cM)$-norm, that is, $\norm{F(\mathbf{x}) - F(\mathbf{y})}_{L^2(\cA)} \leq L \norm{\mathbf{x} - \mathbf{y}}_{L^2(\cA)^n}$.
			\item $F$ is equivariant under unitary conjugation, that is, $F(ux_1u^*,\dots,ux_nu^*) = u F(x_1,\dots,x_n) u^*$ for any unitary $u$ in $\cM$.
		\end{enumerate}
		Then we have
		\begin{equation} \label{eq:spectraldiameter}
			\diam(\Spec(F(\mathbf{x}))) \leq L \left( \sum_{j=1}^m \diam(\Spec(x_j))^2 \right)^{1/2},
		\end{equation}
	\end{proposition}
	
	\begin{proof}
		Fix $\mathbf{x} \in \cM_{\sa}^n$.  Let $c_j$ be the average of the minimum and maximum of $\Spec(x_j)$, so that $\norm{X_j - c_j1}_{L^\infty(\cA)} = (1/2) \diam(\Spec(x_j))$.  Note that for every unitary $U$, we have
		\begin{align*}
			\norm{ux_ju^* - x_j}_{L^2(\cM)} &= \norm{u(x_j - c_j)u^*- (x_j - c_j)}_{L^2(\cM)} \\
			&\leq \norm{(u - 1)(x_j - c_j)}_{L^2(\cM)} + \norm{u(x_j - c_j)(u^* - 1)}_{L^2(\cA)} \\
			&\leq 2 \norm{u - 1}_{L^2(\cM)} \frac{1}{2} \diam(\Spec(x_j)).
		\end{align*}
		Therefore,
		\begin{align*}
			\norm{uF(\mathbf{x})u^* - F(\mathbf{x})}_{L^2(\cM)} &\leq L \norm{u\mathbf{x}u^* - \mathbf{x}}_{L^2(\cM)_{\sa}^m} \\
			&\leq L \norm{u - 1}_{L^2(\cM)} \left(\sum_{j=1}^m \diam(\Spec(x_j))^2 \right)^{1/2}.
		\end{align*}
		For the sake of brevity, let $K = \left(\sum_{j=1}^m \diam(\Spec(x_j))^2 \right)^{1/2}$.  Our previous estimate implies in particular that for $y \in \cM_{\sa}$ and $t \in \R$, then
		\[
		\norm{e^{ity}F(\mathbf{x}) - F(\mathbf{x}) e^{ity}}_{L^2(\cM)} = \norm{e^{ity} F(\mathbf{x}) e^{-ity} - F(\mathbf{x})}_{L^2(\cM)} \leq LK \norm{e^{ity} - 1}_{L^2(\cM)}.
		\]
		As $t \to 0$, we have $(e^{ity} - 1) / t \to iy$ and $e^{ity}F(\mathbf{x}) - F(\mathbf{x}) e^{ity} \to i[y,F(\mathbf{x})]$ in $L^2(\cA)$.  Therefore, for $y \in L^\infty(\cM)_{\sa}$, we have
		\[
		\norm{i[y,F(\mathbf{x})]}_{L^2(\cM)} \leq LK \norm{y}_{L^2(\cM)}.
		\]
		Next, observe that if $y$ is self-adjoint, then so is $i[y,F(\mathbf{x})]$.  Taking $z \in \cM$ and writing $z = \re(z) + i \im(z)$ where $\re(z) = (z + z^*)/2$ and $\im(z) = (z - z^*)/2$, we have $\re(i[z,F(\mathbf{x})]) = i[\re(z),F(\mathbf{x}))]$ and $\im(i[z,F(\mathbf{x})]) = i[\im(z),F(\mathbf{x})]$.  Therefore,
		\begin{align*}
			\norm{i[z,f(\mathbf{x})]}_{L^2(\cM)} &= \left( \norm{\re(i[z,F(\mathbf{x})])}_{L^2(\cM)}^2 + \norm{\im(i[z,F(\mathbf{x})])}_{L^2(\cM)}^2 \right)^{1/2} \\
			&\leq LK \left( \norm{\re(z)}_{L^2(\cM)}^2 + \norm{\im(z)}_{L^2(\cM)}^2 \right)^{1/2} \\
			&= LK \norm{z}_{L^2(\cM)}.
		\end{align*}
		Let $a$ and $b \in \Spec(F(\mathbf{x}))$, and let $\epsilon > 0$.  Then because the projection-valued spectral measure of $F(\mathbf{x})$ has support equal to $\Spec(F(\mathbf{x}))$, the spectral projections $1_{[a-\epsilon,a+\epsilon]}(F(\mathbf{x}))$ and $1_{[b-\epsilon,b+\epsilon]}(F(\mathbf{x}))$ are nonzero. Because $\cM$ is a $\mathrm{II}_1$ factor, there exists a nonzero partial isometry $v$ such that $v^*v \leq 1_{[a-\epsilon,a+\epsilon]}(F(\mathbf{x}))$ and $vv^* \leq 1_{[b-\epsilon,b+\epsilon]}(F(\mathbf{x}))$ (see e.g. \cite[Proposition 6.1.8]{KadisonRingroseII}).  Then
		\begin{align*}
			\norm{vF(\mathbf{x}) - av}_{L^2(\cM)} &= \norm{v \, 1_{[a-\epsilon,a+\epsilon]}(F(\mathbf{x})) (F(\mathbf{x}) - a)}_{L^2(\cM)} \\
			&\leq \norm{v}_{L^2(\cM)} \norm{1_{[a-\epsilon,a+\epsilon]}(F(\mathbf{x})) (F(\mathbf{x}) - a)}_{L^\infty(\cM)} \\
			&\leq \norm{v}_{L^2(\cM)} \, \epsilon.
		\end{align*}
		Similarly, $\norm{F(\mathbf{x}) v - b v}_{L^\infty(\cM)} \leq \norm{v}_{L^2(\cA)} \, \epsilon$, and therefore
		\[
		\norm{[v,F(\mathbf{x})] - (a - b)v}_{L^2(\cM)} \leq 2 \epsilon \norm{v}_{L^2(\cM)}.
		\]
		This implies that
		\begin{align*}
			|a - b| \norm{v}_{L^2(\cM)} &\leq \norm{[v,F(\mathbf{x})]}_{L^2(\cM)} + 2 \epsilon \norm{v}_{L^2(\cM)} \\
			&\leq LK \norm{v}_{L^2(\cM)} + 2 \epsilon \norm{v}_{L^2(\cM)}.
		\end{align*}
		We can cancel $\norm{v}_{L^2(\cM)}$ from both sides.  Then since $\epsilon > 0$ was arbitrary and $a$ and $b$ were arbitrary points in $\Spec(F(\mathbf{x}))$, we obtain $\diam(\Spec(F(\mathbf{x}))) \leq LK$, which is the same as \eqref{eq:spectraldiameter}.
	\end{proof}
	
	\begin{corollary} \label{cor: Lipschitz boundedness}
		Let $\cM$ be a $\mathrm{II}_1$ factor and let $F: \cM^n \to L^2(\cM)^m$ be $L$-Lipschitz with respect to $\norm{\cdot}_{L^2(\cM)}$ and equivariant under unitary conjugation.  Let $\mathbf{r} = (r_1,\dots,r_n) \in (0,\infty)^n$.  Let $t = \max_i (|\tr^{\cM}(F_i(0))|)$.  Then $F$ maps $D_{\mathbf{r}}^{\cM}$ into $(D_{t+9L|\mathbf{r}|}^{\cM})^m$.
	\end{corollary}
	
	\begin{proof}
		For an operator $x$, we write $\re(x) = (x+x^*)/2$ and $\im(x) = (x - x^*)/2i$, so that $\re(x)$ and $\im(x)$ are self-adjoint and $x = \re(x) + i \im (x)$.
		
		Fix $i \in \{1,\dots,m\}$, and let $G_i: \cM_{\sa}^{2n} \to L^2(\cM)_{\sa}$ be given by
		\[
		G_i(y_1,\dots,y_{2n}) = \re (F(y_1+iy_2,\dots,y_{2n-1} + i y_{2n}))_i.
		\]
		Then $G_i$ is $L$-Lipschitz with respect to the $L^2$-norm.  Thus, by the previous proposition,
		\begin{align*}
			\norm{G_i(y_1,\dots,y_{2n}) - \tr^{\cM}(G_i(y_1,\dots,y_{2n})))}_{\cM} &\leq \diam(\Spec(G_i(y_1,\dots,y_{2n})) \\
			&\leq L \left( \sum_{j=1}^{2n} \diam(\Spec(y_j))^2 \right)^{1/2} \leq 2L \left( \sum_{j=1}^{2n} \norm{y_j}_{\cM}^2 \right)
		\end{align*}
		By substituting $x_j = y_{2j-1} + i y_{2j}$, this is equivalent to
		\begin{align*}
			\norm{\re F_i(x_1,\dots,x_n) - \tr^{\cM}(\re F_i(x_1,\dots,x_n))}_{\cM} &\leq 2L \left( \sum_{j=1}^n (\norm{\re(x_j)}^2 + \norm{\im(x_j)}^2) \right) \\
			&\leq 4 L \left( \sum_{j=1}^n \norm{x_j}^2 \right) \\
			&= 4 L |\mathbf{r}|.
		\end{align*}
		The same reasoning applies to $\im F_i(x_1,\dots,x_n)$, and hence adding the real and imaginary parts together,
		\[
		\norm{F_i(x_1,\dots,x_n) - \tr^{\cM}( F_i(x_1,\dots,x_n))}_{\cM} \leq 8 L |\mathbf{r}|.
		\]
		Also,
		\[
		|\tr^{\cM}(F_i(x_1,\dots,x_n)) - \tr^{\cM}(F_i(0,\dots,0))| \leq L \norm{\mathbf{x}}_{L^2(\cM)^n} \leq L|\mathbf{r}|
		\]
		and by assumption $|\tr^{\cM}(F_i(0,\dots,0))| \leq t$.  Combining these estimates with the triangle inequality shows $\norm{F_i^{\cM}(x_1,\dots,x_n)} \leq t + 9L|\mathbf{r}|$ as desired.
	\end{proof}
	
	\begin{corollary} \label{cor: factor gradient definability}
		Let $\mathrm{T}_{\mathrm{II}_1}$ be the theory of $\mathrm{II}_1$ factors.  Fix a consistent theory $\mathrm{T} \models \mathrm{T}_{\mathrm{II}_1}$.  Let $\phi(\mathbf{x}) = \phi(x_1,\dots,x_n)$ be a definable predicate in $n$ variables.  Suppose that $\phi$ is $c$-semiconvex and $c$-semiconcave as a function of $\mathbf{x}$ for some $c > 0$.  Then $\nabla \phi^{\cM}$ is a $\mathrm{T}$-definable function.
	\end{corollary}
	
	\begin{proof}
		Let $\cM \models \mathrm{T}$.  By Corollary \ref{cor: gradient Lipschitz}, $\nabla \phi^{\cM}$ is $c$-Lipschitz.  Thus, by Corollary \ref{cor: Lipschitz boundedness}, $\nabla \phi^{\cM}$ maps $D_{\mathbf{r}}$ into $(D_{t+9c|\mathbf{r}|})^n$.  Hence, Proposition \ref{prop: semiconvex semiconcave gradient} implies that $\nabla \phi$ is a definable function relative to $\mathrm{T}$.
	\end{proof}

	\subsection{Semiconvex semiconcave regularization} \label{subsec: regularization}
	
	In this section, we prove Theorem \ref{thm: approximation} using a version of the sup-convolution and inf-convolution operations that have been widely used in the study of Hamilton-Jacobi equations.  These operations fit quite naturally with continuous model theory where $\sup$'s and $\inf$'s are baked into the recipe for formulas.
	
	\begin{proposition} \label{prop: semiconcave regularization}
		Let $\phi(\mathbf{x},\mathbf{y})$ be a $\mathrm{T}_{\tr}$-definable predicate, where $\mathbf{x}$ is an $m$-tuple and $\mathbf{y}$ is a countable tuple.  Fix $t > 0$ and $\mathbf{r} = (r_1,\dots,r_n) \in (0,\infty)^n$, and write $|\mathbf{r}| = (r_1^2 + \dots + r_n^2)^{1/2}$.  Let
		\[
		\psi(\mathbf{x},\mathbf{y}) = \inf_{\mathbf{z} \in \mathcal{D}_{\mathbf{r}}} \left[ \phi(\mathbf{z},\mathbf{y}) + \frac{1}{2t} d(\mathbf{x},\mathbf{z})^2 \right].
		\]
		Then
		\begin{enumerate}[(1)]
			\item $\psi^{\cM}(\mathbf{x},\mathbf{y})$ is a $1/t$-semiconcave function of $\mathbf{x}$.
			\item For each $\mathbf{r}' \in (0,\infty)^n$, $\psi^{\cM}$ is $2|\mathbf{r}+\mathbf{r}'|/t$-Lipschitz as a function of $\mathbf{x}$ on $D_{\mathbf{r}'}^{\cM}$.
			\item For each $\cM$ and for $\mathbf{x} \in D_{\mathbf{r}}$ and $\mathbf{y} \in D_{\mathbf{s}}$, we have
			\[
			-\omega_{\phi,\mathbf{r},\mathbf{s}}\left( \sqrt{2t \omega_{\phi,\mathbf{r},\mathbf{s}}(2|\mathbf{r}|)} \right) \leq \psi^{\cM}(\mathbf{x},\mathbf{y}) - \phi^{\cM}(\mathbf{x},\mathbf{y}) \leq 0,
			\]
			where $\omega_{\phi,\mathbf{r},\mathbf{s}}$ is the modulus of uniform continuity for $\phi$ as a function of $\mathbf{x}$, over $\mathbf{x} \in D_{\mathbf{r}}$ and $\mathbf{y} \in D_{\mathbf{s}}$.
			\item  For each $\mathbf{x} \in D_{\mathbf{r}'}^{\cM}$, there exists $\mathbf{p} \in D_{\mathbf{r}+ \mathbf{r}'}^{\cM}$ such that $(1/t)\mathbf{p} \in \overline{\nabla}_{\mathbf{x}} \psi^{\cM}(\mathbf{x},\mathbf{y})$.
			\item Suppose that $\phi^{\cM}$ is $1/u$-semiconvex as a function of $\mathbf{x}$ for some $u > 0$, and assume $0 < t < u$.  Then $\psi^{\cM}$ is $1/(u - t)$-semiconvex.
		\end{enumerate}
	\end{proposition}
	
	\begin{proof}
		(1) Note that $(1/2t)d^{\cM}(\mathbf{x},\mathbf{y})^2$ is $1/t$-semiconcave as a function of $\mathbf{x}$.  Hence, so is $\phi^{\cM}(\mathbf{z},\mathbf{y}) + \frac{1}{2t} d^{\cM}(\mathbf{x},\mathbf{y})^2$.  It is a standard analysis exercise to show that the infimum of a family of family of $1/t$-semiconcave functions is also $1/t$-semiconcave.  Hence, $\psi^{\cM}$ is a $1/t$-semiconcave function of $\mathbf{x}$.
		
		(2) Note that the gradient of $\frac{1}{2} d^{\cM}(\mathbf{x},\mathbf{z})^2$ as a function of $\mathbf{x}$ on the real Hilbert space $L^2(\cM)^n$ (with the inner product given by the real part of the complex inner product) is $\mathbf{x} - \mathbf{z}$.  For $\mathbf{x}$ in $D_{\mathbf{r}'}^{\cM}$, we have $\mathbf{x} - \mathbf{z} \in D_{\mathbf{r}+\mathbf{r}'}$ and so $\norm{\mathbf{x} - \mathbf{z}}_{L^2(\cM)^n} \leq |\mathbf{r} + \mathbf{r}'|$.  Thus, $\frac{1}{2} d^{\cM}(\mathbf{x},\mathbf{z})^2$ is $|\mathbf{r}+\mathbf{r}'|$-Lipschitz as a function of $\mathbf{x}$ on this domain.  The Lipschitz condition is preserved when taking the infimum of a family of functions.
		
		(3) Note that $\psi^{\cM} \leq \phi^{\cM}$ because $\mathbf{x}$ is a candidate for $\mathbf{z}$ in the infimum.  For the opposite inequality, it suffices to show that for $\mathbf{z} \in D_{\mathbf{r}}^{\cM}$, and we have
		\[
		\left[ \phi^{\cM}(\mathbf{z},\mathbf{y}) + \frac{1}{2t} d^{\cM}(\mathbf{x},\mathbf{z})^2 \right] - \phi^{\cM}(\mathbf{x},\mathbf{y}) \geq -\omega_{\phi,\mathbf{r},\mathbf{s}}\left( \sqrt{2t \omega_{\phi,\mathbf{r},\mathbf{s}}(2|\mathbf{r}|)} \right).
		\]
		We consider two cases.  First, if $d^{\cM}(\mathbf{x},\mathbf{z}) \geq \sqrt{2t \omega_{\phi,\mathbf{r},\mathbf{s}}(2|\mathbf{r}|)}$, then we get
		\begin{align*}
			\left[ \phi^{\cM}(\mathbf{z},\mathbf{y}) + \frac{1}{2t} d^{\cM}(\mathbf{x},\mathbf{z})^2 \right] - \phi^{\cM}(\mathbf{x},\mathbf{y}) &\geq \frac{1}{2t} d^{\cM}(\mathbf{x},\mathbf{z})^2 - \omega_{\phi,\mathbf{r},\mathbf{s}}(d^{\cM}(\mathbf{x},\mathbf{z})) \\
			&\geq \frac{1}{2t} d^{\cM}(\mathbf{x},\mathbf{z})^2 - \omega_{\phi,\mathbf{r},\mathbf{s}}(2|\mathbf{r}|) \\
			&\geq 0 \\
			&\geq -\omega_{\phi,\mathbf{r},\mathbf{s}}\left( \sqrt{2t \omega_{\phi,\mathbf{r},\mathbf{s}}(2|\mathbf{r}|)} \right).
		\end{align*}
		On the other hand, if $d^{\cM}(\mathbf{x},\mathbf{z}) \leq \sqrt{2t \omega_{\phi,\mathbf{r},\mathbf{s}}(2|\mathbf{r}|)}$, then
		\begin{align*}
			\left[ \phi^{\cM}(\mathbf{z},\mathbf{y}) + \frac{1}{2t} d^{\cM}(\mathbf{x},\mathbf{z})^2 \right] - \phi^{\cM}(\mathbf{x},\mathbf{y}) &\geq \phi^{\cM}(\mathbf{z},\mathbf{y}) - \phi^{\cM}(\mathbf{x},\mathbf{y}) \\
			&\geq -\omega_{\phi,\mathbf{r},\mathbf{s}}(d^{\cM}(\mathbf{x},\mathbf{z})) \\
			&\geq -\omega_{\phi,\mathbf{r},\mathbf{s}}\left( \sqrt{2t \omega_{\phi,\mathbf{r},\mathbf{s}}(2|\mathbf{r}|)} \right).
		\end{align*}
		
		(4) Fix $\mathbf{x} \in D_{\mathbf{r}'}^{\cM}$.  Let $\cN$ be a $\aleph_1$-saturated elementary extension of $\cM$.  Then a minimizer $\mathbf{z}$ for the infimum defining $\psi$ exists in $\cN$.  Note that for any $\mathbf{x}' \in \cN^n$, we have
		\[
		\psi_{\mathbf{r},t}^{\cN}(\mathbf{x}',\mathbf{y}) \leq \phi^{\cN}(\mathbf{z},\mathbf{y}) + \frac{1}{2t} d(\mathbf{x}',\mathbf{z})^2,
		\]
		so that
		\begin{align*}
			\psi^{\cN}(\mathbf{x}',\mathbf{y}) - \psi^{\cN}(\mathbf{x},\mathbf{y}) &\leq \frac{1}{2t} d^{\cN}(\mathbf{x}',\mathbf{z})^2 - \frac{1}{2t} d^{\cN}(\mathbf{x},\mathbf{z})^2 \\
			&= \frac{1}{2t} \norm{\mathbf{x}' - \mathbf{x}}_{L^2(\cN)^n}^2 + \frac{1}{t} \re \ip{\mathbf{x}' - \mathbf{x}, \mathbf{x} + \mathbf{z}}_{L^2(\cN)^n}.
		\end{align*}
		Therefore, $\mathbf{x} + \mathbf{z} \in \overline{\nabla}_{\mathbf{x}} \psi^{\cN}(\mathbf{x},\mathbf{y})$.  Note that $\mathbf{x} + \mathbf{z} \in D_{\mathbf{r}+\mathbf{r}'}^{\cM}$.  Now let $\mathbf{p}$ be the conditional expectation of $\mathbf{x} + \mathbf{z}$ onto $\cM$.  Then for $\mathbf{x}' \in \cM^n$, we have
		\begin{align*}
			\psi^{\cM}(\mathbf{x}',\mathbf{y}) - \psi^{\cM}(\mathbf{x},\mathbf{y}) &= \psi^{\cN}(\mathbf{x}',\mathbf{y}) - \psi^{\cN}(\mathbf{x},\mathbf{y}) \\
			&\leq \frac{1}{2t} \norm{\mathbf{x}' - \mathbf{x}}_{L^2(\cN)^n}^2 + \frac{1}{t} \re \ip{\mathbf{x}' - \mathbf{x}, \mathbf{x} + \mathbf{z}}_{L^2(\cN)^n} \\
			&= \frac{1}{2t} \norm{\mathbf{x}' - \mathbf{x}}_{L^2(\cM)^n}^2 + \frac{1}{t} \re \ip{\mathbf{x}' - \mathbf{x},\mathbf{p}}_{L^2(\cM)^n},
		\end{align*}
		and hence $(1/t)\mathbf{p} \in \overline{\nabla}_{\mathbf{x}} \psi^{\cM}(\mathbf{x},\mathbf{y})$.
		
		(5) Fix $\mathbf{x}_0$, $\mathbf{x}_1 \in \cM^n$ and $\mathbf{y} \in \cM^I$.  Fix $\epsilon > 0$, and let $\mathbf{z}_0$ and $\mathbf{z}_1$ in $D_{\mathbf{r}}^{\cM}$ be approximate minimizers satisfying
		\[
		\phi^{\cM}(\mathbf{x}_0,\mathbf{y}) + \frac{1}{2t} \norm{\mathbf{x}_0 - \mathbf{z}_0}_{L^2(\cM)^n}^2 < \psi^{\cM}(\mathbf{x}_0,\mathbf{y}) + \epsilon,
		\]
		and analogously for $\mathbf{x}_1$ and $\mathbf{z}_1$.  For $\lambda \in [0,1]$, let $\mathbf{x}_\lambda = (1 - \lambda) \mathbf{x}_0 + \lambda \mathbf{x}_1$ and $\mathbf{z}_\lambda = (1 - \lambda) \mathbf{z}_0 + \lambda \mathbf{z}_1$.  Then using the $1/u$-semiconvexity of $\phi^{\cM}$ as a function of $\mathbf{x}$,
		\begin{align*}
			\psi^{\cM}(\mathbf{x}_\lambda,\mathbf{y}) &\leq \phi^{\cM}(\mathbf{z}_\lambda,\mathbf{y}) + \frac{1}{2t} \norm{\mathbf{x}_\lambda - \mathbf{z}_\lambda}_{L^2(\cM)^n}^2 \\
			&\leq (1 - \lambda) \phi^{\cM}(\mathbf{z}_0,\mathbf{y}) + \lambda \phi^{\cM}(\mathbf{z}_1,\mathbf{y}) + \frac{\lambda(1 - \lambda)}{2u} \norm{\mathbf{z}_1 - \mathbf{z}_0}_{L^2(\cM)^n}^2 + \frac{1}{2t} \norm{\mathbf{x}_\lambda - \mathbf{z}_\lambda}_{L^2(\cN)^n}^2 \\
			&\leq (1 - \lambda) \left[\psi^{\cM}(\mathbf{x}_0,\mathbf{y}) + \epsilon - \frac{1}{2t} \norm{\mathbf{x}_0 - \mathbf{z}_0}_{L^2(\cM)^n}^2 \right] + \lambda \left[\psi^{\cM}(\mathbf{x}_1,\mathbf{y}) + \epsilon - \frac{1}{2t} \norm{\mathbf{x}_1 - \mathbf{z}_1}_{L^2(\cM)^n}^2 \right] \\
			& \qquad + \frac{\lambda(1 - \lambda)}{2u} \norm{\mathbf{z}_1 - \mathbf{z}_0}_{L^2(\cM)^n}^2 + \frac{1}{2t} \norm{\mathbf{x}_\lambda - \mathbf{z}_\lambda}_{L^2(\cN)^n}^2.
		\end{align*}
		After some computation,
		\[
		(1 - \lambda) \norm{\mathbf{x}_0 - \mathbf{z}_0}_{L^2(\cM)^n}^2 + \lambda \norm{\mathbf{x}_1 - \mathbf{z}_1}_{L^2(\cM)^n}^2  - \norm{\mathbf{x}_\lambda - \mathbf{z}_\lambda}_{L^2(\cN)^n}^2 = \lambda(1 - \lambda) \norm{(\mathbf{x}_1 - \mathbf{x}_0) - (\mathbf{z}_1 - \mathbf{z}_0)}_{L^2(\cM)^n}^2.
		\]
		Hence,
		\[
		\psi^{\cM}(\mathbf{x}_\lambda,\mathbf{y}) \leq \epsilon + (1 - \lambda) \psi^{\cM}(\mathbf{x}_0,\mathbf{y}) + \lambda \psi^{\cM}(\mathbf{x}_1,\mathbf{y}) + \lambda(1 - \lambda) \left[\frac{1}{2u} \norm{\mathbf{z}_1 - \mathbf{z}_0} - \frac{1}{2t} \norm{(\mathbf{x}_1 - \mathbf{x}_0) - (\mathbf{z}_1 - \mathbf{z}_0)}_{L^2(\cM)^n}^2 \right]
		\]
		In order to finish the proof, it suffices to show that
		\begin{equation} \label{eq: horrid computation}
			\frac{1}{2u} \norm{\mathbf{z}_1 - \mathbf{z}_0} - \frac{1}{2t} \norm{(\mathbf{x}_1 - \mathbf{x}_0) - (\mathbf{z}_1 - \mathbf{z}_0)}_{L^2(\cM)^n}^2 \leq \frac{1}{2(u-t)} \norm{\mathbf{x}_1 - \mathbf{x}_0}_{L^2(\cM)^n}^2,
		\end{equation}
		because then we obtain the $1/(t-u)$-semiconvexity of $\psi$ in $\mathbf{x}$ by taking $\epsilon \to 0$.  To demonstrate \eqref{eq: horrid computation}, let $\mathbf{x}' = \mathbf{x}_1 - \mathbf{x}_0$ and $\mathbf{z}' = \mathbf{z}_1 - \mathbf{z}_0$.  For ease of notation, we drop the subscripts $L^2(\cM)^n$.  Note
		\begin{align*}
			\frac{1}{2u} \norm{\mathbf{z}'}^2 - \frac{1}{2t} \norm{\mathbf{x}' - \mathbf{z}'}^2 &= \frac{1}{2u} \norm{\mathbf{z}'}^2 - \frac{1}{2t} \norm{\mathbf{z}'}^2 + \frac{1}{t} \re \ip{\mathbf{x}',\mathbf{z}'} - \frac{1}{2t} \norm{\mathbf{x}'}^2 \\
			&= - \left( \frac{1}{2t} - \frac{1}{2u} \right) \norm{\mathbf{z}'}^2 + \frac{1}{t} \re \ip{\mathbf{x}',\mathbf{z}'} - \frac{u}{2t(u-t)} \norm{\mathbf{x}'}^2 + \frac{u}{2t(u-t)} \norm{\mathbf{x}'}^2 - \frac{1}{2t} \norm{\mathbf{x}'}^2 \\
			&= -\frac{1}{2} \norm*{\left(\frac{u-t}{ut}\right)^{1/2} \mathbf{z}' - \left( \frac{u}{t(u-t)} \right)^{1/2} \mathbf{x}'}^2 + \frac{1}{2(u-t)} \norm{\mathbf{x}'}^2 \\
			&\leq\frac{1}{2(u-t)} \norm{\mathbf{x}'}^2,
		\end{align*}
		which is what we wanted to prove.
	\end{proof}
	
	\begin{remark} \label{rem: semiconvex regularization}
		There is a symmetrical statement for sup-convolution rather than inf-convolution.  Let $\phi(\mathbf{x},\mathbf{y})$ be a $\mathrm{T}_{\tr}$-definable predicate, where $\mathbf{x}$ is an $m$-tuple and $\mathbf{y}$ is a countable tuple.  Fix $t > 0$ and $\mathbf{r} = (r_1,\dots,r_n) \in (0,\infty)^n$. Let
		\[
		\psi(\mathbf{x},\mathbf{y}) = \sup_{\mathbf{z} \in \mathcal{D}_{\mathbf{r}}} \left[ \phi(\mathbf{z},\mathbf{y}) - \frac{1}{2t} d(\mathbf{x},\mathbf{z})^2 \right].
		\]
		Then the same claims as in the previous proposition hold, switching ``semiconvex'' and ``semiconcave,'' reversing the direction of the inequality in (3), and replacing $\overline{\nabla}$ with $\underline{\nabla}$ in (4).
	\end{remark}
	
	Combining the two statements, we arrive at the following result, which is a more precise version of Theorem \ref{thm: approximation}.
	
	\begin{theorem} \label{thm: approximation full}
		Let $\phi(\mathbf{x},\mathbf{y})$ be a $\mathrm{T}_{\tr}$-definable predicate, where $\mathbf{x}$ is an $m$-tuple and $\mathbf{y}$ is a countable tuple.  Fix $t > 0$ and $\mathbf{r} = (r_1,\dots,r_n) \in (0,\infty)^n$ and $\mathbf{R} = (R_1,\dots,R_n) \in (0,\infty)^n$ with $r_j \leq R_j$.  Let
		\[
		\psi(\mathbf{x},\mathbf{y}) = \sup_{\mathbf{w} \in D_{\mathbf{r}}} \inf_{\mathbf{z} \in \mathcal{D}_{\mathbf{R}}} \left[ \phi(\mathbf{z},\mathbf{y}) + \frac{1}{4t} d(\mathbf{w},\mathbf{z})^2 - \frac{1}{2t} d(\mathbf{x},\mathbf{w}) \right].
		\]
		Then
		\begin{enumerate}[(1)]
			\item $\psi^{\cM}$ is $1/t$-semiconvex and $1/t$-semiconcave.
			\item We have for $\mathbf{x} \in D_{\mathbf{r}}^{\cM}$ and $\mathbf{y} \in D_{\mathbf{s}}^{\cM}$ that
			\[
			-\omega_{\phi,\mathbf{R},\mathbf{s}}\left( \sqrt{4t \omega_{\phi,\mathbf{R},\mathbf{s}}(2|\mathbf{R}|)} \right) \leq \psi^{\cM}(\mathbf{x},\mathbf{y}) - \phi^{\cM}(\mathbf{x},\mathbf{y}) \leq \omega_{\phi,\mathbf{r},\mathbf{s}}\left( \sqrt{2t \omega_{\phi,\mathbf{r},\mathbf{s}}(2|\mathbf{r}|)} \right).
			\]
			\item $\psi^{\cM}$ is differentiable, and for each $\mathbf{x} \in D_{\mathbf{r}'}^{\cM}$, we have $\nabla_{\mathbf{x}} \psi^{\cM}(\mathbf{x},\mathbf{y}) \in D_{(\mathbf{r}+\mathbf{r}')/t}$.
			\item $\nabla_{\mathbf{x}} \psi(\mathbf{x},\mathbf{y})$ is a $\mathrm{T}_{\tr}$-definable function.
		\end{enumerate}
	\end{theorem}
	
	\begin{proof}
		Let
		\[
		\theta(\mathbf{x},\mathbf{y}) = \inf_{\mathbf{z} \in \mathcal{D}_{\mathbf{R}}} \left[ \phi(\mathbf{z},\mathbf{y}) + \frac{1}{4t} d(\mathbf{x},\mathbf{z})^2 \right].
		\]
		Thus, $\theta$ is obtained from $\phi$ as in Proposition \ref{prop: semiconcave regularization} with parameter $2t$ rather than $t$.  Then $\psi$ is obtained from $\theta$ by the symmetrical process as in Remark \ref{rem: semiconvex regularization} with parameter $t$.
		
		(1) By the symmetrical statement to Proposition \ref{prop: semiconcave regularization} (1), $\psi$ is $1/t$-semiconvex.  By Proposition \ref{prop: semiconcave regularization} (1), $\theta$ is $1/2t$-semiconcave, and thus by the symmetrical statement to Proposition \ref{prop: semiconcave regularization}, $\psi$ is $1/(2t - t) = 1/t$-semiconcave.
		
		(2) For the left inequality, note that $\psi - \phi \geq \theta - \phi$ and apply Proposition \ref{prop: semiconcave regularization}(3) for $\theta$.  For the right inequality, let
		\[
		\eta(\mathbf{x},\mathbf{y}) = \sup_{\mathbf{z} \in D_{\mathbf{r}}} \left[ \phi(\mathbf{z},\mathbf{y}) - \frac{1}{2t} d(\mathbf{x},\mathbf{z}) \right].
		\]
		In other words, $\eta$ is obtained from $\phi$ as in Remark \ref{rem: semiconvex regularization} with parameter $t$.  Thus, $\eta$ is obtained from $\phi$ in the same way as $\psi$ is obtained from $\theta$.  Since $\theta \leq \phi$ on $D_{\mathbf{R}}$ and hence also on $D_{\mathbf{r}}$, we have $\psi \leq \eta$, and thus $\psi - \phi \leq \eta - \phi$.  Then we can apply the symmetrical statement to Proposition \ref{prop: semiconcave regularization} (3) to $\eta$ to obtain
		\[
		\psi^{\cM}(\mathbf{x}) - \phi^{\cM}(\mathbf{x}) \leq \eta^{\cM}(\mathbf{x}) - \phi^{\cM}(\mathbf{x}) \leq \omega_{\phi,\mathbf{r},\mathbf{s}}\left( \sqrt{2t \omega_{\phi,\mathbf{r},\mathbf{s}}(2|\mathbf{r}|)} \right) \text{ for } \mathbf{x} \in D_{\mathbf{r}}^{\cM}, \mathbf{y} \in D_{\mathbf{s}}^{\cM}.
		\]
		
		(3) By Proposition \ref{prop: semiconvex semiconcave gradient}, $\psi^{\cM}(\mathbf{x},\mathbf{y})$ is differentiable, so $\nabla_{\mathbf{x}} \psi^{\cM}$ exists as an element of $L^2(\cM)_{\sa}^n$.  By the symmetrical statement to Proposition \ref{prop: semiconcave regularization} (4), there exists some $\mathbf{p} \in D_{\mathbf{r}+\mathbf{r}'}^{\cM}$ with $(1/t) \mathbf{p} \in \underline{\nabla}_{\mathbf{x}} \psi^{\cM}(\mathbf{x},\mathbf{y})$.  Differentiability implies that the subgradient vector is unique, and so $\nabla_{\mathbf{x}} \psi^{\cM}(\mathbf{x},\mathbf{y}) = (1/t) \mathbf{p}$.  Hence, $\nabla_{\mathbf{x}} \psi^{\cM}(\mathbf{x},\mathbf{y}) \in D_{(\mathbf{r}+\mathbf{r}')/t}^{\cM}$.
		
		(4) This follows from Proposition \ref{prop: semiconvex semiconcave gradient} (3).
	\end{proof}
	
	\begin{proof}[{Proof of Theorem \ref{thm: approximation}}]
		Given $\phi$ and $\epsilon > 0$ and $\mathbf{r}$ and $\mathbf{s}$, choose $t$ sufficiently small that the error $\phi - \psi$ in (2) of the previous theorem is less than $\epsilon$.  Then $\psi$ has the desired properties.
	\end{proof}
	
	\subsection{Definable closure and definable functions} \label{subsec: definable function}
	
	We are now ready to conclude the proof of Theorem \ref{thm: definable realization} showing that every element in the definable closure of $A$ can be realized as a definable function of $A$ (see Definition \ref{def: definable function} and Definition \ref{def: DCL and ACL}).
	
	\begin{proof}[{Proof of Theorem \ref{thm: definable realization}}]
		Since $z \in \dcl^{\cM}(A)$, the function $x \mapsto d(x,z)$ can be expressed as $d(z,x) = \eta^{\cM_A}(x)$ where $\eta$ is a definable predicate $\eta$ in $\cL_A$ relative to $\Th(\cM_A)$.  Hence, the same is true for
		\[
		\theta^{\cM}(x) = \re \ip{x,z}_{L^2(\cM)} = \frac{1}{2} \left[ \norm{z}_{L^2(\cM)}^2 + \norm{x}_{L^2(\cM)}^2 - \norm{z - x}_{L^2(\cM)}^2 \right].
		\]
		Fix $r > 0$.  Then $\theta^{\cM}(x)$ can be uniformly approximated on the domain $D_r^{\cM}$ by a sequence of $\cL_{\tr,A}$-formulas $\theta_k$.  The $\cL_{\tr,A}$-formulas $\theta_k(z)$ can be expressed as $\phi_k(x,(a_j)_{j\in \N})$ for some $\cL_{\tr}$-formula $\phi_k(x,(y_j)_{j \in \N})$ and some tuple $a_j$ from $A$ (each $\theta_k$ only needs finitely many constants from $A$, so we can use a countable tuple $(a_j)_{j \in \N}$ that includes all of them).  For simplicity, we write $\mathbf{a} = (a_j)_{j \in \N}$ and $\mathbf{y} = (y_j)_{j \in \N}$.  We will find a $\mathrm{T}_{\tr}$-definable predicate $\phi(x,\mathbf{y})$ by a forced limit construction as in Proposition \ref{prop: pre MK duality}.  By passing to a subsequence assume that $|\phi_k^{\cM}(x,\mathbf{a}) - \phi_{k+1}^{\cM}(x,\mathbf{a})| \leq 1/2^k$ for $x \in D_r^{\cM}$.  Let $g: [0,\infty) \to [0,1]$ be continuous and decreasing with $g = 1$ on $[0,1]$ and $g = 0$ on $[2,\infty]$, and let $g_n(x) = g(2^n x)$.  Let
		\[
		\phi(x,\mathbf{y}) = \phi_1(x,\mathbf{y}) + \sum_{n=1}^\infty g_n(\phi_k^{\cM}(x,\mathbf{y}) - \phi_{k+1}^{\cM}(x,\mathbf{y}))(\phi_{k+1}^{\cM}(x,\mathbf{y}) - \phi_k^{\cM}(x,\mathbf{y}).
		\]
		One can check that $\phi$ is a $\mathrm{T}_{\tr}$-definable predicate satisfying
		\[
		\phi^{\cM}(x,\mathbf{a}) = \re \ip{z, x}_{L^2(\cM)} \text{ for } x \in D_r^{\cM}.
		\]
		
		Let $t > 0$ small enough that $t \norm{z} < r/2$, and as in Theorem \ref{thm: approximation full}, let
		\[
		\psi(x,\mathbf{y}) = \sup_{w \in D_r} \inf_{v \in \mathcal{D}_{2r}} \left[ \phi(v,\mathbf{y}) + \frac{1}{4t} d(w,v)^2 - \frac{1}{2t} d(x,w)^2 \right].
		\]
		When we evaluate $\psi$ in $\cM$ with $\mathbf{y} = \mathbf{a}$, we get
		\[
		\psi^{\cM}(x,\mathbf{a}) = \sup_{w \in D_r} \inf_{v \in D_{2r}} \left[ \re \ip{v,z}_{L^2(\cM)} + \frac{1}{4t} d(w,v)^2 - \frac{1}{2t} d(x,w)^2 \right].
		\]
		To determine where the supremum and infimum are achieved we complete the square.  Note
		\[
		\re \ip{v,z}_{L^2(\cM)} + \frac{1}{4t} \norm{w - v}_{L^2(\cM)}^2 = \re \ip{w,z}_{L^2(\cM)} - t \norm{z}^2 + \frac{1}{2} \norm{\sqrt{2t}z - (w - v) / \sqrt{2t} }_{L^2(\cM)}^2.
		\]
		The infimum is achieved when $v = w - 2tz$, which is a valid point in $D_{2r}^{\cM}$ because we assumed that $w \in D_r^{\cM}$ and $t \norm{z} < r/2$.  Hence, the infimum over $v$ equals
		\[
		\re \ip{w,z}_{L^2(\cM)} - t \norm{z}^2 - \frac{1}{2t} \norm{x - w}_{L^2(\cM)}^2.
		\]
		Completing the square again,
		\[
		\re \ip{w,z}_{L^2(\cM)} - t \norm{z}^2 - \frac{1}{2t} \norm{x-w}_{L^2(\cM)}^2 = \re \ip{x,z}_{L^2(\cM)} - \frac{t}{2} \norm{z}^2 - \frac{1}{2} \norm{\sqrt{t}z - (w - x)/\sqrt{t}}^2.
		\]
		If $x \in D_{r/2}^{\cM}$, then $w = x + tz$ is in $D_r^{\cM}$ since we assumed $t \norm{z} < r/2$, and thus the supremum is achieved when $w = x + tz$, and its value is $\re \ip{x,z}_{L^2(\cM)} - \frac{t}{2} \norm{z}_{L^2(\cM)}^2$.  Hence,
		\[
		\psi^{\cM}(x,\mathbf{a}) = \re \ip{x,z}_{L^2(\cM)} - \frac{t}{2} \norm{z}_{L^2(\cM)}^2 \text{ for } x \in D_{r/2}^{\cM}.
		\]
		In particular,
		\[
		\nabla_x \psi^{\cM}(0,\mathbf{a}) = z.
		\]
		By Theorem \ref{thm: approximation full}, $\nabla_x \psi^{\cM}(x,\mathbf{y})$ is a $\mathrm{T}_{\tr}$-definable function and thus also $f(\mathbf{y}) = \nabla_x \psi^{\cM}(0,\mathbf{y})$ is a $\mathrm{T}_{\tr}$-definable function which satisfies $f(\mathbf{a}) = z$.
	\end{proof}
	
	\bibliographystyle{plain}
	\bibliography{definable_closure_bibliography.bib}
	
\end{document}